\documentclass[10pt]{amsart}
\usepackage{amsfonts,amssymb,amsmath,layout,color}

\newcount\minute
\newcount\hour
\def\currenttime{%
	\minute\time
	\hour\minute
	\divide\hour60
	\the\hour:\multiply\hour60\advance\minute-\hour\the\minute}
\def\draftnote{{\it \today \quad  \currenttime \hfill  tex-file :   \jobname}}

\catcode`@=11
\@addtoreset{equation}{section}

\catcode`@=12
\newtheorem{Theorem}{Theorem}[section]

\newtheorem{Proposition}{Proposition}[section]
\newtheorem{Lemma}{Lemma}[section]

\newtheorem{Hyp.}{Hyp.}[section]

\begin{document}

\title[]{Bilinear control of a degenerate hyperbolic equation}

\author{P. Cannarsa} 
\address{Dipartimento di Matematica, Universit\`a di Roma "Tor Vergata",
Via della Ricerca Scientifica, 00133 Roma, Italy}
\email{cannarsa@mat.uniroma2.it}

\author{P. Martinez} 
\address{Institut de Math\'ematiques de Toulouse; UMR 5219, Universit\'e de Toulouse; CNRS \\ 
UPS IMT F-31062 Toulouse Cedex 9, France} \email{patrick.martinez@math.univ-toulouse.fr}

\author{C. Urbani} 
\address{Dipartimento di Matematica, Universit\`a di Roma "Tor Vergata",
Via della Ricerca Scientifica, 00133 Roma, Italy}
\email{urbani@mat.uniroma2.it}

\thanks{This research was partly supported by the Institut Mathematique de Toulouse and Istituto Nazionale di Alta Matematica. Part of this work was done during the confrence "VIII Partial differential equations, optimal design and numerics 2019", held in Benasque and supported by LIA COPDESC. We thank the Centro de Ciencias Pedro Pascual, in Benasque, Spain, for this opportunity. Moreover, the first author acknowledges support by the MIUR Excellence Department Project awarded to the Department of Mathematics, University of Rome Tor Vergata, CUP E83C18000100006.}

\begin{abstract}
We consider the linear degenerate wave equation, on the interval $(0, 1)$
$$ w_{tt} - (x^\alpha w_x)_x = p(t) \mu (x) w, $$
with bilinear control $p$
and Neumann boundary conditions. We study the
controllability of this nonlinear control system, locally around a
constant reference trajectory, the $\lq\lq$ground state".

Under some classical and generic assumption on $\mu$, we prove that there exists a threshold value for time, $T_0= \frac{4}{2-\alpha}$, such that the reachable set is
\begin{itemize}
\item a neighborhood of the ground state if $T>T_0$,
\item contained in a $C^1$-submanifold of infinite codimension if $T<T_0$
\item a $C^1$-submanifold of codimension $1$ if $\alpha \in [0,1)$, and a neighborhood of the ground state
 if $\alpha \in (1,2)$ if $T=T_0$, the case $\alpha =1$ remaining open.
\end{itemize}

This extends to the degenerate case the work of Beauchard \cite{B-bi} concerning the bilinear control of the classical wave equation ($\alpha =0$), and adapts to bilinear controls the work of Alabau-Boussouira, Cannarsa and Leugering \cite{ACL} on the degenerate wave equation where additive control are considered. 
Our proofs are based on a careful analysis of the spectral problem, and on  Ingham type results, which are extensions of 
the Kadec's $\frac{1}{4}$ theorem.

\end{abstract}
\maketitle

\section{Introduction}

\subsection{The context and the problem we study} \hfill

Degenerate partial differential equations appear in many domains, in particular physics, climate dynamics, biology, economics (see, e.g., \cite{memoire,Epstein,Ghil76}). Control of degenerate parabolic equations is, by now, a fairly well-developed subject (see, for instance, \cite{sicon2008, memoire, CMV-MCRF, CMV-strong}),
but very few results are available in the case of degenerate hyperbolic equations. 
To our best knowledge, a class of degenerate wave equations has been studied from the point of view of control theory
in \cite{ACL}, where boundary control is studied using HUM and multiplier methods, 
and in \cite{ZG3,ZG2,ZG1}, where locally distributed control are considered.

On the other hand, in many applications, one is naturally led to use bilinear controls, as such controls are more realistic than additive ones to govern the evolution of certain systems (see \cite{acu-exa,acu-sup,CanFloKha,cu, Floridia, flo2, Flo-Trom} for parabolic equations.). For example, \cite{Khapalov} mentions in particular
\begin{itemize}
\item the linearized nuclear chain reaction
$$ u_t = a^2 \Delta u + v(x,t)u, $$
where $u(x, t)$ is the neutron density at point $x$ at time $t$ and $v$ is the bilinear control that modelizes the effect of the "control rods";

\item the approximate controllability of the rod made of a \emph{smart material}
$$ u_{tt} + u_{xxxx} + v(t) u_{xx} =0, $$
with hinged ends, where $u$ is the displacement of the beam, and $v$ is related to the magnitude of the electric field
that heats the beam in order to control the vibrations.
\end{itemize}


This is why, in this paper, we address a bilinear control problem for the equation

\begin{equation}
\label{CMU-eq-ctr}
\begin{cases}
w_{tt} - (x^\alpha w_x)_x = p(t) \mu (x) w , \quad &x\in (0,1), t \in (0,T), 
\\
(x^\alpha w_x)(x=0)=0 , \quad &t \in (0,T) ,\\
w_x (x=1)=0 , \quad &t \in (0,T) , \\
w(x,0)=w_0 (x) ,  \quad &x\in (0,1), 
\\
w_t (x,0)=w_1 (x) , \quad &x\in (0,1) .
\end{cases}
\end{equation}
Here, 
\begin{itemize}
\item $\alpha \in [0,2)$ is the degereracy parameter ($\alpha =0$ for the classical wave equation and $\alpha \in (0,2)$ in the degenerate case), 
\item $p \in L^2 (0,T)$ is a multiplicative control, 
\item $\mu$ is an admissible potential (and a key feature will be to analyze to which class $\mu$ has to belong in order to prove a controllability result.)
\end{itemize}

Let us recall that the action of bilinear controls is weaker than the one of additive controls, in the sense that, with bilinear controls, one cannot expect the same kind of controllability results that can be proved with additive controls. This fact is described by the negative result obtained by Ball-Marsden-Slemrod in \cite{BMS}, where it is shown that the attainable set of any abstract linear system, subject to a bilinear control, has a dense complement.

For the Schr\"odinger equation, and then for the classical wave equation, attainability results with bilinear controls were obtained by Beauchard and Laurent in \cite{B-L} and Beauchard in \cite{B-bi}. In particular, it is proved in \cite{B-bi} that, 
\begin{itemize}
\item starting from the$\lq\lq$ground state" (which is the constant state associated to the first eigenvalue, $0$), the solution is more regular than expected, namely
$$ (w(T), w_t (T)) \in H^3  (0,1) \times H^2 (0,1) ,$$
\item generically with respect to $\mu$ (under a suitable condition relating $\mu$ and the eigenvalues
and eigenfunctions of the Laplacian operator with Neumann boundary conditions), the wave equation is locally controllable along the ground state with respect to the $H^3  (0,1) \times H^2 (0,1)$ topology, in time $T>2$ 
with controls in $L^2((0, T ),\mathbb R)$.
\end{itemize}

The goal of our paper is to extend these results to the degenerate case $\alpha \in (0,2)$. We obtain the following results:
\begin{itemize}
\item we show that  the solution starting from the ground state  is more regular than expected, that is, $$ (w(T), w_t (T)) \in H^3 _\alpha (0,1) \times H^2 _\alpha (0,1) ,$$ where the spaces $H^2 _\alpha (0,1)$ and $H^3 _\alpha (0,1)$ are the classical weighted Sobolev spaces associated with the degenerate elliptic operator that appears in \eqref{CMU-eq-ctr} (see Proposition \ref{prop-"thm3"-red});
\item generically with respect to $\mu$ in some suitable Banach space, we prove that any target, which is close to the ground state in the $H^3 _\alpha (0,1) \times H^2 _\alpha (0,1)$ topology, is reachable in time 
$$ T > \frac{4}{2-\alpha} $$
with controls in $L^2((0, T ),\mathbb R)$ (see Theorem \ref{thm-ctr});
\item when
$$ 0< T < \frac{4}{2-\alpha} ,$$
the set of reachable targets close to the ground state in the $H^3 _\alpha (0,1) \times H^2 _\alpha (0,1)$ topology
is contained in a $C^1$-manifold of infinite dimension and of infinite codimension (see Theorem \ref{thm-ctr<}),
\item and when 
$$ T= \frac{4}{2-\alpha} ,$$
we prove that the situation is different for $\alpha \in [0,1)$ and $\alpha \in [1,2)$ (see Theorem \ref{thm-ctr=}):
\begin{itemize}
\item if $\alpha \in [0,1)$, the set of reachable targets close to the ground state in the $H^3 _\alpha (0,1) \times H^2 _\alpha (0,1)$ topology
is a $C^1$-manifold of codimension $1$,

\item and if $\alpha \in (1,2)$ (except for a countable set of $\alpha$, where a more precise analysis would have to be performed), any target close to the ground state in the $H^3 _\alpha (0,1) \times H^2 _\alpha (0,1)$ topology is reachable, as in the case $T>\frac{4}{2-\alpha}$. Note that in particular the case $\alpha =1$ remains an open and interesting case.
\end{itemize}
\end{itemize}

Our approach follows the strategy proposed by Beauchard \cite{B-bi} in the nondegenerate case.
However, it is worth noting that several new difficulties appear in the degenerate case:
\begin{itemize}
\item the spectral analysis: as often happens concerning degenerate elliptic operator, the eigenvalues and eigenfunctions of the associated spectral problem are given in terms of Bessel functions, see Propositions \ref{prop-vp-w} when $\alpha \in [0,1)$ and Proposition \ref{prop-vp-s} when $\alpha \in [1,2)$. The spectral analysis here is entirely new;

\item the unboundness of the eigenfunctions: in the work by Beauchard \cite{B-bi}, a key point  is that the eigenfunctions (cosine functions) are bounded. We will prove that for  $\alpha >0$, the eigenfunctions are no longer bounded (see Lemmas \ref{lem-prop-phi-n-0-1-w} and \ref{lem-prop-phi-n-0-1-s}) and this fact brings new difficulties that need to be overcome;

\item the role of the potential $\mu$ since the solutions of degenerate wave equations are less regular with respect to those of the classical wave equation, the potential $\mu$ will play a crucial role and this explains the extra conditions this function has to satisfy;

\item the moment problem: we obtain
\begin{itemize}
\item positive local controllability results when
$$ T > \frac{4}{2-\alpha} ,$$
using Ingham's arguments,

\item $\lq\lq$negative" local controllability results when
$$ 0 < T \leq \frac{4}{2-\alpha} $$
combining general results on the families of exponentials $(e^{i\omega _n t})_{n\in \mathbb Z}$ in $L^2(0,T)$ (see \cite{Avdonin-Ivanov}) with the gap property satisfied by the eigenvalues of the problem.
\end{itemize}

As for the critical case $T=T_0$ is concerned, we prove some elementary but useful extensions of the classical Kadec's $\frac{1}{4}$ Theorem (\cite{Young} Theorem 1.14 p. 42), that allow us to treat this special case (see Lemmas \ref{lem-Riesz-T0} and \ref{lem-Riesz-T0-1-2}).

\end{itemize}
To conclude, let us observe that, as the reader may have noticed, we have assumed the degeneracy exponent $\alpha$ to be in the interval $[0,2)$: this restriction is partly due to our method, but, on the other hand, it is  known that problem \eqref{CMU-eq-ctr}, with  (additive) boundary control, fails to be controllable for $\alpha\ge 2$ (see \cite{ACL}).

\subsection{Plan of the paper}

\begin{itemize}
\item In section \ref{sec-wp}, we recall the functional setting and the well-posedness results for the degenerate wave equation.
\item In section \ref{sec-results}, we state our main results:
\begin{itemize}
\item the analysis of the eigenvalue problem associated to \eqref{CMU-eq-ctr}, see Proposition \ref{prop-vp-w} for $\alpha \in [0,1)$ and Proposition \ref{prop-vp-s} for $\alpha \in [1,2)$,
\item a hidden regularity result of the value map, which is the fundamental observation before studying the bilinear control problem, see Proposition \ref{prop-"thm3"-red},
\item a positive bilinear control result, see Theorem \ref{thm-ctr}, when $T>\frac{4}{2-\alpha}$,
\item a bilinear control result, see Theorem \ref{thm-ctr=}, when $T=\frac{4}{2-\alpha}$,
\item a $\lq\lq$negative" bilinear control result, see Theorem \ref{thm-ctr<}, when $T<\frac{4}{2-\alpha}$.
\end{itemize}
\item In section \ref{sec-wellposedness}, we prove the well posedness results.
\item In section \ref{sec-spectral}, we prove Propositions \ref{prop-vp-w} and \ref{prop-vp-s}.
\item In section \ref{sec-hidden}, we prove Proposition \ref{prop-"thm3"-red}.
\item In section \ref{sec-pr-thm}, we prove Theorem \ref{thm-ctr}.
\item In section \ref{sec-pr-thm=}, we prove Theorem \ref{thm-ctr=}.
\item In section \ref{sec-pr-thm<}, we prove Theorem \ref{thm-ctr<}.
\end{itemize}



\section{Functional setting and well-posedness} 
\label{sec-wp}

\subsection{Functional setting for $\alpha \in [0,1)$}\hfill

For  $0 \leq \alpha < 1$, we consider
\begin{multline}
\label{*H^1_a-w}
 H^1_{\alpha} (0,1):=\Bigl \{ u \in L^2 (0,1),\\
 u \text{ absolutely continuous in } [0,1],
 x^{\alpha /2} u_x \in L^2 (0,1) \Bigr \},
\end{multline}
and
\begin{equation}
\label{H2alpha-w}
H^2_{\alpha} (0,1):=   \Bigl \{ u \in H^1_{\alpha }(0,1), x^\alpha u_x \in H^1 (0,1)  \Bigr \} .
\end{equation}
$H^1_{\alpha} (0,1)$ is endowed with the natural scalar product
$$ \forall f,g \in  H^1_{\alpha} (0,1), \quad (f,g) = 
\int _{0} ^1 \bigl(  x ^\alpha f_x g_x + fg \bigr) \, dx .$$
The operator $A:D(A)\subset L^2(0,1) \to L^2(0,1)$ will be  defined by
\begin{equation}
\label{D(A)-w}
\begin{cases}
\forall u \in D(A), \quad  Au:= (x ^\alpha  u_x)_x , \\
D(A) :=  \{ u \in H^2_{\alpha} (0,1) , (x^\alpha u_x) (0)=0, u_x (1)=0 \} .
\end{cases} 
\end{equation}
Then, the following results hold:
\begin{Proposition}
\label{Prop-A-w}

Let $\alpha \in [0,1)$. Then,

a) $H^1_{\alpha} (0,1)$ is a Hilbert space,

b) $A: D(A) \subset L^2 (0,1) \to L^2 (0,1)$ is a self-adjoint negative operator with dense domain.

\end{Proposition}


\subsection{Functional setting for $\alpha \in [1,2)$} \hfill

For  $1 \leq \alpha < 2$, we consider the following spaces :
\begin{multline}
\label{*H^1_a-s}
 H^1_{\alpha} (0,1):= \Bigl \{ u \in L^2(0,1), \\
 u \text{ locally absolutely continuous in } (0,1],
 x^{\alpha/2} u_x \in  L^2(0,1)  \Bigr \} ,
\end{multline}
and
\begin{equation}
\label{H2alpha-s}
 H^2_{\alpha} (0,1):=\{ u \in H^1_{\alpha} (0,1) \ \mid 
 x^\alpha u_x \in H^1(0,1) \}  .
\end{equation}
The operator $A:D(A)\subset L^2(0,1)\to L^2(0,1)$ will be defined by
\begin{equation*}
\begin{cases}
\forall u \in D(A), \quad    Au:= (x^\alpha  u_x)_x,
 \\
D(A) :=   \{ u \in H^2_{\alpha}(0,1), (x^\alpha  u_x) (0) =0 , u_x (1)=0 \} .
\end{cases}
\end{equation*}
Then, the following results hold.

\begin{Proposition}
\label{Prop-A-s}
Let $\alpha \in [1,2)$. Then,

a) $H^1_{\alpha} (0,1)$ is a Hilbert space,

b) $A: D(A) \subset L^2 (0,1) \to L^2 (0,1)$ is a self-adjoint negative operator with dense domain.

\end{Proposition}
We deduce that, for any $\alpha\in[0,2)$, $A$ is the infinitesimal generator of an analytic semigroup of contractions $e^{tA}$
on $L^2(0,1)$. 


\subsection{Well posedness of the problem} \hfill

Consider the non-homogeneous problem

\begin{equation}
\label{CMU-eq-ctr-f}
\begin{cases}
w_{tt} - (x^\alpha w_x)_x = p(t) \mu (x) w  + f(x,t), \quad &x\in (0,1), t \in (0,T), 
\\
(x^\alpha w_x)(x=0)=0 , \quad &t \in (0,T) ,\\
w_x (x=1)=0 , \quad &t \in (0,T) , \\
w(x,0)=w_0 (x) ,  \quad &x\in (0,1), 
\\
w_t (x,0)=w_1 (x) , \quad &x\in (0,1) .
\end{cases}
\end{equation}
In order to recast \eqref{CMU-eq-ctr-f} into the a first order problem, we introduce
$$ \mathcal{W} := \left( \begin{array}{c} w \\ w_t \end{array} \right),
\quad
\mathcal{W}_0 := \left( \begin{array}{c} w_0 \\ w_1 \end{array} \right),
\quad
\mathcal{F} (x,t) := \left( \begin{array}{c} 0 \\ f(x,t) \end{array} \right),
$$
the state space
$$ \mathcal{X} := H^1 _\alpha (0,1) \times L^2 (0,1) ,$$
and the operators
\begin{equation}
\label{def-mathcalA}
\mathcal{A} := \left( \begin{array}{cc} 0 & \text{Id} \\ A & 0 \end{array} \right), \quad
D(\mathcal{A}) := D(A) \times H^1 _\alpha (0,1) ,
\end{equation}
and
\begin{equation}
\label{def-mathcalB}
\mathcal{B} := \left( \begin{array}{cc} 0 & 0 \\ \mu & 0 \end{array} \right), \quad
D(\mathcal{B}) := H^1 _\alpha (0,1) \times L^2 (0,1).
\end{equation}
So, problem \eqref{CMU-eq-ctr-f} can be rewritten as
\begin{equation}
\label{CMU-eq-ctr-f-ordre1}
\begin{cases}
\mathcal{W}'(t) = \mathcal{A}\mathcal{W}(t) + p(t) \mathcal{B} \mathcal{W} (t) + \mathcal{F} (t) ,
\\
\mathcal{W} (0) = \mathcal{W}_0 .
\end{cases}
\end{equation}
We also introduce the space
\begin{equation}
\label{space-muV1infty}
V^{(1, \infty)} _{\alpha} := 
\{ \mu \in H^1 _\alpha (0,1), x^{\alpha /2} \mu _x \in L^\infty (0,1) \} .
\end{equation}
We can now state the well-posedness result for problem \eqref{CMU-eq-ctr-f-ordre1}.

\begin{Proposition}
\label{prop-"prop2"}

Let $T>0$, $p\in L^2 (0,T)$ and $f \in L^2 ((0,T), H^1 _\alpha (0,1))$. 
Assume that
\begin{equation}
\label{eq-space-mu-wp}
\mu \in V^1 _\alpha :=
\begin{cases}
H^1 _\alpha (0,1) \quad &\text{ if } \alpha \in [0,1), \\
V^{(1, \infty)} _{\alpha}  \quad &\text{ if } \alpha \in [1,2).
\end{cases}
\end{equation}
Then, for all $\mathcal{W}_0 \in D(\mathcal{A})$, there exists a unique classical solution of \eqref{CMU-eq-ctr-f-ordre1}, i.e. a function
$$  \mathcal{W} \in C^0 ([0,T], D(\mathcal{A})) ,$$
such that the following equality holds in $D(\mathcal{A})$: for every $t\in [0,T]$,
\begin{equation}
\label{eq-duhamel}
\mathcal W (t) = e^{\mathcal{A}t} \mathcal{W} _0 + \int _0 ^t e^{\mathcal{A}(t-s)} (\mathcal{B} \mathcal{W} (s) + \mathcal{F}(s)) \, ds .
\end{equation}
Moreover, there exists $C=C(\alpha,T,p) >0$ such that $\mathcal{W}$ satisfies
\begin{equation}\label{estimW}
\Vert \mathcal{W} \Vert _{C^0([0,T],D(\mathcal{A}))}
\leq C \, \Bigl( \Vert \mathcal{W}_0\Vert _{D(\mathcal{A})} + \Vert F\Vert _{L^2(0,T;D(\mathcal{A}))}\Bigr).
\end{equation} 
\end{Proposition}


\section{Main results}
\label{sec-results}


\subsection{Preliminary result: Spectral problem}\label{se:main} \hfill

We investigate the eigenvalues and eigenfunctions of the operator
$$ - Au :=  - (x^\alpha u_x)_x , \quad x\in (0,1)$$
with Neumann boundary conditions:
$$ (x^\alpha u_x)(x=0) = 0 , \quad \text{ and } \quad u_x (x=1) .$$
Hence, we look for solutions $(\lambda, \Phi)$ of the problem
\begin{equation}
\label{vp}
\begin{cases}
- (x^\alpha \Phi ') ' = \lambda \Phi, \quad &x \in (0,1) , \\
x^\alpha \Phi ' (x) = 0 , \quad &x = 0 , \\
\Phi ' (1)=0 .
\end{cases}
\end{equation}
The difference with the spectral analysis of \cite{CMV-MCRF, CMV-strong} is in the boundary condition at the point $x=1$ which leads to new difficulties.


\subsubsection{Eigenvalues and eigenfunctions when $\alpha \in [0,1)$} \hfill

\begin{Proposition}
\label{prop-vp-w}
Given $\alpha \in [0,1)$, set
$$ \kappa _\alpha := \frac{2-\alpha}{2}, \quad \nu_\alpha := \frac{1-\alpha}{2-\alpha},$$
and consider the Bessel function $J_{-\nu_\alpha}$ of negative order $-\nu_\alpha$ and of first kind, and the positive zeros $(j_{-\nu_\alpha + 1, n}) _{n\geq 1}$ of the Bessel function $J_{-\nu_\alpha +1}$. 

Then, the set of solutions $(\lambda,\Phi)$ of problem \eqref{vp} is
$$ \mathcal S = \{(\lambda_{\alpha,n}, \rho \, \Phi _{\alpha,n}), n \in \mathbb N, \rho \in \mathbb R \},$$
where
\begin{itemize}
\item for $n=0$,
\begin{equation}
\label{eq-vp-w-vp0}
\lambda _{\alpha,0} = 0, \quad \Phi _{\alpha,0} (x)=1 ,
\end{equation}

\item 
for $n\geq 1$,
\begin{equation}
\label{eq-vp-w-vpm} 
\lambda _{\alpha,n} = \kappa _\alpha ^2 \, j_{-\nu_\alpha + 1, n} ^2 ,
\quad \Phi _{\alpha,n} (x) = K_{\alpha,n} \, x^{\frac{1-\alpha}{2}} \, J_{-\nu_\alpha} \Bigl( j_{-\nu_\alpha + 1, n}\,  x^{\frac{2-\alpha}{2}}\Bigr),
\end{equation}
where the positive constant $K_{\alpha,n}$ is chosen such that $\Vert \Phi _{\alpha,n} \Vert _{L^2 (0,1)} = 1$.
\end{itemize}
Moreover, the sequence $(\Phi _{\alpha,n})_{n\geq 0}$ forms an orthonormal basis of $L^2 (0,1)$. Additionally, the sequence $(\sqrt{\lambda _{\alpha,n+1}}  - \sqrt{\lambda _{\alpha,n}})_{n\geq 1}$ is decreasing and 
\begin{equation}
\label{gap-sqrt-weak}
\sqrt{\lambda _{\alpha,n+1}}  - \sqrt{\lambda _{\alpha,n}} \to \frac{2-\alpha}{2} \,  \pi \quad \text{ as } n\to \infty .
\end{equation}

\end{Proposition}


\subsubsection{Eigenvalues and eigenfunctions when $\alpha \in [1,2)$} \hfill

\begin{Proposition}
\label{prop-vp-s}

Given $\alpha \in [1,2)$, set
$$ \kappa _\alpha := \frac{2-\alpha}{2}, \quad \nu_\alpha := \frac{\alpha -1}{2-\alpha},$$
and consider the Bessel function $J_{\nu_\alpha}$ of positive order $\nu_\alpha$ and of first kind, and the positive zeros $(j_{\nu_\alpha + 1, n}) _{n\geq 1}$ of the Bessel function $J_{\nu_\alpha +1}$. 

Then, the set of solutions $(\lambda,\Phi)$ of problem \eqref{vp} is
$$ \mathcal S = \{(\lambda_{\alpha,n}, \rho \, \Phi _{\alpha,n}), n \in \mathbb N, \rho \in \mathbb R \},$$
where
\begin{itemize}
\item for $n=0$,
\begin{equation}
\label{eq-vp-s-vp0}
\lambda _{\alpha,0} = 0, \quad \Phi _{\alpha,0} (x)=1 ,
\end{equation}

\item for $n\geq 1$,
\begin{equation}
\label{eq-vp-s-vpm} 
\lambda _{\alpha,n} = \kappa _\alpha ^2 \, j_{\nu_\alpha + 1, n} ^2 ,
\quad \Phi _{\alpha,n} (x) = K_{\alpha,n} \, x^{\frac{1-\alpha}{2}} \, J_{\nu_\alpha} \Bigl( j_{\nu_\alpha + 1, n}\,  x^{\frac{2-\alpha}{2}}\Bigr) ,
\end{equation}
where the positive constant $K_{\alpha,n}$ is chosen such that $\Vert \Phi _{\alpha,n} \Vert _{L^2 (0,1)} = 1$.
\end{itemize}
Moreover, the sequence $(\Phi _{\alpha,n})_{n\geq 0}$ forms an orthonormal basis of $L^2 (0,1)$. Additionally, the sequence $(\sqrt{\lambda _{\alpha,n+1}}  - \sqrt{\lambda _{\alpha,n}})_{n\geq 1}$ is decreasing and 
\begin{equation}
\label{gap-sqrt-strong}
\sqrt{\lambda _{\alpha,n+1}}  - \sqrt{\lambda _{\alpha,n}} \to \frac{2-\alpha}{2} \,  \pi \quad \text{ as } n\to \infty  .
\end{equation}

\end{Proposition}


\subsection{Hidden regularity} \hfill

\subsubsection{Some notations} \hfill

Let us start by introducing some notation which will be used in the proofs of our results. To avoid possible problems generated by the eigenvalue $0$, we define
 \begin{equation}
\label{notation}
\lambda _{\alpha ,n} ^* := 
\begin{cases} 
1 & \text{ for } n=0, \\ 
\lambda _{\alpha ,n} & \text{ for } n \geq 1 .
\end{cases}
\end{equation}
It will be useful to introduce the following intermediate Sobolev spaces for any $s >0$:
\begin{equation}
\label{def-Hs}
H^s _{(0)} := D((-A) ^{s/2})
= \Bigl \{ \psi \in L^2 (0,1),\, \sum _{k=0} ^\infty (\lambda ^* _{\alpha,k} ) ^s \langle \psi, \Phi _{\alpha,k} \rangle _{L^2(0,1)} ^2 < \infty \Bigr \},
\end{equation}
equipped with the norm
$$ \Vert \psi \Vert _{H^s _{(0)}} := \Bigl( \sum _{k=0} ^\infty (\lambda ^* _{\alpha,k} ) ^s \langle \psi, \Phi _{\alpha,k} \rangle _{L^2(0,1)} ^2 \Bigr) ^{1/2} .$$
We also define the following spaces
\begin{equation}
\label{space-muV1}
V^{(2, \infty)} _{\alpha} := 
\{ \mu \in H^2 _\alpha (0,1), x^{\alpha /2} \mu _x \in L^\infty (0,1) \} , 
\end{equation}
and
\begin{equation}
\label{space-muV2}
V^{(2, \infty, \infty)} _{\alpha} := 
\{ \mu \in H^2 _\alpha (0,1), x^{\alpha /2} \mu _x \in L^\infty (0,1), (x^\alpha \mu_x)_x \in L^\infty (0,1) \} ,
\end{equation}
and the closed subspace of $H^2 _\alpha (0,1)$
\begin{equation}
\label{space-muV10}
V^{(2, 0)} _{\alpha} := 
\{ w \in H^2 _\alpha (0,1), (x^{\alpha } w _x) (0)=0 \} .
\end{equation}
Given $(w_0,w_1) \in H^1 _\alpha (0,1) \times L^2 (0,1)$ and $p\in L^2 (0,T)$, we will denote by
$w^{(w_0,w_1;p)}$ the solution of \eqref{CMU-eq-ctr} associated to the initial conditions $w_0,w_1$ and control $p$. In particular, when $(w_0, w_1)=(1,0)$ and $p=0$, we note that
the constant function equal to $1$ satisfies \eqref{CMU-eq-ctr}, hence
$$ w^{(1,0;0)} = 1, \quad w_t ^{(1,0;0)} = 0 .$$
In the following, we will be interested in the regularity of the solution of \eqref{CMU-eq-ctr} starting from the ground state $(1,0)$, that is, the solution $w^{(1,0;p)}$ (or, more simply, $w^{(p)}$) of
\begin{equation}
\label{CMU-eq-ctr-10}
\begin{cases}
w_{tt} - (x^\alpha w_x)_x = p(t) \mu (x) w , \quad &x\in (0,1), t \in (0,T), 
\\
(x^\alpha w_x)(x=0)=0 , \quad &t \in (0,T) ,\\
w_x (x=1)=0 , \quad &t \in (0,T) , \\
w(x,0)=1 ,  \quad &x\in (0,1), 
\\
w_t (x,0)=0 , \quad &x\in (0,1) .
\end{cases}
\end{equation}


\subsubsection{A hidden regularity result} \hfill

We consider the solution $w$ of \eqref{CMU-eq-ctr-10} and, if $p\in L^2(0,T)$ and $\mu\in V^1_\alpha$, thanks to Proposition \ref{prop-"prop2"} we know that
$$ (w(T),w_t (T)) \in D(A) \times H^1 _\alpha (0,1)) .$$
Before stating Theorem \ref{thm-ctr}, we will prove the following result, which extends the regularity result of \cite[Theorem 3]{B-bi} to the degenerate case:

\begin{Proposition}
\label{prop-"thm3"-red}
Let $T>0$ and
\begin{equation}
\label{eq-space-mu}
\mu \in V^2 _\alpha :=
\begin{cases}
V^{(2, \infty)} _{\alpha} \quad &\text{ if } \alpha \in [0,1), \\
V^{(2, \infty, \infty)} _{\alpha}  \quad &\text{ if } \alpha \in [1,2) .
\end{cases}
\end{equation}
Then, for all $p \in L^2 (0,T)$, the solution $w^{(p)}$ of \eqref{CMU-eq-ctr-10} has the following additional regularity
\begin{equation}
\label{reg-wT-red}
(w^{(p)} (T), w^{(p)} _t (T))  \in H^3 _{(0)} \times D(A) .
\end{equation}
Moreover, the map
\begin{equation}
\label{def-thetaT-red}
\Theta _T: L^2 (0,T) \to H^3 _{(0)} \times D(A), \quad \Theta _T (p) := (w^{(p)} (T), w^{(p)} _t (T)) 
\end{equation}
is of class $C^1$.

\end{Proposition}


\subsection{Main controllability results} \hfill

Because of the negative result contained in \cite{BMS}, one could not expect any controllability
property to hold in the spaces $H^2 _{(0)}(0,1) \times H^1 _\alpha (0,1)$. However, since the
multiplication operator $Bu := \mu u$ does not preserve the space $H^3 _{(0)}(0,1)$,
the chance
to achieve controllability results in $H^3 _{(0)}(0,1) \times H^2 _{(0)}(0,1)$
is still open. For this purpose, we will need additional assumptions on the admissible potential $\mu$. Furthermore, we
observe that controllability properties will depend on a threshold value for the controllability time because of the finite speed of propagation, as it always happens for hyperbolic equations.

\subsubsection{Threshold value of $T$ and the admissible potentials $\mu$} \hfill

We will show that the value
\begin{equation}
\label{def-T0}
T_0 := \frac{4}{2-\alpha} 
\end{equation}
is the threshold time for controllability. Let us define the following subclass of admissible potentials $\mu$
\begin{equation}
\label{hyp-mu}
V^{(adm)} := \{ \mu \in V^2 _\alpha, \text{ such that } \exists c>0, \forall n \geq 0, \quad \vert \langle \mu, \Phi _{\alpha,n} \rangle _{L^2 (0,1)} \vert \geq \frac{c}{\lambda ^* _n } \} .
\end{equation}
We observe that the space $V^{(adm)}$ is non empty. Indeed, in the proposition that follows
we exhibit an admissible potential $\mu \in V^{(adm)}$:

\begin{Proposition}
\label{prop-mu-ex-dense}
The function $\mu: \mu (x) = x^{2-\alpha}$ belongs to $V^{(adm)}$. Moreover the space $ V^{(adm)}$ is dense in  $V^2 _\alpha$.
\end{Proposition}
We refer to the recent work of Urbani \cite{Cristina-these} (Chap 5) for sufficient conditions for building polynomials functions that fulfill the last condition in \eqref{hyp-mu} are given.


\subsubsection{Controllability result for $T > T_0$} \hfill

\begin{Theorem}
\label{thm-ctr}

Given $\alpha \in [0,2)$, let $\mu \in V ^{(adm)} $ (defined in \eqref{hyp-mu}) and
\begin{equation}
\label{hyp-T0}
T > T_0 .
\end{equation}
Then, there exists a neighbourhood $\mathcal V (1,0)$ of $(1,0)$ in $H^3 _{(0)} (0,1) \times D(A)$ such that,
for all $(w_0 ^f, w_1 ^f) \in \mathcal V (1,0)$, there exists a unique $p^f \in L^2 (0,T)$ close to $0$ such that
$$ (w^{(p^f)} (T), w^{(p^f)} _t (T)) = (w_0 ^f, w_1 ^f) .$$
Moreover, the application 
$$ \Gamma _{\alpha,T} : \mathcal V (1,0) \to L^2 (0,T), \quad (w_0 ^f, w_1 ^f) \mapsto p^f $$
is of class $C^1$.
\end{Theorem}


\subsubsection{Controllability result when $T = T_0$} \hfill

\begin{Theorem}
\label{thm-ctr=}
Let $\mu \in V ^{(adm)} $ (defined in \eqref{hyp-mu}) and
\begin{equation}
 T= T_0 . 
\end{equation}
Then,

\begin{itemize}
\item for $\alpha \in [0,1)$, the reachable set is locally a $C^1$-submanifold of $H^3 _{(0)} (0,1) \times D(A)$ of codimension 1,
\item for $\alpha \in (1,2)$ and $\frac{1}{2-\alpha} \notin \mathbb N$, the reachable set is a whole neighborhood of $(0,1)$ in $H^3_{(0)}(0,1)\times D(A)$.
\end{itemize}
\end{Theorem}

\noindent What happens for $\alpha = 2 - \frac{1}{k}$, $k\in \mathbb N ^*$ (in particular for $\alpha =1$) is still an open problem. These values are the points where the nature of the set $\{e^{i\omega _{\alpha,n}t}, n\in \mathbb Z \}$ changes. Indeed,
\begin{itemize}
\item for $\alpha \in [0,1)$, $\{e^{i\omega _{\alpha,n}t}, n\in \mathbb Z \}$ is a Riesz basis of $L^2(0,T_0)$ , 
\item for $\alpha \in (1,\frac{3}{2})$, $\{e^{i\omega _{\alpha,n}t}, n\in \mathbb Z \}$ has a deficiency equal to $1$ in $L^2(0,T_0)$, 
\item for $\alpha \in (\frac{3}{2}, \frac{5}{3})$, $\{e^{i\omega _{\alpha,n}t}, n\in \mathbb Z \}$ has a deficiency equal to $2$ in $L^2(0,T_0)$,
\end{itemize} 
and so on. A detailed analysis is given in Lemma \ref{cor-Riesz} that derives from the Kadec's $\frac{1}{4}$ Theorem (see Lemma \ref{lem-Riesz-T0-1-2}).

We would like to draw the attention to the different nature of the reachable set for the weak ($\alpha \in [0,1)$) and the strong ($\alpha\in[1,2)$) degeneracy: while in the first case it is a submanifold of codimension 1, in the latter case it is a complete neighborhood of $(1,0)$ (except for the aforementioned particular values of $\alpha$).


\subsubsection{The controllability result when $T < \frac{4}{2-\alpha}$} \hfill

\begin{Theorem}
\label{thm-ctr<}
Let $\mu \in V ^{(adm)} $ (defined in \eqref{hyp-mu}) and 
\begin{equation}
T< T_0  . 
\end{equation}
Then the reachable set is locally contained in a $C^1$-submanifold of $H^3 _{(0)} (0,1) \times D(A)$ of infinite dimension and of infinite codimension. 
\end{Theorem}


\section{Functional setting: proof of Propositions \ref{Prop-A-w}, \ref{Prop-A-s} and \ref{prop-"prop2"}} 
\label{sec-wellposedness}

\subsection{Proof of Propositions \ref{Prop-A-w}} \hfill

\subsubsection{Integration by parts} \hfill

Let us prove the following integration by parts formula.

\begin{Lemma}
\label{lem-IPP-w}
Let $\alpha\in[0,1)$, then
\begin{equation}
\label{IPP-w}
\forall\, f,g \in H^2 _\alpha (-1,1), \quad
\int _{0} ^1 (x ^\alpha f'(x) )' \, g(x) \, dx 
= - \int _{0} ^1 x ^\alpha \, f'(x) \, g '(x) \, dx .
\end{equation}
\end{Lemma}

\begin{proof}[Proof of Lemma \ref{lem-IPP-w}.]
If $f \in H^2 _\alpha (0,1)$, then 
$$ F(x):= x ^\alpha f'(x) \in H^1 (0,1) .$$
Let $g \in H^2 _\alpha (0,1)$, and $\varepsilon \in (0,1)$.
Decompose
$$ \int _{0} ^1 F'(x) g(x) \, dx  = 
\int _{0} ^\varepsilon F'(x) g(x) \, dx
+ \int _{\varepsilon} ^1  F'(x) g(x) \, dx .$$
Then, since $g \in H^2 _\alpha (0,1) \subset H^1 (\varepsilon,1)$, the classical integration by parts formula gives
\begin{multline*}
\int _{\varepsilon} ^1 F'(x) g(x) \, dx
= [F(x) g(x) ] _{\varepsilon} ^1 - \int _{\varepsilon} ^1 F(x) g'(x) \, dx
\\
=  [F(x) g(x) ] _{\varepsilon} ^1 - \int _{\varepsilon} ^1 (x^{\alpha/2} f'(x)) \,  (x^{\alpha/2} g'(x)) \, dx .
\end{multline*}
Now, since  $x^{\alpha/2} f'$ and $x^{\alpha/2} g'$ belong to $L^2 (0,1)$, we have
$$ 
\int _{\varepsilon} ^1 F(x) g'(x) \, dx \to \int _{0} ^1 F(x) g'(x) \, dx
\quad \text{ as } \varepsilon \to 0 ,$$
and since $F'$ and $g$ belong to $L^2 (0,1)$, we get
$$  \int _{0} ^\varepsilon F'(x) g(x) \, dx \to 0 
\quad \text{ as } \varepsilon \to 0 .$$
It remains to study the boundary terms:
first, because of Neumann boundary condition at $1$, we have
$$ [F(x) g(x) ] _{\varepsilon} ^1 = - F(\varepsilon) g(\varepsilon) .$$
We note that $F(\varepsilon) \to 0$ as $\varepsilon \to 0$, and $g$ is absolutely continuous on $[0,1]$, hence
$$ [F(x) g(x) ] _{\varepsilon} ^1 =  - F(\varepsilon) g(\varepsilon)
\to 0 \text{ as } \varepsilon \to  0 . $$
\end{proof}


\subsubsection{Proof of Proposition \ref{Prop-A-w}} \hfill

Part a) of Proposition \ref{Prop-A-w} is well known, see e.g. \cite{CMV-MCRF}. Let us prove Part b).

First, we note that $D(A)$ is dense in $X$, since it contains
all the functions of class $C^\infty$, compactly supported in $(0,1)$.

We derive from Lemma \ref{lem-IPP-w} that 
$$ \forall\, f \in D(A), \quad 
\langle Af , f \rangle = \int _{0} ^1 ( \vert x \vert ^\alpha f'(x))' f(x) \, dx 
= - \int _{0} ^1  x ^\alpha f'(x) ^2 \, dx \leq 0 ,$$
therefore $A$ is dissipative.

In order to show that $A$ is symmetric, we apply Lemma \ref{lem-IPP-w} twice to obtain that
\begin{multline*}
\forall\, f,g \in D(A), \quad 
\langle Af , g \rangle 
= \int _{0} ^1 ( x ^\alpha f'(x))' g(x) \, dx 
= - \int _{0} ^1  x ^\alpha f'(x) g'(x) \, dx 
\\
= - \int _{0} ^1  (x ^\alpha g'(x)) f'(x) \, dx 
= \int _{0} ^1 ( x ^\alpha g'(x))' f(x) \, dx 
= \langle f , Ag \rangle .
\end{multline*}

Finally, we check that $I-A$ is surjective. Let $f\in L^2 (0,1)$. Then, by Riesz theorem,
there exists one and only one $u\in H^1 _\alpha (0,1)$ such that 
$$\forall\, v \in H^1 _\alpha (0,1), \quad  
\int _{0} ^1 \bigl( uv+ x ^\alpha u' v' \bigr) = \int _{0} ^1 fv .$$
In particular, the above relation holds true for all $v$ of class $C^\infty$, compactly supported in $(0,1)$. Thus, $x\mapsto x ^\alpha u'$
has a weak derivative given by
$$ \Bigl( x ^\alpha u' \Bigr) '= - (f-u) .$$
Since $f-u \in L^2 (0,1)$, we obtain that $( x ^\alpha u')'
\in L^2 (0,1)$. Hence, $u\in H^2 _\alpha (0,1)$. Now, choosing first $v$ of class $C^\infty$ compactly supported in $[\frac{1}{2}, 1]$, but not equal to $0$ at the point $x=1$, we derive that
\begin{multline*}
 \int _{0} ^1 fv 
 = \int _{0} ^1 \bigl( uv+ x ^\alpha u' v' \bigr)
\\
= \int _{0} ^1  uv + [x ^\alpha u' v] _0 ^1 - \int _{0} ^1 (x ^\alpha u')' v
= [x ^\alpha u' v] _0 ^1 + \int _{0} ^1  ( u-(x ^\alpha u')' )v ,
\end{multline*}
therefore $u'(1) v(1)=0$ that implies $u'(1)=0$. In the same way, by choosing $v$ of class $C^\infty$ compactly supported in $[0,\frac{1}{2}]$, but not equal to $0$ at the point $x=0$, we obtain that $(x^\alpha u' (0)=0$. Thus $u\in D(A)$, and 
$(I-A)u=f$. So, the operator $I-A$ is surjective. This concludes the proof of Proposition \ref{Prop-A-w}, part b). \qed 


\subsection{Proof of Propositions \ref{Prop-A-s}} \hfill

\subsubsection{Integration by parts} \hfill

Let us prove the following integration by parts formula:

\begin{Lemma}
\label{lem-IPP-s}
Let $\alpha\in[1,2)$, then
\begin{equation}
\label{IPP-s}
\forall\, f,g \in H^2 _\alpha (0,1), \quad
\int _{0} ^1 (x ^\alpha f'(x) )' \, g(x) \, dx 
= - \int _{0} ^1 x ^\alpha \, f'(x) \, g '(x) \, dx .
\end{equation}
\end{Lemma}

\begin{proof}[Proof of Lemma \ref{lem-IPP-s}.]
If $f \in H^2 _\alpha (0,1)$, then 
$$ F(x):=  x  ^\alpha f'(x) \in H^1 (0,1) .$$
Let $g \in H^2 _\alpha (0,1)$, and $\varepsilon \in (0,1)$.
Decompose
$$ \int _{0} ^1 F'(x) g(x) \, dx  = 
\int _{0} ^\varepsilon F'(x) g(x) \, dx
+ \int _{\varepsilon} ^1  F'(x) g(x) \, dx .$$
Since $g \in H^2 _\alpha (0,1) \subset H^1 (0,1)$, the classical integration by parts formula gives
$$ 
\int _{0} ^1 F'(x) g(x) \, dx  = 
\int _{0} ^\varepsilon F'(x) g(x) \, dx
+ [F(x) g(x) ] _{\varepsilon} ^1 - \int _{\varepsilon} ^1 F(x) g'(x) \, dx .$$
To prove equation \eqref{IPP-s}, we have to let $\varepsilon \to 0$ in the above identity. First, note that
$$ \int _{\varepsilon} ^1 F(x) g'(x) \, dx
= \int _{\varepsilon} ^1 ( x^\alpha f'(x) ) \,  g'(x) \, dx
= \int _{\varepsilon} ^1 (x  ^{\alpha /2} f'(x) ) \, (x  ^{\alpha /2}g'(x) ) \, dx ,$$
and since $x\mapsto x^{\alpha /2} f'(x)$ and $x\mapsto x^{\alpha /2} g'(x)$ belong to $L^2 (0,1)$, we have that
$$ \int _{\varepsilon} ^1 (x  ^{\alpha /2} f'(x) ) \, (x  ^{\alpha /2}g'(x) ) \, dx
\to \int _{0} ^1 (x  ^{\alpha /2} f'(x) ) \, (x  ^{\alpha /2}g'(x) ) \, dx \quad \text{ as } \varepsilon \to 0. $$
Hence,
$$ \int _{\varepsilon} ^1 F(x) g'(x) \, dx \to \int _{0} ^1 F(x) g'(x) \, dx \quad \text{ as } \varepsilon \to 0.$$
Moreover, since $F'$ and $g$ belong to $L^2 (0,1)$, we get that
$$  \int _{0} ^\varepsilon F'(x) g(x) \, dx \to 0 
\quad \text{ as } \varepsilon \to 0 .$$
It remains to study the boundary terms:
first, because of Neumann boundary conditions at $1$, we have $F(1)=0$ and $g$ has a finite limit as $x\to 1$. Therefore,
$$ [F(x) g(x) ] _{\varepsilon} ^1 = - F(\varepsilon) g(\varepsilon) .$$
Now, we note that 
\begin{multline*}
 \forall\, x \in (0,1), \quad (F(x)g(x))' = F'(x) g(x) + F(x) g'(x) 
\\
= F'(x) g(x) + (x^{\alpha /2} f'(x)) (x^{\alpha /2} g'(x)),
\end{multline*}
and, since $F'$, $g$, $x^{\alpha /2} f'$, $x^{\alpha /2} g'$ belong to $L^2(0,1)$, 
we obtain that $(Fg)' \in L^1(0,1)$. Thus, $Fg$ is absolutely continuous on $(0,1]$ and it has a limit as $x\to 0$. This means that there exists $L$ such that
$$ F(x) g(x) \to L \text{ as } x \to 0 ^+.$$
We claim that $L=0$. Indeed:
\begin{itemize}
\item function $x\mapsto x^\alpha f'(x)$ belongs to $H^1(0,1)$, hence it has a limit as $x\to 0^+$:
$$ x^\alpha f'(x) \to \ell \quad \text{ as } x \to 0 ^+ ;$$
\item if $\ell \neq 0$,  
$$ x^{\alpha /2} f'(x) \sim \frac{\ell}{x^{\alpha /2}} \quad \text{ as } x \to 0 ^+ .$$
However, since $\alpha \geq 1$, we have that $\frac{\ell}{x^{\alpha /2}} \notin L^2 (0,1)$, so $\ell =0$;
\item moreover,
$$ \forall\, x \in (0,1), \quad x^\alpha f'(x) = \int _0 ^x (s^\alpha f'(s))' \, ds $$
and using the Cauchy-Schwarz inequality, one has
$$ \forall\, x\in (0,1), \quad \vert x ^\alpha f'(x) \vert \leq C \sqrt{ x } ;$$
\item finally, 
$$ \forall\, x\in (0,1), \quad \vert x ^\alpha f'(x) g(x) \vert \leq C \sqrt{ x } \vert g(x) \vert,$$
thus,
$$\forall\, x\in (0,1), \quad \vert F(x) g(x) \vert \leq C \sqrt{ x } \vert g(x) \vert.$$
If $L\neq 0$, then for $x$ sufficiently close to $0$ we have
$$ \vert g(x) \vert \geq \frac{CL}{2 \sqrt{x} },$$
which is in contradiction with $g\in L^2(0,1)$. Therefore, $L=0$.
\end{itemize}
This implies that
$$ [F(x) g(x) ] _{\varepsilon} ^1
=  - F(\varepsilon) g(\varepsilon) 
\to 0 \quad \text{ as } \varepsilon \to 0^+ .$$
This concludes the proof of Lemma \ref{lem-IPP-s}. 
\end{proof}


\subsubsection{Proof of Proposition \ref{Prop-A-s} } \hfill

Part a) of Proposition \ref{Prop-A-s} is well known, see e.g. \cite{CMV-strong}.
Therefore, we prove Part b).
The strategy of the proof is similar to the one of Proposition \ref{Prop-A-w}, part b),
and relies on the integration by parts formula given in Lemma \ref{lem-IPP-w}. We only need to check the boundary conditions.
As already noted, since $u\in H^2 _\alpha (0,1)$, this implies that $x^\alpha u'(x) \to 0$ as $x\to 0$. Hence the boundary condition is satisfied at $x=0$. Taking now $v$ of class $C^\infty$, not equal to $0$ at the point $x=1$, we derive that
\begin{multline*}
\int _{0} ^1 fv 
= \int _{0} ^1 \bigl( uv+ x ^\alpha u' v' \bigr)
\\
= \int _{0} ^1  uv + [x ^\alpha u' v] _0 ^1 - \int _{0} ^1 (x ^\alpha u')' v
= [x ^\alpha u' v] _0 ^1 + \int _{0} ^1  ( u-(x ^\alpha u')' )v ,
\end{multline*}
thus  $u'(1) v(1)=0$, and therefore $u'(1)=0$. We obtain that $u\in D(A)$, and 
$(I-A)u=f$. So, the operator $I-A$ is surjective. This concludes the proof of Proposition \ref{Prop-A-s}, part b). \qed


\subsection{Proof of Proposition \ref{prop-"prop2"}} \hfill

First, let us prove the following regularity result.

\begin{Lemma}
\label{lem-wellposed}
Let $\mu\in V^1_\alpha$. Then, the operator $\mathcal B$ defined in \eqref{def-mathcalB} satisfies
$$\mathcal B \in \mathcal L_c (D(\mathcal A),D(\mathcal A)) .$$

\end{Lemma}

\begin{proof}[Proof of Lemma \ref{lem-wellposed}] We have to prove that
$$ w_0 \in D(A) \implies \mu w_0 \in H^1 _\alpha (0,1)$$
and that there exists $C >0$ such that
\begin{equation}
\label{eq-well-posed}
 \Vert \mu w_0 \Vert _{H^1 _\alpha (0,1)} \leq C \Vert  w_0 \Vert _{D(A)} .
 \end{equation}
We distinguish the cases $\alpha \in [0,1)$ and $\alpha \in [1,2)$.
\begin{itemize}
\item $\alpha \in [0,1)$:
we can decompose $w'_0$ as follows $w_0 ' = (x^{\alpha /2} w_0 ') (x^{-\alpha /2})$. Since $w_0 \in H^1 _\alpha (0,1)$, we deduce that $w_0 ' \in L^1 (0,1)$ and thus $w_0 \in L^\infty (0,1)$. The same holds for $\mu$ because $V^1_\alpha=H^1_\alpha(0,1)$ for $\alpha\in[0,1)$. Hence, $(\mu w_0)' = \mu ' w_0 + \mu w_0 ' \in L^1 (0,1)$ and therefore $\mu w_0$ is absolutely continuous on $[0,1]$. Furthermore, we have that $x^{\alpha /2} (\mu w_0)' = (x^{\alpha /2} \mu ') w_0 + \mu (x^{\alpha /2} w_0 ' ) \in L^2 (0,1)$ and so we infer that $\mu w_0\in H^1_\alpha(0,1)$. Finally, there exists $C >0$ such that
$$ \forall\, w \in H^1 _\alpha (0,1), \quad \Vert w \Vert _{L^\infty (0,1)} \leq C \Vert  w \Vert _{H^1 _\alpha (0,1)} ,$$
and this implies that \eqref{eq-well-posed} holds.

\item $\alpha \in [1,2)$:
first, we observe that $\mu \in V^{1, \infty} _\alpha (0,1)$ implies that $\vert \mu _x \vert \leq \frac{C}{x^{\alpha /2}}$. Therefore, we get that $\mu_ x \in L^1 (0,1)$, and so $\mu \in L^\infty (0,1)$ and $\mu w_0 \in L^2 (0,1)$. Moreover, $x^{\alpha /2} (\mu w_0)' = (x^{\alpha /2} \mu _x) w_0 + (x^{\alpha /2} w_0 ') \mu$ and since $x^{\alpha /2} \mu _x \in L^\infty (0,1)$ and $w_0 \in L^2 (0,1)$, we have that $(x^{\alpha /2} \mu _x) w_0 \in L^2 (0,1)$. Furthermore, since $x^{\alpha /2} w_0 ' \in L^2 (0,1)$ and $\mu \in L^\infty (0,1)$, we deduce that $(x^{\alpha /2} w_0 ') \mu \in L^2 (0,1)$ and hence $x^{\alpha /2} (\mu w_0)' \in L^2 (0,1)$. By reasoning as in the case $\alpha\in[0,1)$, we deduce that \eqref{eq-well-posed} is verified. 
\end{itemize}
\end{proof}

\begin{proof}[Proof of Proposition \ref{prop-"prop2"}]
We prove the existence and uniqueness of the solution of problem \eqref{CMU-eq-ctr-f-ordre1} by a fixed point argument. We consider the map 
$$ \mathcal K: C^0 ([0,T], D(\mathcal A)) \to C^0 ([0,T], D(\mathcal A)) $$
defined by
\begin{equation}
\forall\, t\in [0,T], \quad \mathcal K (\mathcal{W}) (t):=e^{\mathcal{A}t}\mathcal{W}_0+\int_0^te^{\mathcal{A}(t-s)}(p(s)\mathcal{B}\mathcal{W}(s)+\mathcal{F}(s)) \, ds.
\end{equation}
We first prove that $\mathcal K$ is well-defined, which means that it maps $C^0([0,T], D(\mathcal{A}))$ into itself. We observe that, for any $\mathcal{W}\in C^0([0,T], D(\mathcal{A}))$, $\mathcal{B}\mathcal{W}\in C^0([0,T], D(\mathcal{A}))$ and thus $p\mathcal{B}\mathcal{W}\in L^2((0,T), D(\mathcal{A}))$. Hence, it is possible to apply the classical result of existence of strict solutions (see, for instance,  \cite[Proposition 3.3]{BdPDM}) and deduce that $\mathcal{K} (\mathcal{W})\in C^0([0,T], D(\mathcal{A}))$.

Moreover, for any $\mathcal{W}_1,\mathcal{W}_2\in  C^0([0,T], D(\mathcal{A}))$, it holds that
\begin{equation*}
\begin{split}
\Vert \mathcal{K} (\mathcal{W}_1)(t) - \mathcal{K} (\mathcal{W}_2)(t)\Vert _{D(\mathcal{A})}
&=\Bigl \Vert \int_0 ^t e^{\mathcal{A}(t-s)} p(s) \mathcal{B} \left(\mathcal{W}_1(s)-\mathcal{W}_2(s)\right) \, ds \Bigr \Vert _{D(\mathcal{A})}
\\
&\leq \int_0 ^t \vert p(s)\vert \Vert  e^{\mathcal{A}(t-s)}\mathcal{B}\left(\mathcal{W}_1(s)-\mathcal{W}_2(s)\right)\Vert _{D(\mathcal{A})} \, ds
\\
&\leq C_1 \int_0 ^t \vert  p(s)\vert \Vert \mathcal{B}\left(\mathcal{W}_1(s)-\mathcal{W}_2(s)\right) \Vert \, ds
\\
&\leq C_1 C_{\mathcal{B}} \Vert p\Vert _{L^1(0,T)}\Vert \mathcal{W}_1-\mathcal{W}_2\Vert _{C^0([0,T],D(\mathcal{A}))}.
\end{split}
\end{equation*}
Suppose $C_1 C_{\mathcal{B}}\Vert p\Vert_{L^1(0,T)}<1$. Then $\mathcal{K}$ is a contraction and therefore has a unique fixed point. 
Furthermore, we have that
\begin{equation*}
\begin{split}
\Vert \mathcal{W}&\Vert _{C^0([0,T],D(\mathcal{A}))}
\leq \sup_{t\in [0,T]}\Vert e^{\mathcal{A}t} \mathcal{W}_0 
+ \int_0 ^t e^{\mathcal{A}(t-s)}(p(s)\mathcal{B}\mathcal{W}(s) + \mathcal{F}(s)) \, ds \Vert _{D(\mathcal{A})}
\\
&\leq C_1\Bigl( \Vert \mathcal{W}_0 \Vert _{D(\mathcal{A})} + \int_0 ^T \vert p(s)\vert \Vert \mathcal{B}\mathcal{W}(s)\Vert _{D(\mathcal{A})} + \Vert F(s)\Vert _{D(\mathcal{A})} \, ds \Bigr)
\\
&\leq C_1 \Bigl( \Vert \mathcal{W}_0\Vert _{D(\mathcal{A})} + C_{\mathcal{B}} \Vert \mathcal{W}\Vert _{C^0([0,T],D(\mathcal{A}))}\Vert p\Vert _{L^1(0,T)} + \sqrt{T} \Vert F\Vert _{L^2(0,T;D(\mathcal{A}))}\Bigr).
\end{split}
\end{equation*}
Therefore,
\begin{equation}
\Vert \mathcal{W}\Vert _{C^0([0,T],D(\mathcal{A}))}
\leq \frac{C_1}{1-C_1 C_{\mathcal{B}} \Vert p\Vert _{L^1(0,T)}}(\Vert \mathcal{W}_0\Vert_{D(\mathcal{A})}+\sqrt{T}\Vert F\Vert_{L^2(0,T;D(\mathcal{A}))}).
\end{equation}
We have thus obtained the conclusion under the extra hypothesis that $p$ satisfies $C_1 C_{\mathcal{B}}\Vert p\Vert_{L^1(0,T)}<1$. In the general case, it is sufficient to represent $[0,T]$ as the union of a finite family of sufficiently small subintervals where we can repeat the above argument in each one.

Equivalently, \eqref{estimW} can be proved by Gronwall's Lemma, obtaining:
\begin{equation}
\Vert \mathcal{W}\Vert_{C^0([0,T],D(\mathcal{A}))}
\leq C_1 \Bigl( \Vert \mathcal{W}_0\Vert_{D(\mathcal{A})}+\sqrt{T}\Vert F\Vert _{L^2(0,T;D(\mathcal{A}))}\Bigr) \, e^{C_1\Vert p\Vert_{L^1(0,T)}} .
\end{equation}
\end{proof}


\section{Spectral problem: proof of Propositions \ref{prop-vp-w} and \ref{prop-vp-s} }
\label{sec-spectral}

\subsection{A classical change of variables} \hfill

First, we note that if $(\lambda, \Phi)$ solves \eqref{vp}, then $\lambda \geq 0$: indeed, multiplying by $\Phi$, we obtain
$$ \lambda \int _0 ^1 \Phi ^2 = \int _0 ^1 - (x^\alpha \Phi ') ' \Phi
= [ - (x^\alpha \Phi ')  \Phi ] _0 ^1 + \int _0 ^1 x^\alpha (\Phi ') ^2 
= \int _0 ^1 x^\alpha (\Phi ') ^2 .$$
Moreover, if $\lambda =0$, then $x\mapsto x^\alpha \Phi '$ is constant and by imposing the boundary conditions we find that it is actually equal to $0$. Thus, the constant functions are the ones and only ones associated to the eigenvalue $\lambda =0$.

We now investigate the positive eigenvalues: if $\lambda >0$, we introduce the function $\psi$ defined by the relation
$$ \Phi (x) = x^{\frac{1-\alpha}{2}} \psi \Bigl( \frac{2}{2-\alpha} \sqrt{\lambda} x^{\frac{2-\alpha}{2}}\Bigr)  ,$$
and the associated new space variable
$$ y = \frac{2}{2-\alpha} \sqrt{\lambda} x^{\frac{2-\alpha}{2}},$$
(see, e.g., \cite{CMV-strong,Gueye}). After some classical computations, we obtain that $\psi$ satisfies the following problem:

\begin{equation}
\label{vp-psi}
\begin{cases}
y^2 \psi '' (y) + y \psi '(y) + \Bigl( y^2 - (\frac{1-\alpha}{2-\alpha} )^2 \Bigr) \psi (y) = 0 , \quad y\in (0, \frac{2}{2-\alpha} \sqrt{\lambda}) , \\
 y^{\frac{1}{2-\alpha}} \psi ' (y) + \frac{1-\alpha}{2-\alpha}  y^{\frac{\alpha -1}{2-\alpha}} \psi  (y) \to 0 \text{ as } y \to 0 , \\
\sqrt{\lambda} \psi ' (\frac{2}{2-\alpha} \sqrt{\lambda}) + \frac{1-\alpha}{2} \psi ' (\frac{2}{2-\alpha} \sqrt{\lambda})  =0 .
\end{cases}
\end{equation}
The first equation in \eqref{vp-psi} is the Bessel equation of order $\nu=\Bigl\vert \frac{1-\alpha}{2-\alpha}\Bigr \vert$ (see \cite{Lebedev} or \cite{Watson}). 

In the following, we solve \eqref{vp} using \eqref{vp-psi} and distinguishing several cases
\begin{itemize}
\item $\alpha \in [0,1)$,
\item $\alpha \in (1,2)$ and $\nu_\alpha = \frac{\alpha -1}{2-\alpha} \notin \mathbb N$,
\item $\alpha \in (1,2)$ and $\nu_\alpha = \frac{\alpha -1}{2-\alpha} \in \mathbb N^*$,
\item $\alpha =1$.
\end{itemize}
Our study will be based on well-known properties of Bessel functions, and it is similar to the one in \cite{CMV-strong}, the only difference, which makes the analysis interesting, lying in the boundary condition at point $x=1$. In every case, the strategy is
\begin{itemize}
\item to exhibit a basis of the vector space of dimension $2$ of the solution of the first equation \eqref{vp-psi}, and this is where Bessel funtions appear, and where the distinction of the different cases is necessary,
\item to take into account the functional setting of the problem, in order to eliminate some possible solutions,
\item finally, to impose the boundary conditions at points $x=0$ and $x=1$, in order to determine the eigenvalues and the associated eigenfunctions.
\end{itemize}


\subsection{The case $\alpha \in [0,1)$} \hfill

\subsubsection{The study of the ODE} \hfill

In this section we assume that $\alpha \in [0,1)$. In this case, we have 
$$ \nu _\alpha := \frac{1-\alpha}{2-\alpha} \in \left(0, \frac{1}{2}\right] .$$
The ODE we need to solve is
\begin{equation}
\label{*eq-bessel-ordre-nu}
y^2 \psi ''(y) + y \psi'(y) +(y^2-\nu ^2) \psi(y)=0 , y \in I \subset (0,+\infty)
\end{equation}
with $\nu = \nu_\alpha$ and $I = \left(0, \frac{2}{2-\alpha} \sqrt{\lambda}\right)$.
The above equation is called {\it Bessel's equation for functions of order $\nu$}. 
The fundamental theory of ordinary differential equations establishes that the solutions of \eqref{*eq-bessel-ordre-nu} generate a vector space $S_\nu$ of dimension 2. Consider the Bessel function of order $\nu$ and {\it of the first kind} $J_\nu$:
\begin{equation}
\label{*def-Jnu}
 J_\nu (y)
:= \sum_{m = 0} ^\infty  \frac{(-1)^m}{m! \ \Gamma (m+\nu+1) } \left( \frac{y}{2}\right) ^{2m+\nu}
=\sum_{m = 0} ^\infty  c_{\nu,m} ^+ y^{2m+\nu} , \qquad y \geq 0
\end{equation}
and $J_{-\nu}$:
\begin{equation}
\label{*def-J-nu}
 J_{-\nu} (y)
:= \sum_{m = 0} ^\infty  \frac{(-1)^m}{m! \ \Gamma (m-\nu+1) } \left( \frac{y}{2}\right) ^{2m-\nu}
=\sum_{m = 0} ^\infty  c_{\nu,m} ^- y^{2m-\nu} , \qquad y > 0 .
\end{equation}
Since $\nu \not \in  \mathbb N$, the two functions $J_\nu$ and $J_{-\nu}$ are linearly independent and therefore the pair $(J_\nu,J_{-\nu})$ forms a fundamental system of solutions of  \eqref{*eq-bessel-ordre-nu},
(see \cite[section 3.1, (8), p. 40]{Watson}, \cite[section 3.12, eq. (2), p. 43]{Watson} or \cite[eq. (5.3.2), p. 102]{Lebedev})). Hence,
\begin{multline}
\label{sol-gene-EDO-psi}
\begin{cases}
y^2 \psi ''(y) + y \psi'(y) +(y^2-\nu _\alpha ^2) \psi(y)=0 , 
\\
y \in I
\end{cases}
\\ 
\quad  \implies \quad 
\exists \ C_+, C_- \in \mathbb R, \quad \begin{cases}
\psi (y) = C_+ J_{\nu _\alpha} (y) + C_- J_{-\nu _\alpha} (y) , \\
 y\in I .
 \end{cases}
\end{multline}
Thus, going back to the original variables, we obtain that

\begin{multline}
\label{sol-gene-EDO-phi}
\begin{cases}
- (x^\alpha \Phi ') ' = \lambda \Phi , 
\\
 x \in (0,1)
\end{cases}
\\ 
\quad  \implies \quad 
\exists \ C_+, C_- \in \mathbb R, \quad \begin{cases}
\Phi (x) = C_+ \Phi _+ (x) + C_- \Phi _- (x) , \\
x \in (0,1) ,
 \end{cases}
\end{multline}
with
\begin{equation}
\label{phi+-w}
\begin{split}
\Phi _+ (x) &= x^{\frac{1-\alpha}{2}} J_{\nu_\alpha} \Bigl( \frac{2}{2-\alpha} \sqrt{\lambda} x^{\frac{2-\alpha}{2}}\Bigr)
= x^{\frac{1-\alpha}{2}} \sum_{m = 0} ^\infty  c_{\nu_\alpha,m} ^+ \Bigl( \frac{2}{2-\alpha} \sqrt{\lambda} x^{\frac{2-\alpha}{2}} \Bigr) ^{2m+\nu_\alpha}
\\
&= \sum_{m = 0} ^\infty  c_{\nu_\alpha,m} ^+ \Bigl( \frac{2}{2-\alpha} \sqrt{\lambda}  \Bigr) ^{2m+\nu_\alpha} x^{1-\alpha + (2-\alpha)m} 
= \sum_{m = 0} ^\infty  \tilde c_{\alpha,\lambda,m} ^+  x^{1-\alpha + (2-\alpha)m} 
\end{split}
\end{equation}
and
\begin{equation}
\label{phi--w}
\begin{split}
\Phi _- (x) &= x^{\frac{1-\alpha}{2}} J_{-\nu_\alpha} \Bigl( \frac{2}{2-\alpha} \sqrt{\lambda} x^{\frac{2-\alpha}{2}}\Bigr)
= x^{\frac{1-\alpha}{2}} \sum_{m = 0} ^\infty  c_{\nu_\alpha,m} ^- \Bigl( \frac{2}{2-\alpha} \sqrt{\lambda} x^{\frac{2-\alpha}{2}} \Bigr) ^{2m-\nu_\alpha}
\\
&= \sum_{m = 0} ^\infty  c_{\nu_\alpha,m} ^- \Bigl( \frac{2}{2-\alpha} \sqrt{\lambda}  \Bigr) ^{2m-\nu_\alpha} x^{(2-\alpha)m} 
= \sum_{m = 0} ^\infty  \tilde c_{\alpha,\lambda,m} ^-  x^{(2-\alpha)m} .
\end{split}
\end{equation}


\subsubsection{Information given by the functional setting} \hfill

Note that
$$ \Phi _+ (x) \to 0 \quad \text{ as } x \to 0^+,$$
hence $\Phi _+ \in L^2 (0,1)$. Moreover,
$$ \Phi ' _+ (x) \sim \tilde c_{\alpha,\lambda,0} ^+ \frac{1-\alpha}{x^{\alpha}} \quad \text{ as } x \to 0^+,$$
therefore, $\Phi _+$ is absolutely continuous on $[0,1]$. Furthermore,
$$ x^{\alpha/2} \Phi ' _+ (x) \sim \tilde c_{\alpha,\lambda,0} ^+ \frac{1-\alpha}{x^{\alpha /2}} \quad \text{ as } x \to 0^+,$$
thus, $\Phi _+ \in H^1 _\alpha (0,1)$. Finally,
$$ (x^\alpha \Phi ' _+)' (x) \to 0 \quad \text{ as } x \to 0^+,$$ 
and we deduce that $\Phi _+ \in H^2 _\alpha (0,1)$.

In the same way, one easily checks that $\Phi _- \in H^2 _\alpha (0,1)$.

\subsubsection{Information given by the boundary condition at $x=0$} \hfill

An eigenfunction must satisfy additionally the boundary conditions. In particular, we should have that $x^\alpha \Phi '(x) \to 0$ as $x\to 0$. We observe that
$$ x^\alpha \Phi ' _- (x) \to 0 \quad \text{ as } x \to 0, $$
while
$$ x^\alpha \Phi ' _+ (x) \to \tilde c_{\alpha,\lambda,0} ^+ (1-\alpha) \neq 0 \quad \text{ as } x \to 0. $$
Therefore, we conclude that
$$
\begin{cases}
- (x^\alpha \Phi ') ' = \lambda \Phi , \quad  x \in (0,1)
\\
(x^\alpha \Phi ')(0)=0
\end{cases}
\quad  \implies \quad 
\exists \ C_- \in \mathbb R, \quad \begin{cases}
\Phi (x) = C_- \Phi _- (x) , \\
x \in (0,1) ,
 \end{cases}
$$

\subsubsection{Information given by the boundary condition at $x=1$} \hfill

As regards the boundary condition at $x=1$, $\Phi$ has to solve $\Phi '(1)=0$. We compute $\Phi'_{-}$:
\begin{multline*}
\Phi _- ' (x) = \frac{1-\alpha}{2} x^{\frac{-1-\alpha}{2}} J_{-\nu_\alpha} \Bigl( \frac{2}{2-\alpha} \sqrt{\lambda} x^{\frac{2-\alpha}{2}}\Bigr)
\\
+ x^{\frac{1-\alpha}{2}} \sqrt{\lambda} x^{-\alpha /2} J' _{-\nu_\alpha} \Bigl( \frac{2}{2-\alpha} \sqrt{\lambda} x^{\frac{2-\alpha}{2}}\Bigr) .
\end{multline*}
and we deduce that the following relation must hold
\begin{equation}
\label{eq-vp-w1}
 \frac{1-\alpha}{2}  J_{-\nu_\alpha} \Bigl( \frac{2}{2-\alpha} \sqrt{\lambda} \Bigr)
+  \sqrt{\lambda}  J' _{-\nu_\alpha} \Bigl( \frac{2}{2-\alpha} \sqrt{\lambda} \Bigr)  = 0 .
\end{equation}
This is the equation that characterizes the eigenvalues $\lambda$. Multiplying by $\frac{2}{2-\alpha}$, 
\eqref{eq-vp-w1} becomes
\begin{equation}
\label{eq-vp-w2}
\frac{2}{2-\alpha} \frac{1-\alpha}{2}  J_{-\nu_\alpha} \Bigl( \frac{2}{2-\alpha} \sqrt{\lambda} \Bigr)
+ \frac{2}{2-\alpha} \sqrt{\lambda}  J' _{-\nu_\alpha} \Bigl( \frac{2}{2-\alpha} \sqrt{\lambda} \Bigr)  = 0 .
\end{equation}
Introducing the variable
$$ X_\lambda =  \frac{2}{2-\alpha} \sqrt{\lambda} ,$$
we have
\begin{equation}
\label{eq-vp-w3}
\nu _\alpha  \, J_{-\nu_\alpha} (X_\lambda)
+ X_\lambda  \, J' _{-\nu_\alpha} (X_\lambda)  = 0 .
\end{equation}
We now consider the following well-known relation (see \cite[p. 45, formula (4)]{Watson})
\begin{equation}
\label{Watson-}
z \, J_\nu ' (z) - \nu \, J_\nu (z) = z \, J_{\nu + 1} (z) .
\end{equation}
from which we deduce that
\begin{equation*}
\nu _\alpha  J_{-\nu_\alpha} (X_\lambda)
+ X_\lambda  J' _{-\nu_\alpha} (X_\lambda)
= X_\lambda  J' _{-\nu_\alpha} (X_\lambda)
- (-\nu_\alpha) J_{-\nu_\alpha} (X_\lambda)
= X_\lambda J_{-\nu_\alpha + 1} (X_\lambda).
\end{equation*}
Thus, equation \eqref{eq-vp-w2} is equivalent to
\begin{equation}
\label{eq-vp-w4}
X_\lambda J_{-\nu_\alpha + 1} (X_\lambda)  = 0 ,
\end{equation}
which implies that
\begin{equation}
\label{eq-vp-w5}
J_{-\nu_\alpha + 1} (X_\lambda)  = 0 .
\end{equation}
Thus, the possible values for $X_\lambda$ are the positive zeros of $J_{-\nu_\alpha + 1}$:
$$ \exists\,n\geq1\,:\, \frac{2}{2-\alpha} \sqrt{\lambda} = X_\lambda = j_{-\nu_\alpha + 1, n}.$$
We obtain that the eigenvalues of \eqref{vp} have the following form:
$$ \exists\,n\geq1\,:\,\lambda_n = \kappa _\alpha ^2 j_{-\nu_\alpha + 1, n} ^2.$$

Vice-versa, given $n\geq 1$,
consider
$$ \lambda _n := \kappa _\alpha ^2 \, j_{-\nu_\alpha + 1, n} ^2
\quad \text{ and } \quad 
\Phi _n (x) = x^{\frac{1-\alpha}{2}} \, J_{-\nu_\alpha} \Bigl( j_{-\nu_\alpha + 1, n} x^{\frac{2-\alpha}{2}}\Bigr) .$$
From the previous analysis we deduce that $\Phi _n \in H^2 _\alpha (0,1)$ and that $(\lambda _n, \Phi _{n})$ solves \eqref{vp}.

Finally, the proof of \eqref{gap-sqrt-weak} follows directly from \cite[p. 135]{Kom-Lor} : since
$-\nu_\alpha + 1 \geq \frac{1}{2}$, the sequence $(j_{-\nu_\alpha + 1, n+1} - j_{-\nu_\alpha + 1, n})_{n\geq 1}$ is nonincreasing and
$$ j_{-\nu_\alpha + 1, n+1} - j_{-\nu_\alpha + 1, n} \to \pi \quad \text{ as } n \to \infty .$$
This concludes the proof of Proposition \ref{prop-vp-w}. \qed 

\subsubsection{Additional information on the eigenvalues and eigenfunctions} \hfill
\label{sec-info-w}

We are going to prove useful properties of the eigenfuncions for $\alpha\in[0,1)$.

\begin{Lemma}
\label{lem-prop-phi-n-0-1-w}
Let $\alpha \in [0,1)$ and $n\geq 1$. Then, $\Phi _{\alpha,n}$ has finite limits as $x\to 0^+$ and $x\to 1^-$, and satisfies
\begin{equation}
\label{phin1-}
\vert \Phi _{\alpha,n} (1) \vert  = \sqrt{2-\alpha} ,
\end{equation}
and
\begin{equation}
\label{Leb-asymp6}
\Phi _{\alpha,n} (0) \sim  c_{\nu_\alpha,0} ^- \, \sqrt{\frac{(2-\alpha)\pi}{2}} \, (j_{-\nu_\alpha + 1, n}) ^{\frac{1}{2}-\nu_\alpha}
\quad \text{as } n \to +\infty 
\end{equation}
(where the coefficient $c_{\nu_\alpha,0} ^-$ is defined in \eqref{*def-J-nu}). In particular, the sequence $(\Phi _{\alpha,n} (0))_{n\geq 1}$ is bounded if and only if $\alpha =0$.

\end{Lemma}

\begin{proof}[Proof of Lemma \ref{lem-prop-phi-n-0-1-w}] First, we note that $ j_{-\nu_\alpha + 1, n}$ is not a zero of $J_{-\nu_\alpha}$:
\begin{equation}
\label{pas-zero}
\forall\, \alpha \in [0,1), \forall\, n\geq 1, \quad J_{-\nu_\alpha} (j_{-\nu_\alpha + 1, n}) \neq 0 .
\end{equation}
Indeed, if $J_{-\nu_\alpha} (j_{-\nu_\alpha + 1, n}) =0$, we derive from \eqref{Watson-} that $J' _{-\nu_\alpha} (j_{-\nu_\alpha + 1, n}) =0$, and then the Cauchy problem satisfied by $J_{-\nu_\alpha}$ would imply that $J_{-\nu_\alpha}$ is constantly equal to zero. 

We also derive from \eqref{Watson-} that

\begin{equation}
\label{lien-J'J-}
J' _{-\nu_\alpha} (j_{-\nu_\alpha + 1, n}) =   \frac{\nu_\alpha}{j_{-\nu_\alpha + 1, n}} J _{-\nu_\alpha} (j_{-\nu_\alpha + 1, n}) .
\end{equation}
We compute the value of the constants $K_{\alpha,n}$ that appear in \eqref{eq-vp-w-vpm}
$$
 1 = K_{\alpha,n} ^2 \int _0 ^1 x^{1-\alpha} \, J_{-\nu_\alpha} \Bigl( j_{-\nu_\alpha + 1, n}\,  x^{\frac{2-\alpha}{2}}\Bigr) ^2 \, dx .
$$
Thanks to the change of variables $y= x^{\frac{2-\alpha}{2}}$, we get
$$ 1 = K_{\alpha,n} ^2 \, \frac{2}{2-\alpha} \, \int _0 ^1 y \, J_{-\nu_\alpha} \Bigl( j_{-\nu_\alpha + 1, n}\, y \Bigr) ^2 \, dy ,$$
and applying \cite[formula (5.14.5) p.129]{Lebedev}, we obtain
$$ 1 = K_{\alpha,n} ^2 \, \frac{1}{2-\alpha} \,  \left( J' _{-\nu_\alpha} \left( j_{-\nu_\alpha + 1, n}\right) ^2 + \left(1 - \frac{\nu_\alpha ^2}{j_{-\nu_\alpha + 1, n} ^2}\right) J _{-\nu_\alpha} \left( j_{-\nu_\alpha + 1, n}\right) ^2 \right) .$$
Therefore,
$$  K_{\alpha,n} = \left( \frac{2-\alpha}{ J' _{-\nu_\alpha} ( j_{-\nu_\alpha + 1, n}) ^2 + \left(1 - \frac{\nu_\alpha ^2}{j_{-\nu_\alpha + 1, n} ^2}\right) J _{-\nu_\alpha} ( j_{-\nu_\alpha + 1, n}) ^2} \right) ^{1/2} ,$$
and using \eqref{lien-J'J-}, we obtain a simple expression for $K_{\alpha,n}$
\begin{equation}
\label{form-Kn-}
\forall\, \alpha \in [0,1), \forall\, n\geq 1, \quad K_{\alpha,n} = \frac{\sqrt{2-\alpha}}{\vert J _{-\nu_\alpha} ( j_{-\nu_\alpha + 1, n}) \vert } .
\end{equation}

Thus, from \eqref{eq-vp-w-vpm} we deduce
the value of $\vert \Phi _{\alpha,n} (1) \vert $ given in \eqref{phin1-}, and also the value of $\Phi _{\alpha,n} (0)$. Indeed, from \eqref{*def-J-nu}, we obtain that for all $n\geq1$, $\Phi _{\alpha,n}$ has a finite limit as $x\to 0^+$, and that
\begin{equation}
\label{phin0-1iere}
\Phi _{\alpha,n} (0) = \frac{\sqrt{2-\alpha}}{\vert J _{-\nu_\alpha} ( j_{-\nu_\alpha + 1, n}) \vert }
\, c_{\nu_\alpha,0} ^- \, (j_{-\nu_\alpha + 1, n}) ^{-\nu_\alpha} .
\end{equation}
Moreover, using the following classical asymptotic development (\cite[formula (5.11.6) p. 122]{Lebedev}):
\begin{equation}
\label{Leb-asymp1}
J_\nu (z) = \sqrt{\frac{2}{\pi z}} \left[ \cos(z-\frac{\nu\pi}{2}-\frac{\pi}{4}) (1+ O(\frac{1}{z^2})) + O(\frac{1}{z}) \right] \quad \text{as } z\to \infty ,
\end{equation}
we obtain that
\begin{equation}
\label{Leb-asymp3}
J_\nu (z)^2  = \frac{2}{\pi z} \cos^2 (z-\frac{\nu\pi}{2}-\frac{\pi}{4})+ O(\frac{1}{z^{2}}) \quad \text{as } z\to \infty .
\end{equation}
Applying the latter formula with $\nu +1$, we get
$$  z J_{\nu +1} (z)^2  = \frac{2}{\pi} \cos^2 (z-\frac{(\nu+1)\pi}{2}-\frac{\pi}{4})+ O(\frac{1}{z})
= \frac{2}{\pi} \sin^2 (z-\frac{\nu\pi}{2}-\frac{\pi}{4})+ O(\frac{1}{z}) .$$
Therefore,
$$ z J_{\nu } (z)^2 + z J_{\nu +1} (z)^2 = \frac{2}{\pi} + O(\frac{1}{z}) ,$$
which gives that
\begin{equation}
\label{Leb-asymp4}
z J_{\nu } (z)^2 + z J_{\nu +1} (z)^2 \to \frac{2}{\pi} \quad \text{as } z \to +\infty .
\end{equation}
So, we deduce that
$$ j_{-\nu_\alpha + 1, n} J_{-\nu _\alpha} ( j_{-\nu_\alpha + 1, n}) ^2 + j_{-\nu_\alpha + 1, n} J_{-\nu _\alpha +1} ( j_{-\nu_\alpha + 1, n}) ^2 \to  \frac{2}{\pi} \quad \text{as } n \to +\infty .$$
Hence,
\begin{equation}
\label{Leb-asymp5}
J_{-\nu _\alpha } ( j_{-\nu_\alpha + 1, n}) ^2 \sim  \frac{2}{\pi j_{-\nu_\alpha + 1, n}}  \quad \text{as } n \to +\infty ,
\end{equation}
and then, combining \eqref{Leb-asymp5} with \eqref{phin0-1iere} we finally obtain \eqref{Leb-asymp6}. 
\end{proof}


\subsection{The case $\alpha \in [1,2)$} \hfill

\subsubsection{Analysis of the ODE for $\alpha \in (1,2)$ and $\nu_\alpha = \frac{\alpha -1}{2-\alpha} \notin \mathbb N$} \hfill

In this section we want to solve problem \eqref{vp}, hence \eqref{vp-psi}, for $\nu_\alpha = \frac{\alpha -1}{2-\alpha} \notin \mathbb N$.
As recalled before, when $\nu_\alpha \notin \mathbb N$, $J_{\nu_\alpha}$ and $J_{-\nu_\alpha}$ form a fundamental system of solutions of \eqref{*eq-bessel-ordre-nu} (with $\nu = \nu_\alpha$). Therefore, \eqref{sol-gene-EDO-psi} and \eqref{sol-gene-EDO-phi} still hold. However, the difference lies in the functions $\Phi _+$ and $\Phi _-$: indeed, in this case we have
\begin{equation}
\label{phi+-s}
\begin{split}
\Phi _+ (x) &= x^{\frac{1-\alpha}{2}} J_{\nu_\alpha} \Bigl( \frac{2}{2-\alpha} \sqrt{\lambda} x^{\frac{2-\alpha}{2}}\Bigr)
= x^{\frac{1-\alpha}{2}} \sum_{m = 0} ^\infty  c_{\nu_\alpha,m} ^+ \Bigl( \frac{2}{2-\alpha} \sqrt{\lambda} x^{\frac{2-\alpha}{2}} \Bigr) ^{2m+\nu_\alpha}
\\
&= \sum_{m = 0} ^\infty  c_{\nu_\alpha,m} ^+ \Bigl( \frac{2}{2-\alpha} \sqrt{\lambda}  \Bigr) ^{2m+\nu_\alpha} x^{(2-\alpha)m} 
= \sum_{m = 0} ^\infty  \tilde c_{\alpha,\lambda,m} ^+  x^{(2-\alpha)m} 
\end{split}
\end{equation}
and
\begin{equation}
\label{phi--s}
\begin{split}
\Phi _- (x) &= x^{\frac{1-\alpha}{2}} J_{-\nu_\alpha} \Bigl( \frac{2}{2-\alpha} \sqrt{\lambda} x^{\frac{2-\alpha}{2}}\Bigr)
= x^{\frac{1-\alpha}{2}} \sum_{m = 0} ^\infty  c_{\nu_\alpha,m} ^- \Bigl( \frac{2}{2-\alpha} \sqrt{\lambda} x^{\frac{2-\alpha}{2}} \Bigr) ^{2m-\nu_\alpha}
\\
&= \sum_{m = 0} ^\infty  c_{\nu_\alpha,m} ^- \Bigl( \frac{2}{2-\alpha} \sqrt{\lambda}  \Bigr) ^{2m-\nu_\alpha} x^{(2-\alpha)m} 
= \sum_{m = 0} ^\infty  \tilde c_{\alpha,\lambda,m} ^-  x^{1-\alpha + (2-\alpha)m} .
\end{split}
\end{equation}  

\subsubsection{Information given by the functional setting for $\nu_\alpha = \frac{\alpha -1}{2-\alpha} \notin \mathbb N$} \hfill

We note that 
$$ \Phi _+ (x) \to \tilde c_{\alpha,\lambda,0} ^+ \quad \text{ as } x \to 0^+,$$
hence, $\Phi _+ \in L^2 (0,1)$. Moreover,
$$ x^{\alpha/2} \Phi ' _+ (x) \sim \tilde c_{\alpha,\lambda,1} ^+ (2-\alpha) x^{1-\frac{\alpha}{2}}  \quad \text{ as } x \to 0^+,$$
which implies that $\Phi _+ \in H^1 _\alpha (0,1)$. Furthermore, we have that
$$ (x^\alpha \Phi ' _+)' (x) \to \tilde c_{\alpha,\lambda,1} ^+ (2-\alpha) \quad \text{ as } x \to 0^+,$$ 
thus, $\Phi _+ \in H^2 _\alpha (0,1)$. However, for $\Phi_-$ it holds that
$$ x^{\alpha/2} \Phi ' _- (x) \sim \tilde c_{\alpha,\lambda,1} ^- (1-\alpha) x^{-\alpha/2}  \quad \text{ as } x \to 0^+,$$ 
and we deduce that $\Phi _- \notin H^1 _\alpha (0,1)$, and, in particular, $\Phi _- \notin H^2 _\alpha (0,1)$. Therefore, $C_-=0$ and \eqref{sol-gene-EDO-phi} yields
\begin{equation}
\label{sol-gene-EDO-phi-s1}
\begin{cases}
- (x^\alpha \Phi ') ' = \lambda \Phi , 
\\
 x \in (0,1)
\end{cases}
\quad  \implies \quad 
\exists \ C_+ \in \mathbb R, \quad \begin{cases}
\Phi (x) = C_+ \Phi _+ (x) , \\
x \in (0,1) .
 \end{cases}
\end{equation}

\subsubsection{Information given by the boundary condition at $x=0$ for $\nu_\alpha = \frac{\alpha -1}{2-\alpha} \notin \mathbb N$} \hfill

Observe that in this case $ (x^\alpha \Phi _+ ') (0) = 0$, and therefore the boundary condition at $x=0$ is automatically satisfied.

\subsubsection{Information given by the boundary condition at $x=1$ for $\nu_\alpha = \frac{\alpha -1}{2-\alpha} \notin \mathbb N$} \hfill

In order to be an eigenfunction, $\Phi$ has to solve the second boundary condition: $\Phi '(1)=0$. We recall that
\begin{multline*}
\Phi _+ ' (x) = \frac{1-\alpha}{2} x^{\frac{-1-\alpha}{2}} J_{\nu_\alpha} \Bigl( \frac{2}{2-\alpha} \sqrt{\lambda} x^{\frac{2-\alpha}{2}}\Bigr)
+ x^{\frac{1-\alpha}{2}} \sqrt{\lambda} x^{-\alpha /2} J' _{\nu_\alpha} \Bigl( \frac{2}{2-\alpha} \sqrt{\lambda} x^{\frac{2-\alpha}{2}}\Bigr) .
\end{multline*}
Hence, if $\Phi$ is an eigenfunction, $C_+ \neq 0$ and $\Phi ' _+ (1)=0$, and so
\begin{equation}
\label{eq-vp-s11}
 \frac{1-\alpha}{2}  J_{\nu_\alpha} \Bigl( \frac{2}{2-\alpha} \sqrt{\lambda} \Bigr)
+  \sqrt{\lambda}  J' _{\nu_\alpha} \Bigl( \frac{2}{2-\alpha} \sqrt{\lambda} \Bigr)  = 0 .
\end{equation}
This is the equation that characterizes the eigenvalues $\lambda$. Multiplying by $\frac{2}{2-\alpha}$, 
\eqref{eq-vp-s11} becomes
\begin{equation}
\label{eq-vp-s12}
\frac{2}{2-\alpha} \frac{1-\alpha}{2}  J_{\nu_\alpha} \Bigl( \frac{2}{2-\alpha} \sqrt{\lambda} \Bigr)
+ \frac{2}{2-\alpha} \sqrt{\lambda}  J' _{\nu_\alpha} \Bigl( \frac{2}{2-\alpha} \sqrt{\lambda} \Bigr)  = 0 .
\end{equation}
Introducing once again the variable
$$ X_\lambda =  \frac{2}{2-\alpha} \sqrt{\lambda} ,$$
equation \eqref{eq-vp-s12} can be rewritten as
\begin{equation}
\label{eq-vp-s13}
- \nu _\alpha  \, J_{\nu_\alpha} (X_\lambda)
+ X_\lambda  \, J' _{\nu_\alpha} (X_\lambda)  = 0 .
\end{equation}
Using again \eqref{Watson-}, we have
$$ - \nu _\alpha  \, J_{\nu_\alpha} (X_\lambda)
+ X_\lambda  \, J' _{\nu_\alpha} (X_\lambda) = X_\lambda J_{\nu_\alpha + 1} (X_\lambda) .$$
Thus, \eqref{eq-vp-s12} implies
\begin{equation}
\label{eq-vp-s14}
X_\lambda J_{\nu_\alpha + 1} (X_\lambda)  = 0 ,
\end{equation}
or, equivalently
\begin{equation}
\label{eq-vps15}
J_{\nu_\alpha + 1} (X_\lambda)  = 0 .
\end{equation}
The possible values for $X_\lambda$ are the positive zeros of $J_{\nu_\alpha + 1}$:
$$ \exists\, n \geq 1, \quad \frac{2}{2-\alpha} \sqrt{\lambda} = X_\lambda = j_{\nu_\alpha + 1, n}.$$
The above identity provides the following expression for the eigenvalues
$$ \exists\, n \geq 1, \quad \lambda_n = \kappa _\alpha ^2 j_{\nu_\alpha + 1, n} ^2 .$$

Vice-versa, given $n\geq 1$,
consider
$$ \lambda _n := \kappa _\alpha ^2 \, j_{\nu_\alpha + 1, n} ^2
\quad \text{ and } \quad 
\Phi _n (x) = x^{\frac{1-\alpha}{2}} \, J_{\nu_\alpha} \Bigl( j_{\nu_\alpha + 1, n} x^{\frac{2-\alpha}{2}}\Bigr) .$$
From the previous analysis we deduce that $\Phi _n \in H^2 _\alpha (0,1)$ and that $(\lambda _n, \Phi _n)$ solves \eqref{vp}.

Finally, the proof of \eqref{gap-sqrt-strong} follows directly from \cite[ p. 135]{Kom-Lor}: since
$\nu_\alpha + 1 \geq 1 > \frac{1}{2}$, the sequence $(j_{\nu_\alpha + 1, n+1} - j_{\nu_\alpha + 1, n})_{n\geq 1}$ is decreasing 
$$ j_{\nu_\alpha + 1, n+1} - j_{\nu_\alpha + 1, n} \to \pi \quad \text{ as } n \to \infty .$$
This concludes the proof of the related part of Proposition \ref{prop-vp-s}. \qed


\subsubsection{The main changes for $\nu_\alpha = \frac{\alpha -1}{2-\alpha} \in \mathbb N$} \hfill

In this case, it has been proved in \cite{CMV-strong} that \eqref{sol-gene-EDO-phi-s1} remains true (with $\Phi _+$ defined in \eqref{phi+-s}, the only difference is that now the fundamental system of the solutions of \eqref{*eq-bessel-ordre-nu} (with $\nu = \nu_\alpha$) involves $J_\nu$ and $Y_\nu$, the Bessel's function of order $\nu$ and of second kind (see \cite[section 3.54, eq. (1)-(2), p. 64]{Watson} or \cite[eq. (5.4.5)-(5.4.6), p. 104]{Lebedev}). Studying the behavior as $x\to 0$, one obtains once again \eqref{sol-gene-EDO-phi-s1} (we refer to \cite{CMV-strong} for the details). Then, one can conclude the proof of Proposition \ref{prop-vp-s} by reasoning as for the case $\nu_\alpha\notin \mathbb{N}$. \qed


\subsubsection{Additional information on the eigenvalues and eigenfunctions} \hfill
\label{sec-info-s}


We are going to prove the following results, that will be useful in the following:

\begin{Lemma}
\label{lem-prop-phi-n-0-1-s}
Let $\alpha \in [1,2)$. Then, the $\Phi _{\alpha,n}$ has finite limit as $x\to 0^+$ and $x\to 1^-$, and satisfies
\begin{equation}
\label{phin1+}
\forall\, n \geq 1, \quad \vert \Phi _{\alpha,n} (1) \vert  = \sqrt{2-\alpha} ,
\end{equation}
and
\begin{equation}
\label{Leb-asymp6+}
\Phi _{\alpha,n} (0) \sim  c_{\nu_\alpha,0} ^+ \, \sqrt{\frac{(2-\alpha)\pi}{2}} \, (j_{\nu_\alpha + 1, n}) ^{\frac{1}{2} + \nu_\alpha }
\quad \text{as } n \to +\infty 
\end{equation}
(where the coefficient $c_{\nu_\alpha,0} ^+$ is defined in \eqref{*def-Jnu}). In particular, the sequence $(\Phi _{\alpha,n} (0))_{n\geq 1}$ is unbounded.

\end{Lemma}

\noindent {\it Proof of Lemma \ref{lem-prop-phi-n-0-1-s}.} First, we note that $ j_{\nu_\alpha + 1, n}$ is not a zero of $J_{\nu_\alpha}$:
\begin{equation}
\label{pas-zero+}
\forall\, \alpha \in [1,2), \forall\, n\geq 1, \quad J_{\nu_\alpha} (j_{\nu_\alpha + 1, n}) \neq 0 .
\end{equation}
Indeed, if $J_{\nu_\alpha} (j_{\nu_\alpha + 1, n}) =0$, we derive from \eqref{Watson-} that $J' _{\nu_\alpha} (j_{\nu_\alpha + 1, n}) =0$, and then the Cauchy problem satisfied by $J_{\nu_\alpha}$ would imply that $J_{\nu_\alpha}$ is constantly equal to zero. 

We also deduce from \eqref{Watson-} that

\begin{equation}
\label{lien-J'+}
J' _{\nu_\alpha} (j_{\nu_\alpha + 1, n}) =  \frac{\nu_\alpha}{j_{\nu_\alpha + 1, n}} J _{\nu_\alpha} (j_{\nu_\alpha + 1, n}),\quad \forall\,n\geq1 .
\end{equation}

With the same strategy adopted in Lemma \ref{lem-prop-phi-n-0-1-w}, we compute the value of $K_{\alpha,n}$ that appears
in \eqref{eq-vp-s-vpm} and we find that
\begin{equation}
\label{form-Kn+}
\forall\, \alpha \in [1,2), \forall\, n\geq 1, \quad K_{\alpha,n} = \frac{\sqrt{2-\alpha}}{\vert J _{\nu_\alpha} ( j_{\nu_\alpha + 1, n}) \vert } .
\end{equation}

Therefore, we obtain from \eqref{eq-vp-s-vpm}
the value of $\vert \Phi _{\alpha,n} (1) \vert $ given in \eqref{phin1+}, and the value of $\Phi _{\alpha,n} (0)$. Indeed, using \eqref{*def-Jnu}, we have
\begin{equation}
\label{phin0-1iere+}
\Phi _{\alpha,n} (0) = \frac{\sqrt{2-\alpha}}{\vert J _{\nu_\alpha} ( j_{\nu_\alpha + 1, n}) \vert }
\, c_{\nu_\alpha,0} ^+ \, (j_{\nu_\alpha + 1, n}) ^{\nu_\alpha} ,
\end{equation}
and, in particular, function $\Phi _{\alpha,n}$ has a finite limit as $x\to 0$. Moreover, using \eqref{Leb-asymp4}, we obtain
$$ j_{\nu_\alpha + 1, n} J_{\nu _\alpha} ( j_{\nu_\alpha + 1, n}) ^2 + j_{\nu_\alpha + 1, n} J_{\nu _\alpha +1} ( j_{\nu_\alpha + 1, n}) ^2 \to  \frac{2}{\pi} \quad \text{as } n \to +\infty .$$
Hence,
\begin{equation}
\label{Leb-asymp5+}
J_{\nu _\alpha } ( j_{\nu_\alpha + 1, n}) ^2 \sim  \frac{2}{\pi j_{\nu_\alpha + 1, n}}  \quad \text{as } n \to +\infty ,
\end{equation}
and, combining \eqref{Leb-asymp5+} with \eqref{phin0-1iere}, we deduce \eqref{Leb-asymp6+}. \qed



\section{Proof of Proposition \ref{prop-"thm3"-red}}
\label{sec-hidden}

Let us give a more precise formulation of Proposition \ref{prop-"thm3"-red}:

\begin{Proposition}
\label{prop-"thm3"}
Let $\mu \in V^2 _\alpha(0,1)$ (defined in \eqref{eq-space-mu}). Then, 

a) for all $p \in L^2 (0,T)$, the solution $w^{(p)}$ of \eqref{CMU-eq-ctr-10} has the following additional regularity
\begin{equation}
\label{reg-wT}
(w^{(p)} (T), w^{(p)} _t (T))  \in H^3 _{(0)}(0,1) \times D(A) ,
\end{equation}
so, the map 
\begin{equation}
\label{def-thetaT}
\Theta _T: L^2 (0,T) \to H^3 _{(0)}(0,1) \times D(A), \quad \Theta _T (p) := (w^{(p)} (T), w^{(p)} _t (T)) 
\end{equation}
is well-defined,

b) given $p \in L^2 (0,T)$, $\Theta_T$ is differentiable at $p$, and $D\Theta _T (p): L^2 (0,T) \to H^3 _{(0)}(0,1) \times D(A)$ is a continuous linear application and satisfies
$$ D\Theta _T (p)\cdot q = (W^{(p,q)} (T), W^{(p,q)} _t (T)),$$
where $W^{(p,q)}$ is the solution of 
\begin{equation}
\label{CMU-eq-ctr-10-W}
\begin{cases}
W^{(p,q)}_{tt} - (x^\alpha W^{(p,q)} _x)_x = p(t) \mu (x) W^{(p,q)} + q(t) \mu (x) w^{(p)} , \quad &x\in (0,1), t \in (0,T), 
\\
(x^\alpha W^{(p,q)}_x)(x=0,t)=0 , \quad &t \in (0,T) ,\\
W^{(p,q)}_x (x=1,t)=0 , \quad &t \in (0,T) , \\
W^{(p,q)}(x,0)=0 ,  \quad &x\in (0,1), 
\\
W^{(p,q)} _t (x,0)=0 , \quad &x\in (0,1) ,
\end{cases}
\end{equation}

c) the map $\Theta _T$ is of class $C^1$.

\end{Proposition}
The proof of Proposition \ref{prop-"thm3"} is based on several steps, the first one consists in analyzing the eigenvalues and eigenfunctions of the operator $\mathcal A$.


\subsection{Eigenvalues and eigenfunctions of $\mathcal A$} \hfill

Let us give the following preliminary result.

\begin{Lemma}
\label{lem-vpAA} 
Consider, for all $n\in\mathbb{Z}$,
\begin{equation}
\label{lem-eq-vpAA}
\omega _{\alpha,n} = 
\begin{cases}
-  \sqrt{\lambda _{\alpha,\vert n \vert }}, \quad & n\leq -1 ,
\\
0, \quad & n=0 ,
\\
  \sqrt{\lambda _{\alpha,n}}, \quad & n\geq 1 ,
\end{cases}
\end{equation}
and
\begin{equation}
\label{lem-eq-fpAA}
\Psi _{\alpha,n} =  
\begin{cases}
\left( \begin{array}{c} \Phi _{\alpha,\vert n \vert}  \\ - i \sqrt{\lambda _{\alpha,\vert n \vert}} \Phi _{\alpha,\vert n \vert} \end{array} \right), \quad & n\leq -1 ,
\\
\ 
\\
\left( \begin{array}{c} \Phi _{\alpha,0} =1 \\ 0 \end{array} \right), \quad & n=0 ,
\\
\ 
\\
\left( \begin{array}{c} \Phi _{\alpha,n}  \\ i \sqrt{\lambda _{\alpha,n}} \Phi _{\alpha,n} \end{array} \right), \quad & n\geq 1 .
\end{cases}
\end{equation}
Then, $\{\omega_{\alpha,n}\}_{n\in\mathbb Z}$ and $\{\Psi_{\alpha,n}\}_{n\in\mathbb Z}$ fulfill
\begin{equation}
\label{eq-obs-vpAA}
\forall\, n \in \mathbb Z, \quad \mathcal A  \Psi _{\alpha,n} = i \, \omega _{\alpha,n} \, \Psi _{\alpha,n}  .
\end{equation}
\end{Lemma}

\begin{proof}[Proof of Lemma \ref{lem-vpAA}]
A direct computation shows that for all $n\geq 0$ we have that
\begin{equation*}
\begin{split}
\mathcal A  \Psi _{\alpha,n}
&= \left( \begin{array}{cc} 0 & \text{Id} \\ A & 0 \end{array} \right) \left( \begin{array}{c} \Phi _{\alpha,n}  \\ i \sqrt{\lambda _{\alpha,n}} \Phi _{\alpha,n} \end{array} \right)
= \left( \begin{array}{c} i \sqrt{\lambda _{\alpha,n}} \Phi _{\alpha,n} \\ A \Phi _{\alpha,n}  \end{array} \right)= \left( \begin{array}{c} i \sqrt{\lambda _{\alpha,n}} \Phi _{\alpha,n} \\ - \lambda _{\alpha,n} \Phi _{\alpha,n}  \end{array} \right)\\
&= i \, \sqrt{\lambda _{\alpha,n}} \, \left( \begin{array}{c}  \Phi _{\alpha,n} \\ i \sqrt{\lambda _{\alpha,n}} \Phi _{\alpha,n}  \end{array} \right)
= i \, \omega _{\alpha,n} \, \Psi _{\alpha,n},
\end{split}
\end{equation*}
and, for $n\leq -1$,
\begin{equation*}
\begin{split}
\mathcal A  \Psi _{\alpha,n}
&= \left( \begin{array}{cc} 0 & \text{Id} \\ A & 0 \end{array} \right) \left( \begin{array}{c} \Phi _{\alpha,\vert n \vert}  \\ - i \sqrt{\lambda _{\alpha,\vert n \vert}} \Phi _{\alpha,\vert n \vert} \end{array} \right)
= \left( \begin{array}{c} - i \sqrt{\lambda _{\alpha, \vert n \vert}} \Phi _{\alpha,\vert n \vert} \\ A \Phi _{\alpha,\vert n \vert}  \end{array} \right)\\
&= \left( \begin{array}{c} - i \sqrt{\lambda _{\alpha, \vert n \vert}} \Phi _{\alpha,\vert n \vert} \\ - \lambda _{\alpha,\vert n \vert } \Phi _{\alpha, \vert n \vert}  \end{array} \right)
= - i \, \sqrt{\lambda _{\alpha,\vert n \vert }} \, \left( \begin{array}{c}  \Phi _{\alpha, \vert n \vert} \\- i \sqrt{\lambda _{\alpha,\vert n \vert }} \Phi _{\alpha,\vert n \vert }  \end{array} \right)
= i \, \omega _{\alpha,n} \, \Psi _{\alpha,n},
\end{split}
\end{equation*}
which are exactly the identities in \eqref{eq-obs-vpAA}.
\end{proof}

In order to rigorously compute the eigenvalues and eigenfuntions of $\mathcal{A}$, one has to introduce the natural extension of $\mathcal A$ to the complex valued functions:
$$ \forall\, (f,h), (g,k) \in D(\mathcal A), \quad \tilde{ \mathcal A } \left( \begin{array}{c} f (x)  + i g(x) \\ h(x) + i k(x)  \end{array} \right)
:= \mathcal A  \left( \begin{array}{c} f (x)  \\ h(x)  \end{array} \right)
+ i \mathcal A  \left( \begin{array}{c} g (x)  \\ k(x)  \end{array} \right) ,$$
and then we can investigate the spectral problem

\begin{equation}
\label{eq-vpAA}
\tilde{ \mathcal A } \tilde \Psi = \tilde \omega \tilde \Psi \quad \text{ with } \quad 
\tilde \Psi (x) = \left( \begin{array}{c} f (x)  + i g(x) \\ h(x) + i k(x)  \end{array} \right), 
\quad (f,h), (g,k) \in D(\mathcal A).
\end{equation}

\begin{Lemma}
\label{lem-vpAA2}
The set of solutions $(\tilde\omega, \tilde \Psi)$ of problem \eqref{eq-vpAA} is
$$\tilde {\mathcal S } = \{ (i \omega _{\alpha,n}, \tilde \rho \, \Psi _{\alpha,n}), n\in \mathbb Z, \tilde \rho \in \mathbb C \}. $$
where $\omega _{\alpha,n}$ is defined in \eqref{lem-eq-vpAA}, and $\Psi _{\alpha,n}$
is defined in \eqref{lem-eq-fpAA}.
\end{Lemma}

\begin{proof}[Proof of Lemma \ref{lem-vpAA2}]
Let $ \tilde \omega = a+ib $, Then, we have
$$ \tilde{ \mathcal A }\left( \begin{array}{c} f   + i g \\ h + i k  \end{array} \right)
= \left( \begin{array}{c} h+ik \\ Af + i Ag  \end{array} \right) , 
\quad 
\tilde \omega \left( \begin{array}{c} f   + i g \\ h + i k  \end{array} \right)
= \left( \begin{array}{c} (af-bg) + i (ag+bf) \\ (ah-bk)+i(ak+bh) \end{array} \right).$$ 
Thus, \eqref{eq-vpAA} reduces to the following system
$$\begin{cases}
h= af-bg  \\
k= ag+bf  \\
Af= ah-bk  \\
Ag= ak+bh .
\end{cases}$$
We deduce that
$$ \begin{cases}
Af = a(af-bg) -b(ag+bf) = (a^2 - b^2) f -2abg\\
Ag = a(ag+bf)+b(af-bg) = (a^2 - b^2) g +2ab f .
\end{cases}$$
To solve the above system, we have to distinguish several cases:
\begin{itemize}
\item $\tilde \omega =0$: in this case we have that $Af=0=Ag$. Thus, $f$ and $g$ are constant: $f=c_1$, $g=c_2$ (with $c_1,c_2 \in \mathbb R$) and $h=0=k$. Therefore,
$$ \tilde \omega = 0, \text{ and } \tilde \Psi = \left( \begin{array}{c} f+ig \\ 0 \end{array} \right) = (c_1+ic_2) \left( \begin{array}{c} 1 \\ 0 \end{array} \right) ;$$

\item {\it$\tilde \omega \neq 0$ and  $f=0$: we must have }$g\neq 0$ (otherwise it would imply $h=0=k$). So, we get that $ab=0$ and $g$ is an eigenfunction of $A$ associated to the eigenvalue $a^2-b^2$. However, since it must hold that $a=0$ or $b=0$, and $A$ is non-positive, the only possibility is that $a=0$, and therefore for any $n\geq 1$, $b^2= \lambda _{\alpha,n}$, and $g=c \Phi _{\alpha,n}$ (for some $c \in \mathbb R$). Moreover, we have that $h=-bg$ and $k=0$. Thus, in this case we have two possible sets of solutions:
$$ \tilde \omega = a+ib = i \sqrt{\lambda _{\alpha,n}}, 
\text{ and } 
\tilde \Psi = \left( \begin{array}{c} i c \Phi _{\alpha,n} \\  - \sqrt{\lambda _{\alpha,n}} c \Phi _{\alpha,n}  \end{array} \right)
= ic \left( \begin{array}{c} \Phi _{\alpha,n} \\ i \sqrt{\lambda _{\alpha,n}}  \Phi _{\alpha,n}  \end{array} \right) ,$$
or
$$ \tilde \omega = a+ib = -i \sqrt{\lambda _{\alpha,n}}, \text{ and } \tilde \Psi = \left( \begin{array}{c} ic \Phi _{\alpha,n} \\  \sqrt{\lambda _{\alpha,n}} c \Phi _{\alpha,n}  \end{array} \right)
= ic \left( \begin{array}{c} \Phi _{\alpha,n} \\ -i \sqrt{\lambda _{\alpha,n}}  \Phi _{\alpha,n}  \end{array} \right); $$

\item {\it$\tilde \omega \neq 0$ and  $g=0$: similarly to the previous case, we deduce that} $a=0$, $b^2= \lambda _{\alpha,n}$  and $f=c \Phi _{\alpha,n}$ for every $n\geq1$ (for some $c \in \mathbb R$). Furthermore, $h=0$ and $k=bf$. Hence, also in this case, we have two possible sets of solutions:
$$ \tilde \omega = a+ib = i \sqrt{\lambda _{\alpha,n}}, 
\text{ and } 
\tilde \Psi = \left( \begin{array}{c}  c \Phi _{\alpha,n} \\  i \sqrt{\lambda _{\alpha,n}} c \Phi _{\alpha,n}  \end{array} \right)
= c \left( \begin{array}{c} \Phi _{\alpha,n} \\ i \sqrt{\lambda _{\alpha,n}}  \Phi _{\alpha,n}  \end{array} \right) ,$$
or
$$ \tilde \omega = a+ib = -i \sqrt{\lambda _{\alpha,n}}, \text{ and } \tilde \Psi = \left( \begin{array}{c} c \Phi _{\alpha,n} \\  -i\sqrt{\lambda _{\alpha,n}} c \Phi _{\alpha,n}  \end{array} \right)
= c \left( \begin{array}{c} \Phi _{\alpha,n} \\ -i \sqrt{\lambda _{\alpha,n}}  \Phi _{\alpha,n}  \end{array} \right); $$

\item {\it$\tilde \omega \neq 0$, $f\neq 0 \neq g$ and $g$ is colinear with $f$: let} $g=r f$ (with some $r\in \mathbb R)$. Then $f$ is an eigenfunction of $A$, and 
$$ (a^2 - b^2)  -2abr = (a^2 - b^2)  + \frac{2ab}{r} .$$
Thus,
$$ 2ab(r+ \frac{1}{r})=0.$$
Since $r\in \mathbb R$, then $ab=0$, and we can reason as in previous cases: so, we get that $a=0$ ( from the non-positivity of $A$), and for every $n\geq1$, $b^2= \lambda _{\alpha,n}$, and $f=c \Phi _{\alpha,n}$ (for some $c \in \mathbb R$). Moreover, we obtain that $g= c \Phi _{\alpha,n}$, $h=0-bg$ and $k=bf$, and this leads to two possibilities: 
$$ \tilde \omega = i \sqrt{\lambda _{\alpha,n}}, 
\text{ and } 
\tilde \Psi 
= c(1+ir) \left( \begin{array}{c} \Phi _{\alpha,n} \\ i \sqrt{\lambda _{\alpha,n}}  \Phi _{\alpha,n}  \end{array} \right) ,$$
or
$$ \tilde \omega = -i \sqrt{\lambda _{\alpha,n}}, \text{ and } \tilde \Psi 
= c(1+ir) \left( \begin{array}{c} \Phi _{\alpha,n} \\ -i \sqrt{\lambda _{\alpha,n}}  \Phi _{\alpha,n}  \end{array} \right); $$

\item {\it$\tilde \omega \neq 0$ and  $f$ and $g$ are free}: in this case we note that the plane generated by $f$ and $g$ is stable under the action of $A$. Let us denote by $P_{f,g}$ the plane generated by $f$ and $g$, and by $A_{f,g}$ the restriction of $A$ to this plane. $A_{f,g}$ is an endomorphism, and it is symmetric (and non-positive). Since
$$ Af = (a^2 - b^2) f -2abg, \quad Ag = (a^2 - b^2) g +2ab f ,$$
the matrix that defines the action of $A_{f,g}$ on the basis composed by $f$ and $g$ is
$$\text{Mat} (A_{f,g}; f,g) = \left( \begin{array}{cc} a^2-b^2 & -2ab \\ 2ab & a^2-b^2  \end{array} \right)
= (a^2-b^2)\,  \text{Id}_{P_{f,g}} \, + \, 2ab \left( \begin{array}{cc} 0 & 1 \\ -1 & 0  \end{array} \right) .$$
So, since $A_{f,g}$ is symmetric, $\text{Mat} (A_{f,g}; f,g)$ is diagonalizable. This would imply that
$$ \left( \begin{array}{cc} 0 & 1 \\ -1 & 0  \end{array} \right)$$
would also be diagonalizable, which is a contradiction.
\end{itemize}
Finally, one easily verifies that these necessary conditions are also sufficient. This concludes the proof of Lemma \ref{lem-vpAA2}.
\end{proof}




\subsection{Two integral expressions for the solution of \eqref{CMU-eq-ctr-10}} \hfill
\label{sec-form-gamman}

Since the family $\{\Phi_{\alpha,n}\}_{n\in\mathbb{N}}$ is an orthonormal basis of $L^2(0,1)$, we can decompose the solution $w^{(p)}$ of \eqref{CMU-eq-ctr-10} under the form
$$ w^{(p)}(x,t) = \sum _{n=0} ^\infty w_n (t) \Phi _{\alpha,n} (x) .$$
We decompose in the same way the nonlinear term
$$ r(x,t) := p(t) \mu (x) w^{(p)} (x,t) = \sum _{n=0} ^\infty r_n (t) \Phi _{\alpha,n} (x) $$
with 
\begin{equation}
\label{def-r_n}
r_n (t) = \langle p(t) \mu(\cdot)  w^{(p)} (\cdot ,t) , \Phi _{\alpha,n} \rangle _{L^2 (0,1)}.
\end{equation}
So, \eqref{CMU-eq-ctr-10} implies that the sequence $(w_n(t))_{n\geq 0}$ satisfies
$$ \begin{cases}
w_0 '' (t) = r_0 (t) , \\ 
w _0 (0)=1, \\ 
w_0 ' (0) = 0 , 
\end{cases}
\quad \text{ and } \quad 
\forall\, n\geq 1, \quad
\begin{cases}
w _n '' (t) + \lambda _{\alpha,n} w _n (t) = r_n (t) , \\ 
w_n (0)=0, \\ 
w_n ' (0) = 0 .
\end{cases}$$
We obtain that
$$ w _0 (t) = 1 + \int _0 ^t r_0 (s) (t-s) \, ds 
\quad \text{ and } \quad w _n (t) = \int _0 ^t r_n (s) \frac{\sin \sqrt{\lambda _{\alpha,n}} (t-s)}{\sqrt{\lambda _{\alpha,n}}} \, ds .$$
Hence, the solution of \eqref{CMU-eq-ctr-10} can be written as
\begin{multline}
\label{formule-sin}
w^{(p)} (x,t) = \left( 1 + \int _0 ^t r_0 (s) (t-s) ds  \right)\\
+ \sum _{n=1} ^\infty \left( \int _0 ^t r_n (s) \frac{\sin \sqrt{\lambda _{\alpha,n}} (t-s)}{\sqrt{\lambda _{\alpha,n}}} ds \right) \Phi _{\alpha,n} (x) ,
\end{multline}
and
\begin{multline}
\label{formule-sin'}
w^{(p)} _t (x,t) = \left( \int _0 ^t r_0 (s) ds  \right) 
\\
+ \sum _{n=1} ^\infty  \left( \sqrt{\lambda _{\alpha,n}} \int _0 ^t r_n (s) \frac{\cos \sqrt{\lambda _{\alpha,n}} (t-s)}{\sqrt{\lambda _{\alpha,n}}} \, ds \right) \Phi _{\alpha,n} (x) ,
\end{multline}
or, equivalently,
$$ \left( \begin{array}{c} w^{(p)}(x,t)  \\ w^{(p)} _t (x,t) \end{array} \right)
= \left( \begin{array}{c} w _0(t)  \\ w' _0 (t) \end{array} \right)
+ \sum _{n=1} ^\infty \left( \begin{array}{c} \Bigl( \int _0 ^t r_n (s) \frac{\sin \sqrt{\lambda _{\alpha,n}} (t-s)}{\sqrt{\lambda _{\alpha,n}}} \, ds  \Bigr) \Phi _{\alpha,n} (x)  \\ \Bigl( \sqrt{\lambda _{\alpha,n}} \int _0 ^t r_n (s) \frac{\cos \sqrt{\lambda _{\alpha,n}} (t-s)}{\sqrt{\lambda _{\alpha,n}}} \, ds \Bigr) \Phi _{\alpha,n} (x) \end{array} \right).
$$
Now, manipulating the above formula and we get
\begin{equation*}
\begin{split}
&\left( \begin{array}{c} w^{(p)}(x,t)  \\ w^{(p)} _t (x,t) \end{array} \right)
- \left( \begin{array}{c} w _0(t)  \\ w' _0 (t) \end{array} \right)=
\\
&\quad= \sum _{n=1} ^\infty \frac{1}{2i \sqrt{\lambda _{\alpha,n}}}  
\left( \begin{array}{c} \left( \int _0 ^t r_n (s)  ( e^{i\sqrt{\lambda _{\alpha,n}} (t-s)}-e^{-i\sqrt{\lambda _{\alpha,n}} (t-s)}) \, ds  \right) \Phi _{\alpha,n} (x)  
\\ 
\left( \int _0 ^t r_n (s) ( e^{i\sqrt{\lambda _{\alpha,n}} (t-s)} + e^{-i\sqrt{\lambda _{\alpha,n}} (t-s)}) \, ds \right) i \sqrt{\lambda _{\alpha,n}}  \Phi _{\alpha,n} (x) \end{array} \right)
\\
&\quad= \sum _{n=1} ^\infty \frac{1}{2i \sqrt{\lambda _{\alpha,n}}}  
\left( \int _0 ^t r_n (s) \, e^{-i\sqrt{\lambda _{\alpha,n}} s} \, ds \right) 
\left( \begin{array}{c} \Phi _{\alpha,n} (x)   \\ i \sqrt{\lambda _{\alpha,n}} \Phi _{\alpha,n} (x)  \end{array} \right)
e^{i\sqrt{\lambda _{\alpha,n}} t}
\\
&\quad\quad- \frac{1}{2i \sqrt{\lambda _{\alpha,n}}}  
\left( \int _0 ^t r_n (s)\, e^{i\sqrt{\lambda _{\alpha,n}} s} \, ds \right) 
\left( \begin{array}{c} \Phi _{\alpha,n} (x)   \\ -i \sqrt{\lambda _{\alpha,n}} \Phi _{\alpha,n} (x)  \end{array} \right)
e^{-i\sqrt{\lambda _{\alpha,n}} t}
\\
&\quad= \sum _{n\in \mathbb Z, n\neq 0} \frac{1}{2i \omega _{\alpha,n}}   \left( \int _0 ^t r_{\vert n \vert} (s)\, e^{-i \omega _{\alpha,n} s} \, ds \right) \left( \begin{array}{c} \Phi _{\alpha,\vert n \vert} (x)   \\ i \omega_{\alpha,n} \Phi _{\alpha,\vert n \vert} (x)  \end{array} \right)
e^{i\omega_{\alpha,n} t} ,
\end{split}
\end{equation*}
that can be expressed more compactly as
\begin{equation}
\begin{split}
\label{formule-exp}
\left( \begin{array}{c} w^{(p)}(x,t)  \\ w^{(p)} _t (x,t) \end{array} \right)
&= \left( \begin{array}{c} 1 + \int _0 ^t r_0 (s) (t-s) \, ds   \\ \int _0 ^t r_0 (s) \, ds  \end{array} \right)
\\
&\quad+ \sum _{n\in \mathbb Z, n\neq 0} \frac{1}{2i \omega _{\alpha,n}}   \left( \int _0 ^t r_{\vert n \vert} (s)\, e^{-i \omega _{\alpha,n} s} \, ds \right) \Psi _{\alpha,n} (x) e^{i\omega_{\alpha,n} t} .
\end{split}
\end{equation}
To lighten the notation, we rewrite \eqref{formule-exp} as
\begin{equation}
\label{formule-exp2}
\left( \begin{array}{c} w^{(p)}(x,T)  \\ w^{(p)} _t (x,T) \end{array} \right)
= \Gamma ^{(p)} _0 (T)
+ \sum _{n\in \mathbb Z^*} \frac{1}{2i \omega _{\alpha,n}}  \, \gamma ^{(p)} _n (T) \, \Psi _{\alpha,n} (x) \, e^{i\omega_{\alpha,n} T} ,
\end{equation}
with 
\begin{equation}
\label{def-gamma0T}
\Gamma ^{(p)} _0 (T) = \left( \begin{array}{c} 1 + \int _0 ^T r_0 (s) (T-s) \, ds   \\ \int _0 ^T r_0 (s) \, ds  \end{array} \right) = \left( \begin{array}{c} \gamma ^{(p)} _{00} (T)   \\ \gamma ^{(p)} _{01} (T)  \end{array} \right) ,
\end{equation}
and
\begin{equation}
\label{def-gammanT}
\forall\, n \in \mathbb Z ^*, \quad 
\gamma ^{(p)} _n (T) = \int _0 ^T r_{\vert n \vert} (s)\, e^{-i \omega _{\alpha,n} s} \, ds ,
\end{equation}
where we recall that $r_{\vert n \vert }$ is defined in \eqref{def-r_n}.

Formula \eqref{formule-exp2} shows the role of the functions $(x,t) \mapsto \Psi _{\alpha,n} (x) \, e^{i\omega_{\alpha,n} t}$,
which are solution of the homogeneous equation
$$ (\Psi _{\alpha,n} (x) \, e^{i\omega_{\alpha,n} t}) _t = \mathcal A (\Psi _{\alpha,n} (x) \, e^{i\omega_{\alpha,n} t}) .$$


\subsection{A sufficient condition to prove Proposition \ref{prop-"thm3"}, part a)} \hfill
\label{sec-suff-cond}

From Proposition \ref{prop-"prop2"}, we already know that $(w(T), w_t(T)) \in D(A) \times H^1 _\alpha (0,1)$. To prove the hidden regularity result, it is useful to consider the expression \eqref{formule-exp2}. We have that
$$ w^{(p)} (T) - \gamma ^{(p)} _{00} (T) = \sum _{n=1} ^\infty 
\frac{1}{2i \omega _{\alpha,n}}  \left( \gamma ^{(p)} _n (T) \, e^{i\omega_{\alpha,n} T} 
- \gamma ^{(p)} _{-n} (T) \, e^{-i\omega_{\alpha,n} T} \right) \Phi _{\alpha,n} (x) ,$$
hence, $w^{(p)}(T) \in H^3 _{(0)} (0,1)$ if and only if
$$ \sum _{n=1} ^\infty \lambda _{\alpha ,n} ^3 \left \vert \frac{1}{2i \omega _{\alpha,n}}  \left( \gamma ^{(p)} _n (T) \, e^{i\omega_{\alpha,n} T} 
- \gamma ^{(p)} _{-n} (T) \, e^{-i\omega_{\alpha,n} T} \right) \right \vert  ^2  < \infty .$$
Moreover,
$$ w^{(p)} _t (T) - \gamma ^{(p)} _{01} (T) = \sum _{n=1} ^\infty 
\frac{1}{2}  \left( \gamma ^{(p)} _n (T) \, e^{i\omega_{\alpha,n} T} 
+ \gamma ^{(p)} _{-n} (T) \, e^{-i\omega_{\alpha,n} T} \right) \Phi _{\alpha,n} (x) ,$$
thus, $w^{(p)}_t (T) \in H^2 _{(0)} (0,1)$ if and only if
$$ \sum _{n=1} ^\infty \lambda _{\alpha ,n} ^2 \left \vert \frac{1}{2}  \left( \gamma ^{(p)} _n (T) \, e^{i\omega_{\alpha,n} T} 
+ \gamma ^{(p)} _{-n} (T) \, e^{-i\omega_{\alpha,n} T} \right) \right \vert  ^2  < \infty .$$
Therefore,
\begin{multline}
\label{condition-H3H2}
\sum _{n \in \mathbb Z} \lambda _{\alpha , \vert n \vert } ^2 \, \vert \gamma ^{(p)} _n (T) \vert  ^2  < \infty
\\
\quad \implies \quad (w^{(p)}(T), w^{(p)}_t(T)) \in H^3 _{(0)} (0,1) \times D(A) .
\end{multline}
In what follows, we prove that
\begin{equation}
\label{eq-cond-suff}
\sum _{n \in \mathbb Z} \lambda _{\alpha , \vert n \vert } ^2 \, \vert \gamma ^{(p)} _n (T) \vert  ^2  < \infty ,
\end{equation}
or, equivalently, using definition \eqref{def-gammanT} of $\gamma ^{(p)} _n (T)$ and definition \eqref{def-r_n} of $r_n$, that
$$ \sum _{n \in \mathbb Z} \lambda _{\alpha , \vert n \vert } ^2 \, \Bigl\vert \, \int _0 ^T \langle p(s) \mu (x) w^{(p)} (\cdot ,s) , \Phi _{\alpha,\vert n \vert} \rangle _{L^2 (0,1)} \, e^{-i \omega _{\alpha,n} s} \, ds \,\Bigr\vert  ^2  < \infty .$$
For this purpose, we will prove
\begin{itemize}
\item a general convergence result (see Lemma \ref{lem-cond-h2} in section \ref{sec-gen-reg}):
$$ \sum _{n =1} ^\infty \lambda _{\alpha ,n} ^2 \, \Bigl\vert \, \int _0 ^T \langle p(s) \, g (\cdot,s), \Phi _{\alpha,n} \rangle _{L^2 (0,1)} \, e^{i\omega _{\alpha,n} s} \, ds \, \Bigr\vert  ^2  < \infty $$
if $g\in C^0 ([0,T], V^{2,0} _\alpha (0,1))$,
\item a regularity result (see Lemma \ref{lem-reg-muw} in section \ref{sec-reg-muw}):
$$ \mu \in V^2 _\alpha(0,1), w\in C^0 ([0,T], D(A)) \quad \implies \quad \mu w \in C^0 ([0,T], V^{(2,0)} _\alpha (0,1)) .$$
\end{itemize}
Then, \eqref{eq-cond-suff} will easily follow, see section \ref{sec-reg-H3-ok}.
(The intermediate lemmas \ref{lem-Sn1-l2}-\ref{lem-S3-h2} are useful to prove Lemma \ref{lem-cond-h2}.)

\subsection{A general regularity result} \hfill
\label{sec-gen-reg}

In this section a fundamental role will be played by the space $V^{(2,0)} _\alpha (0,1)$ defined in \eqref{space-muV10}.

\begin{Lemma}
\label{lem-cond-h2}
Let $T>0$, $p\in L^2 (0,T)$, $g\in C^0 ([0,T], V^{(2,0)} _\alpha (0,1))$.
Consider the sequence
$(S^{(p,g)} _n)_{n\geq 1}$ defined by
\begin{equation}
\label{form-Sn}
\forall \,n \geq 1, \quad  S^{(p,g)} _n = \int _0 ^T p(s) \, \langle g (\cdot,s), \Phi _{\alpha,n} \rangle _{L^2 (0,1)} \, e^{i\sqrt{\lambda _{\alpha,n}} s} \, ds .
\end{equation}
Then, $(S^{(p,g)} _n)_{n\geq 1}$ satisfies
\begin{equation}
\label{prop-Sn-1}
\sum _{n =1} ^\infty \lambda _{\alpha ,n} ^2 \, \vert S^{(p,g)} _n \vert  ^2  < \infty ,
\end{equation}
and moreover, there exists a constant $C(\alpha,T) >0$ independent of $p \in L^2 (0,T)$ and of $g\in C^0 ([0,T], V^{(2,0)} _\alpha (0,1))$ such that
\begin{equation}
\label{prop-Sn-2}
\Bigl( \sum _{n =1} ^\infty \lambda _{\alpha ,n} ^2 \, \vert S^{(p,g)} _n \vert  ^2  \Bigr) ^{1/2}
\leq C(\alpha,T) \, \Vert p \Vert _{L^2(0,T)}  \,  \Vert g \Vert _{C^0 ([0,T], V^{(2,0)} _\alpha (0,1))}  .
\end{equation}
\end{Lemma}

\begin{proof}[Proof of Lemma \ref{lem-cond-h2}]
We proceed as in \cite{B-bi}, and the properties of the space $V^{(2,0)} _\alpha (0,1)$ will be crucial to overcome some new difficulties. (Note that $V^{(2,0)} _\alpha (0,1) = H^2 _\alpha (0,1)$ for $\alpha \in [1,2)$.) 

First, we note that
\begin{equation}\label{Spgn}
\begin{split}
S^{(p,g)} _n& = \int _0 ^T p(s) \, \langle g (\cdot,s), \Phi _{\alpha,n} \rangle _{L^2 (0,1)} \, e^{i\sqrt{\lambda _{\alpha,n}} s} \, ds
\\
&= \int _0 ^T p(s) \, \langle g (\cdot,s), \frac{1}{\lambda _{\alpha,n}} (-A\Phi _{\alpha,n} )\rangle _{L^2 (0,1)} \, e^{i\sqrt{\lambda _{\alpha,n}} s} \, ds
\\
&= \frac{-1}{\lambda _{\alpha,n}} \int _0 ^T p(s) \, \langle g (\cdot,s), (x^\alpha \Phi _{\alpha,n} ' )' \rangle _{L^2 (0,1)} \, e^{i\sqrt{\lambda _{\alpha,n}} s} \, ds .
\end{split}
\end{equation}
Then, integrating by parts, we have
\begin{equation*}
\begin{split}
&\langle g(\cdot,s) , (x^\alpha \Phi _{\alpha,n} ' )' \rangle _{L^2 (0,1)}
= \int _0 ^1 g (x,s) (x^\alpha \Phi _{\alpha,n} '(x) )' \, dx 
\\
&\qquad\qquad\qquad= [ g(x,s) x^\alpha \Phi _{\alpha,n} '(x) ]_0 ^1 - \int _0 ^1 g'(x,s) x^\alpha \Phi _{\alpha,n} '(x) \, dx 
\\
&\qquad\qquad\qquad= [ g(x,s) x^\alpha \Phi _{\alpha,n} '(x) ]_0 ^1 - [ x^\alpha g'(x,s) \Phi _{\alpha,n} (x)] _0 ^1 
+ \int _0 ^1 (x^\alpha g'(x,s))' \Phi _{\alpha,n}(x) \, dx .
\end{split}
\end{equation*}
Using the above expression of the scalar product in \eqref{Spgn}, we get
\begin{equation}
\label{eq-Sn-Sn123}
- \lambda _{\alpha,n} \, S^{(p,g)} _n = S^{(1)} _n - S^{(2)} _n + S^{(3)} _n ,
\end{equation}
with
\begin{equation}
\label{eq-Sni-l2}
\forall\, i\in \{1,2,3\}, \quad S^{(i)} _n = \int _0 ^T h^{(i)} _n (s) \, e^{i\sqrt{\lambda _{\alpha,n}} s} \, ds ,
\end{equation}
and the associated functions
\begin{equation}
\label{def-h1}
h^{(1)} _n (s) = p(s) \, [ g(x,s) x^\alpha \Phi _{\alpha,n} '(x) ]_{x=0} ^{x=1},
\end{equation}
\begin{equation}
\label{def-h2}
\,\,\,h^{(2)} _n (s) = p(s) \, [ x^\alpha g_x(x,s) \Phi _{\alpha,n} (x)]_{x=0} ^{x=1},
\end{equation}
\begin{equation}
\label{def-h3}
h^{(3)} _n (s) = p(s) \, \langle (x^\alpha g_x)_x , \Phi _{\alpha,n} \rangle _{L^2 (0,1)} .
\end{equation}
To conclude the proof we appeal to the following results.


\subsubsection{The term $S^{(1)} _n$ associated to $h^{(1)}$} \hfill
\begin{Lemma}
\label{lem-Sn1-l2}
Let $T>0$, $p\in L^2 (0,T)$ and $g\in C^0 ([0,T], V^{(2,0)} _\alpha (0,1))$. Then, function
$h^{(1)} _n$ defined in \eqref{def-h1} satisfies
$$ \forall\, s \in [0,T], \quad h^{(1)} _n (s) = 0 . $$
\end{Lemma}

\begin{proof}[Proof of Lemma \ref{lem-Sn1-l2}] Since $g(\cdot,s) \in H^2 _\alpha (0,1)$, then $g(\cdot,s) \in H^1 (\frac{1}{2},1)$ and has a finite limit as $x\to 1$. Hence, thanks to the Neumann boundary condition at $x=1$ for $\Phi _{\alpha,n}$, we have
$$ g(x,s) x^\alpha \Phi _{\alpha,n} '(x) \to 0 \quad \text{ as } x \to 1 .$$

When $x\to 0$, we have to distinguish the cases of weak and strong degeneracy:
\begin{itemize}
\item $\alpha \in [0,1)$: first, we notice that $g(\cdot, s)$ has a finite limit as $x\to 0$. Indeed,
$$ g_x (x,s) = (x^{\alpha /2} g_x (x,s)) \, x^{-\alpha /2} ,$$
and since $x\mapsto x^{\alpha /2} g_x (x,s)$ and $x\mapsto x^{-\alpha /2}$ belong to $L^2 (0,1)$, then $g_x (\cdot, s) \in L^1 (0,1)$,  which implies that $g(\cdot, s)$ has a finite limit as $x\to 0$. Therefore,
$$ g(x,s) x^\alpha \Phi _{\alpha,n} '(x) \to 0 \quad \text{ as } x \to 0 $$

\item $\alpha \in [1,2)$: observe that $g$ can be unbounded as $x\to 0$. However, the series of $\Phi _{\alpha,n}$ obtained thanks to 
\eqref{phi+-s} gives that
$$\exists\, C_{\alpha,n}\,:\, \vert x^\alpha \Phi _{\alpha,n} '(x) \vert \leq C_{\alpha,n} x,\quad  \forall\, x \in (0,1).$$
We claim that $x\mapsto x g(x,s) $ has a finite limit as $x\to 0$. Indeed,
$$ (x g(x,s)) _x =  g(x,s) + x g_x(x,s) = g(x,s) + x^{\alpha /2}  g_x(x,s) x^{1-\alpha /2} ,$$
and since $g(\cdot,s)\in L^2 (0,1)$, $x\mapsto x^{\alpha /2}  g_x(x,s) \in L^2 (0,1)$ and $x\mapsto x^{1-\alpha /2} \in L^\infty (0,1)$, we have that $(x g(x,s)) _x \in L^1 (0,1)$. Therefore $x\mapsto x g(x,s) $ has a finite limit as $x\to 0$:
$$ \exists\, \ell ^{(s)}, \quad x g(x,s) \to \ell ^{(s)} \quad \text{ as } x \to 0 .$$
However, since $g(\cdot,s) \in L^2 (0,1)$, we get that $x\mapsto \frac{\ell ^{(s)}}{x} \in L^2 (0,1)$, which is possible only if $\ell ^{(s)} =0$. Thus,
$$ x g(x,s) \to 0 \quad \text{ as } x \to 0 ,$$
and so
$$ g(x,s) x^\alpha \Phi _{\alpha,n} '(x) \to 0 \quad \text{ as } x \to 0 .$$
\end{itemize}
\end{proof}


\subsubsection{The term $S^{(2)} _n$ associated to $h^{(2)}$} \hfill
\begin{Lemma}
\label{lem-Sn2-l2}
Let $T>0$, $p\in L^2 (0,T)$ and $g\in C^0 ([0,T], V^{(2,0)} _\alpha (0,1))$. Then, function
$h^{(2)} _n$ defined in \eqref{def-h2} 
belongs to $L^2 (0,T)$  and there exists $C (\alpha,T)>0$ independent of $p \in L^2 (0,T)$ and of $g\in C^0 ([0,T], V^{(2,0)} _\alpha (0,1))$ and of $n\geq 1$ such that
\begin{equation}
\label{eq-Sn2-h2}
\Vert h^{(2)} _n \Vert _{L^2 (0,T)} \leq C (\alpha,T) \, \Vert p \Vert _{L^2(0,T)}  \,  \Vert g \Vert _{C^0 ([0,T], V^{2,0} _\alpha (0,1))} .
\end{equation}
Furthermore, 
\begin{equation}
\label{Sn123-l2}
\sum _{n=1} ^\infty \vert S^{(2)} _n \vert ^2  < \infty ,
\end{equation}
and there exists $C_2(\alpha,T)>0$ independent of $p \in L^2 (0,T)$ and of $g\in C^0 ([0,T], V^{(2,0)} _\alpha (0,1))$ such that
\begin{equation}
\label{eq-Sn2-l2}
\sum _{n=1} ^\infty \vert S^{(2)} _n \vert ^2 
 \leq C_2 (\alpha,T) ^2 \, \Vert p \Vert _{L^2(0,T)} ^2  \,  \Vert g \Vert _{C^0 ([0,T], V^{(2,0)} _\alpha (0,1))} ^2 .
\end{equation}
\end{Lemma}

\begin{proof}[Proof of Lemma \ref{lem-Sn2-l2}]
We recall that
\begin{multline}
\label{h2-dev}
h^{(2)} _n (s) = \left.p(s) \, (x^\alpha g_x(x,s) \Phi _{\alpha,n} (x))\right|_{x=1}
-\left.p(s) \, (x^\alpha g_x(x,s) \Phi _{\alpha,n} (x))\right|_{x=0}.
\end{multline}
Using the definition of $V^{(2,0)}_{\alpha} (0,1)$ and respectively \eqref{Leb-asymp6} when $\alpha \in [0,1)$ and \eqref{Leb-asymp6+} when $\alpha \in [1,2)$ in \eqref{h2-dev}, we obtain that 
$$  \left.(x^\alpha g_x(x,s) \Phi _{\alpha,n} (x))\right|_{x=0} = 0.$$
Therefore, from \eqref{phin1-} and \eqref{phin1+}, we have
\begin{equation*}
\vert h^{(2)} _n (s) \vert = \vert \left.p(s) \, (x^\alpha g_x(x,s) \Phi _{\alpha,n} (x))\right|_{x=1} \vert
= \sqrt{2-\alpha} \vert \left.p(s) \, (x^\alpha g_x(x,s)\right|_{x=1} \vert .
\end{equation*}
Moreover, since $g(\cdot,s) \in H^2 _\alpha (0,1)$, then  $x\mapsto x^\alpha g_x (x,s)$ belongs to $H^1 (0,1)$. By the continuous injection of $H^1 (0,1)$ into $C^0 ([0,1])$, 
there exists a positive constant $C_\infty$ such that
$$
 \vert \left.(x^\alpha g_x(x,s) )\right|_{x=1} \vert \leq C_\infty \Vert g(\cdot, s )\Vert _{H^2 _\alpha (0,1)} 
 \leq C_\infty \Vert g \Vert _{C^0 ([0,T], V^{(2,0)} _\alpha (0,1))} .
$$
Therefore, we get
$$\vert h^{(2)} _n (s) \vert \leq  C_\infty \, \sqrt{2-\alpha} \, \vert p (s) \vert \, \Vert g \Vert _{C^0 ([0,T], V^{(2,0)} _\alpha (0,1))},\quad \forall \,n \geq 1,
$$
hence, $h^{(2)} _n \in L^2 (0,T)$ and
\begin{equation}
\label{h2-conc}
\exists\, C' _\infty >0\,\,: \,\, \Vert h^{(2)} _n \Vert _{L^2 (0,T)} \leq C' _\infty \, \Vert p \Vert _{L^2 (0,T)} \, \Vert g \Vert _{C^0 ([0,T], V^{(2,0)} _\alpha (0,1))},\, \forall\, n \geq 1.
\end{equation}
This proves \eqref{eq-Sn2-h2}. 

Now, we prove \eqref{Sn123-l2} and \eqref{eq-Sn2-l2}.
These results follow from \eqref{eq-Sn2-h2} and from classical results of Ingham type (we refer, in particular, 
to \cite[Proposition 19, Theorem 6 and Corollary 4]{B-L}).
We have seen in Proposition \ref{prop-vp-w}, when $\alpha \in [0,1)$, and in Proposition \ref{prop-vp-s}, when $\alpha \in [1,2)$, that 
$$ \sqrt{\lambda _{\alpha, n+1}} - \sqrt{\lambda _{\alpha, n}} \to \frac{2-\alpha}{2} \pi \quad \text{ as } n \to \infty .$$
Furthermore, a stronger gap condition holds
$$ \forall\, \alpha \in [0,2), \quad \sqrt{\lambda _{\alpha, n+1}} - \sqrt{\lambda _{\alpha, n}} \geq \frac{2-\alpha}{2} \pi ,$$
hence we are allowed to apply a general result of Ingham (see, e.g., \cite[Theorem 4.3]{Kom-Lor}, generalized by Haraux \cite{Haraux}, see also \cite[Theorem 6]{B-L}), and we derive that given 
$$ T_1 > T_0 := \frac{2\pi}{\frac{2-\alpha}{2} \pi} = \frac{4}{2-\alpha},$$
there exist $C_1 (\alpha,T_1), C_2 (\alpha,T_1) >0$ such that, for for every sequence
$(c_n)_{n\geq 1}$ with finite support and complex values, it holds that  
\begin{equation}
\label{Ingham1}
C_1 \sum _{n=1} ^\infty \vert c_n \vert ^2 
\leq \int _0 ^{T_1} \Bigl \vert \sum _{n=1} ^\infty  c_n e^{i\sqrt{\lambda _{\alpha, n}} t} \Bigr \vert ^2 \, dt  
\leq C_2 \sum _{n=1} ^\infty \vert c_n \vert ^2 .
\end{equation}
Therefore, if $T_1 > T_0$, \eqref{Ingham1} implies that the sequence $(e^{i\sqrt{\lambda _{\alpha, n}} t})_{n\geq 1}$ is a Riesz basis of 
$\overline{ \text{Vect } \{ e^{i\sqrt{\lambda _{\alpha, n}} t}, n\geq 1 \}} \subset L^2 (0,T_1)$ (see \cite[Proposition 19, point (2)]{B-L} ). So, for all $T>0$, there exists a positive constant $C_I(\alpha,T)$ such that
\begin{equation}
\label{Ingham2}
\forall\, f \in L^2 (0,T), \quad \sum _{n=1} ^\infty \left\vert \int _0 ^T f(t) e^{i\sqrt{\lambda _{\alpha, n}}t} \, dt \right \vert ^2 
\leq C_I (\alpha, T) \Vert f \Vert _{L^2 (0,T)} ^2 
\end{equation}
by applying \cite[Proposition 19, point (3)]{B-L} for $T> T_0$, or by extending $f$ by $0$ on $(T,T_0)$ for $T\leq T_0$, see also \cite[Corollary 4]{B-L}).
We can now conclude the proof of Lemma \ref{lem-Sn2-l2}. First, we note from \eqref{h2-dev} that
\begin{equation*}
\begin{split}
\left\vert \int _0 ^T h^{(2)} _n (t)\,  e^{i\sqrt{\lambda _{\alpha, n}}t} \, dt \right \vert
&= \left\vert \int _0 ^T \left.p(t) \, (x^\alpha g_x(x,t) \Phi _{\alpha,n} (x))\right|_{x=1}  \,  e^{i\sqrt{\lambda _{\alpha, n}}t} \, dt \right \vert 
\\
&= \sqrt{2-\alpha} \, \left\vert \int _0 ^T p(t) \, g_x(1,t)  \,  e^{i\sqrt{\lambda _{\alpha, n}}t} \, dt \right \vert ,
\end{split}
\end{equation*}
and then we can apply \eqref{Ingham2} to the function $t \mapsto p(t) \, g_x(1,t)$ (which is independent of $n$), and we obtain that 
\begin{equation*}
\begin{split}
\sum _{n=1} ^\infty \vert S^{(2)} _n \vert ^2
&= \sum _{n=1} ^\infty \Bigl \vert \int _0 ^T h^{(2)} _n (s) \, e^{i\sqrt{\lambda _{\alpha,n}} s} \, ds \Bigr \vert ^2 = (2-\alpha) \sum _{n=1} ^\infty \Bigl \vert \int _0 ^T p(t) \, g_x(1,t)  \,  e^{i\sqrt{\lambda _{\alpha, n}}t} \, dt \Bigr \vert  ^2
\\
&\leq (2-\alpha) \, C_I (\alpha, T) \Vert p(\cdot) \, g_x(1,\cdot) \Vert _{L^2 (0,T)} ^2
\\
&\leq (2-\alpha) \, C_I (\alpha, T) C_\infty ^2 \, \Vert p \Vert _{L^2(0,T)}^2  \,  \Vert g \Vert _{C^0 ([0,T], V^{(2,0)} _\alpha (0,1))} ^2.
\end{split}
\end{equation*}
This concludes the proof of Lemma \ref{lem-Sn2-l2}.
\end{proof}


\subsubsection{The term $S^{(3)} _n$ associated to $h^{(3)}$} \hfill

Finally, we analyze $h^{(3)} _n$, and we prove the following

\begin{Lemma}
\label{lem-S3-h2}
Let $T>0$, $p\in L^2 (0,T)$ and $g\in C^0 ([0,T], V^{(2,0)} _\alpha (0,1))$. Then, the sequence  $(S^{(3)} _n) _{n\geq 1}$ satisfies
\begin{equation}
\label{Sn123-h3}
\sum _{n=1} ^\infty \vert S^{(3)} _n \vert ^2 < \infty ,
\end{equation}
and there exists $C_3 (\alpha,T)>0$ independent of $p \in L^2 (0,T)$ and of $g\in C^0 ([0,T], V^{(2,0)} _\alpha (0,1))$ such that
\begin{equation}
\label{eq-Sn3-l2}
\sum _{n=1} ^\infty \vert S^{(3)} _n \vert ^2 
 \leq C_3 (\alpha,T) ^2 \, \Vert p \Vert _{L^2(0,T)} ^2  \,  \Vert g \Vert _{C^0 ([0,T], V^{(2,0)} _\alpha (0,1))} ^2 .
\end{equation}
\end{Lemma}

\noindent {\it Proof of Lemma \ref{lem-S3-h2}.} First, let us prove that $\sum _n \vert S^{(3)} _n \vert ^2 < \infty$. Notice that
$$
\vert S^{(3)} _n \vert ^2 \leq  \Bigl( \int _0 ^T \vert h^{(3)} _n (s)\vert  \, ds \Bigr) ^2
\leq \Vert p \Vert _{L^2(0,T)} ^2 \Bigl( \int _0 ^T \vert \langle (x^\alpha g_x)_x , \Phi _{\alpha,n} \rangle _{L^2 (0,1)} \vert ^2 \, ds \Bigr) ,
$$
hence,
\begin{equation*}
\begin{split}
\sum _{n=1} ^\infty \vert S^{(3)} _n \vert ^2
&\leq \Vert p \Vert _{L^2(0,T)} ^2 \Bigl( \int _0 ^T \sum _{n=1} ^\infty \vert \langle (x^\alpha g_x)_x , \Phi _{\alpha,n} \rangle _{L^2 (0,1)} \vert ^2 \, ds \Bigr) 
\\
&\leq \Vert p \Vert _{L^2(0,T)} ^2 \Bigl( \int _0 ^T \Vert (x^\alpha g_x)_x \Vert _{L^2 (0,1)} ^2 \, ds \Bigr)\\
&\leq \Vert p \Vert _{L^2(0,T)} ^2 T \Vert g \Vert _{C^0 ([0,T], V^{(2,0)} _\alpha (0,1))} ^2 .
\end{split}
\end{equation*}
This concludes the proof of Lemma \ref{lem-S3-h2}. \qed

The proof of Lemma \ref{lem-cond-h2} follows directly from \eqref{eq-Sn-Sn123} and Lemmas \ref{lem-Sn1-l2}, \ref{lem-Sn2-l2} and \ref{lem-S3-h2}.
\end{proof}


\subsection{Proof of Proposition \ref{prop-"thm3"}, part a)} \hfill
\label{sec-pr-prop-a}

In this section we prove that Lemma \ref{lem-cond-h2} implies that \eqref{eq-cond-suff} holds true,
and then from \eqref{condition-H3H2} we deduce that $(w^{(p)}(T), w^{(p)} _t(T)) \in H^3 _{(0)} (0,1) \times D(A)$, which is the aim in point a) of Proposition \ref{prop-"thm3"}.

\subsubsection{A regularity result} \hfill
\label{sec-reg-muw}

\begin{Lemma}
\label{lem-reg-muw}
If $\mu \in V^2 _\alpha(0,1)$ (defined in \eqref{eq-space-mu}) and $w\in C^0 ([0,T], D(A))$, then $\mu w \in C^0 ([0,T], V^{(2,0)} _\alpha (0,1))$. Moreover, there exists $C(\alpha,T) >0$, independent of $\mu \in V^2 _\alpha$ and $w\in C^0 ([0,T], D(A))$, such that
\begin{equation}
\label{eq-mu-w-muw}
\Vert \mu w \Vert _{C^0 ([0,T], V^{(2,0)} _\alpha (0,1))} 
\leq C(\alpha,T) \,  \Vert \mu \Vert _{V^2 _\alpha}  \, \Vert  w \Vert _{C^0 ([0,T], D(A))} .
\end{equation}
\end{Lemma}

\begin{proof}[Proof of Lemma \ref{lem-reg-muw}] We distinguish the cases $\alpha \in [0,1)$ and $\alpha \in [1,2)$.
\begin{itemize}
\item $\alpha \in [0,1)$: let $\mu \in V^{(2,\infty)} _\alpha (0,1)$ and $w \in V^{(2,0)} _\alpha (0,1)$. As we have already shown, $\mu \in H^2 _\alpha (0,1)$ implies that $\mu \in L^\infty (0,1)$, and so $\mu w \in L^2 (0,1)$. Moreover, we have that $(\mu w)_x = \mu _x w + \mu w_x\in L^1(0,1)$ because $\mu, w \in L^\infty (0,1)$, $\mu _x = (x^{\alpha /2} \mu _x) 
x^{-\alpha /2} \in L^1 (0,1)$ and $w_x \in L^1 (0,1)$. Thus, $\mu w$ is absolutely continuous on $[0,1]$. Furthermore, $x^{\alpha /2} (\mu w)_x = (x^{\alpha /2} \mu _x) w + (x^{\alpha /2} w_x) \mu\in L^2(0,1)$ because $x^{\alpha /2} \mu _x, x^{\alpha /2} w_x \in L^2 (0,1)$ and $w, \mu \in L^\infty (0,1)$. We observe that $(x^\alpha (\mu w)_x)_x = (x^\alpha \mu _x)_x w + (x^\alpha w_x)_x \mu + 2 x^\alpha \mu _x w_x$. Since $(x^\alpha \mu _x)_x,\, (x^\alpha w_x)_x \in L^2 (0,1)$ and $w,\mu \in L^\infty (0,1)$, we deduce that $(x^\alpha \mu _x)_x w + (x^\alpha w_x)_x \mu \in L^2 (0,1)$. Concerning the term $2 x^\alpha \mu _x w_x$, we note that $\mu \in V^{(2,\infty)} _\alpha (0,1)$ implies that $\vert x^\alpha \mu _x w_x \vert \leq C x^{\alpha /2} \vert w_x \vert \in L^2 (0,1)$. Hence, $\mu w \in H^2 _\alpha (0,1)$. It remains to check the condition at $x=0$. We have that $x^{\alpha} (\mu w)_x = x^{\alpha} \mu _x w + x^{\alpha}w_x  \mu $. Since $\mu \in V^{(2,\infty)} _\alpha (0,1)$ and $w\in L^\infty (0,1)$, then $x^{\alpha} \mu _x w \to 0$ as $x\to 0$, and since $w \in V^{2,0} _\alpha (0,1)$ and $\mu \in L^\infty (0,1)$, then $x^{\alpha}w_x  \mu \to 0$ as $x\to 0$.

Thus $\mu w \in V^{(2,0)} _\alpha (0,1)$. We conclude that if $\mu \in V^{(2,\infty)} _\alpha (0,1)$ and $w\in C^0 ([0,T], D(A))$, then $\mu w \in C^0 ([0,T], V^{(2,0)} _\alpha (0,1))$.
\item $\alpha \in [1,2)$: we observe that $\mu \in V^{(2, \infty, \infty)} _\alpha (0,1)$ implies that $\vert \mu _x \vert \leq \frac{C}{x^{\alpha /2}}$, hence $\mu_ x \in L^1 (0,1)$. Therefore, $\mu \in L^\infty (0,1)$ and so $\mu w \in L^2 (0,1)$. Moreover, $x^{\alpha /2} (\mu w)_x = (x^{\alpha /2} \mu _x) w + (x^{\alpha /2} w_x) \mu\in L^2(0,1)$ because $x^{\alpha /2} \mu _x \in L^\infty (0,1)$ and $w \in L^2 (0,1)$. Furthermore, since $x^{\alpha /2} w_x \in L^2 (0,1)$ and $\mu \in L^\infty (0,1)$, we have $(x^{\alpha /2} w _x) \mu \in L^2 (0,1)$. Hence $x^{\alpha /2} (\mu w)_x \in L^2 (0,1)$. Now, consider $(x^\alpha (\mu w)_x)_x = (x^\alpha \mu _x)_x w + (x^\alpha w_x)_x \mu + 2 x^\alpha \mu _x w_x$. Since $(x^\alpha \mu _x)_x \in L^\infty (0,1)$ and $w \in L^2 (0,1)$, we have that $(x^\alpha \mu _x)_x w \in L^2 (0,1)$. Moreover, since $(x^\alpha w _x)_x \in L^2 (0,1)$ and $\mu \in L^\infty (0,1)$, it holds that $(x^\alpha w _x)_x \mu \in L^2 (0,1)$. Concerning the term $2 x^\alpha \mu _x w_x$, we note that $\mu \in V^{(2,\infty, \infty)} _\alpha (0,1)$ implies that $\vert x^\alpha \mu _x w_x \vert \leq C x^{\alpha /2} \vert w_x \vert \in L^2 (0,1)$. Therefore, $\mu w \in H^2 _\alpha (0,1)$. Finally, $\mu w \in H^2 _\alpha (0,1)$, with $\alpha \in [1,2)$, implies that $x^{\alpha} (\mu  w)_x \to 0$ as $x\to 0$. So, we have proved that $\mu w \in V^{(2,0)} _\alpha (0,1)$. And, if $\mu \in V^{(2,\infty,\infty)} _\alpha (0,1)$ and $w\in C^0 ([0,T], D(A))$, then $\mu w \in C^0 ([0,T], V^{(2,0)} _\alpha (0,1)$.
\end{itemize}
This concludes the proof of Lemma \ref{lem-reg-muw}.
\end{proof}


\subsubsection{Proof of Proposition \ref{prop-"thm3"}, part a)} \hfill
\label{sec-reg-H3-ok}

Using formula \eqref{def-gammanT}, Lemma \ref{lem-reg-muw} (with $w= w^{(p)}$) and Lemma \ref{lem-cond-h2}, we obtain that \eqref{eq-cond-suff} holds true and then \eqref{condition-H3H2} shows that $(w^{(p)}(T), w^{(p)} _t(T)) \in H^3 _{(0)} (0,1) \times D(A)$, which proves the validity of point a) of Proposition \ref{prop-"thm3"}. \qed


\subsection{Proof of Proposition \ref{prop-"thm3"}, part b)} \hfill

In Proposition \ref{prop-"thm3"}, part a), we have proved that $\Theta _T$ maps $L^2 (0,T)$ into $H^3 _{(0)}(0,1) \times D(A)$. Now we show that $\Theta _T$ is differentiable at every $p\in L^2 (0,T)$. Let $p_0, q \in L^2 (0,T)$. Then, consider $w^{(p_0)}$ solution of \eqref{CMU-eq-ctr-10} with $p=p_0$, and $w^{(p_0+q)}$ solution of \eqref{CMU-eq-ctr-10} with $p=p_0+q$.

Formally, let us write a limited development of $w^{(p_0+q)}$ with respect to $q$:
$$ w^{(p_0+q)} = w^{(p_0)} + W_1 (q) + \cdots $$
We use this development in \eqref{CMU-eq-ctr-10} to find the equation satisfied by the supposed first order term $W_1 (q)$:
denoting
$$ Pw:= w_{tt} - (x^\alpha w_x)_x ,$$
we have 
$$ P(w^{(p_0)} + W_1 (q) + \cdots) 
= \Bigl(p_0(t)+q(t) \Bigr) \, \mu (x) \, \Bigl( w^{(p_0)} + W_1 (q) + \cdots \Bigr),$$
hence, we deduce that $W_1(q)$ is solution of 
$$ P W_1 (q) = p_0 (t) \mu (x) W_1 (q) + q(t) \mu (x) w^{(p_0)} ,$$
which is the motivation to consider $W_1(q)$ as the solution of \eqref{CMU-eq-ctr-10-W} with $p=p_0$, that is,
$$ W_1 (q) = W ^{(p_0,q)} .$$
So, we introduce
\begin{equation}
\label{def-xi}
v^{(p_0,q)} := w^{(p_0+q)} - w^{(p_0)} - W^{(p_0,q)} ,
\end{equation}
which allows us to write
$$ \Theta_T (p_0+q) = \Theta_T (p_0) + (W ^{(p_0,q)} (T) , W ^{(p_0,q)} _t(T)) + (v ^{(p_0,q)} (T) , v ^{(p_0,q)} _t(T)) .$$
We are going to prove the following lemmas.

\begin{Lemma}
\label{lem-diff-ordre1}
The application
\begin{equation*}
\begin{split}
 L^2 (0,T) &\to H^3 _{(0)} (0,1) \times D(A)\\
q&\mapsto (W ^{(p_0,q)} (T) , W ^{(p_0,q)} _t(T)) 
\end{split}
\end{equation*}
is well-defined, linear and continuous.
\end{Lemma}

and

\begin{Lemma}
\label{lem-diff-ordre2}
The application
\begin{equation*}
\begin{split}
L^2 (0,T)&\to H^3 _{(0)} (0,1) \times D(A) \\
q &\mapsto (v ^{(p_0,q)} (T) , v ^{(p_0,q)} _t(T))
\end{split}
\end{equation*}
is well-defined, and satisfies
\begin{equation}
\label{def-xi-diff}
\frac{\Vert (v ^{(p_0,q)} (T) , v ^{(p_0,q)} _t(T)) \Vert _{H^3 _{(0)} (0,1) \times D(A)} }{\Vert q \Vert _{L^2 (0,T)}} \to 0, \quad \text{ as }\,\, \Vert q \Vert _{L^2 (0,T)} \to 0 .
\end{equation}
\end{Lemma}
Then, we conclude that $\Theta_T$ is differentiable at $p_0$ and that 
$$ D\Theta _T (p_0) \cdot q = (W ^{(p_0,q)} (T) , W ^{(p_0,q)} _t(T)) .$$


\subsubsection{Proof of Lemma \ref{lem-diff-ordre1}.} \hfill
\label{sec-diff-ordre1}

First, we prove that $(W ^{(p_0,q)} (T) , W ^{(p_0,q)} _t(T)) \in H^3 _{(0)} (0,1) \times D(A)$. We observe that problem \eqref{CMU-eq-ctr-10-W} is well-posed. Indeed,
from Proposition \ref{prop-"prop2"} we deduce that 
$w^{(p_0)} \in C^0 ([0,T], D(A))$, and by applying Lemma \ref{lem-reg-muw}
we obtain that $\mu w^{(p_0)} \in C^0 ([0,T], V^{(2,0)} _\alpha (0,1))$ because $\mu \in V^2 _\alpha(0,1)$. Therefore, $\mu w^{(p_0)} \in C^0 ([0,T], H^1 _\alpha (0,1))$. Furthermore, we have that $q \mu w^{(p_0)} \in L^2 (0,T; H^1 _\alpha (0,1))$. Thus, we can apply Proposition \ref{prop-"prop2"} to \eqref{CMU-eq-ctr-10-W} (taking $f= q \mu w^{(p_0)}$), obtaining that
$$ (W ^{(p_0,q)},W ^{(p_0,q)}_t) \in C^0 ([0,T], D(A) \times H^1 _\alpha (0,1)).$$
Moreover, \eqref{estimW} gives that
\begin{equation}
\begin{split}
\label{estimWpq}
\Vert W ^{(p_0,q)} \Vert _{C^0([0,T],D(A))} &+ \Vert W_t ^{(p_0,q)} \Vert _{C^0([0,T],H^1 _\alpha (0,1))}
\\
&\qquad\quad\leq C(T) \Vert q \mu w^{(p_0)} \Vert _{L^2(0,T;H^1 _\alpha (0,1)))}
\\
&\qquad\quad\leq C(T) \Vert q \Vert _{L^2 (0,T)} \Vert \mu w^{(p_0)} \Vert _{C^0([0,T],H^1 _\alpha (0,1)))}.
\end{split}
\end{equation} 
We can now decompose $(W ^{(p_0,q)}(T), W ^{(p_0,q)} _t (T))$ as follows:
denoting
\begin{equation}
\label{def-RnTW}
R_n (s) = \langle p_0(s) \mu (\cdot) W^{(p_0,q)} (\cdot ,s) + q(s) \mu (\cdot) w^{(p_0)} (\cdot, s), \Phi _{\alpha,\vert n \vert} \rangle _{L^2 (0,1)},
\end{equation}
we have 
\begin{equation}
\label{formule-exp2W}
\left( \begin{array}{c} W^{(p_0,q)}(x,T)  \\ W^{(p_0,q)} _t (x,T) \end{array} \right)
= \Gamma ^{(p_0,q)} _0 (T)
\\
+ \sum _{n\in \mathbb Z^*} \frac{1}{2i \omega _{\alpha,n}}  \, \gamma ^{(p_0,q)} _n (T) \, \Psi _{\alpha,n} (x) \, e^{i\omega_{\alpha,n} T} ,
\end{equation}
with 
\begin{equation}
\label{def-gamma0TW}
\Gamma ^{(p_0,q)} _0 (T) = \left( \begin{array}{c}  \int _0 ^T R_0 (s) (T-s) \, ds   \\ \int _0 ^T R_0 (s) \, ds  \end{array} \right) = \left( \begin{array}{c} \gamma ^{(p_0,q)} _{00} (T)   \\ \gamma ^{(p_0,q)} _{01} (T)  \end{array} \right) ,
\end{equation}
and
\begin{equation}
\label{def-gammanTW}
\forall\, n \in \mathbb Z ^*, \quad 
\gamma ^{(p_0,q)} _n (T) = \int _0 ^T R_n (s)\, e^{-i \omega _{\alpha,n} s} \, ds .
\end{equation}
Moreover, the following implication holds true
\begin{multline}
\label{condition-H3H2pq}
\sum _{n \in \mathbb Z} \lambda _{\alpha , \vert n \vert } ^2 \, \vert \gamma ^{(p_0,q)} _n (T) \vert  ^2  < \infty
\\
\quad \implies \quad (W ^{(p_0,q)} (T), W ^{(p_0,q)} _t(T)) \in H^3 _{(0)} (0,1) \times D(A) .
\end{multline}
Therefore, to show that $(W^{(p_0,q)}(T),W^{(p_0,q)}_t(T))\in H^3_{(0)}(0,1)\times D(A)$, we have to prove the convergence of the above series. We decompose as follows: $\forall\, n \neq 0$
\begin{equation}
\label{form-gamman+pq}
\begin{split}
\gamma ^{(p_0,q)} _n (T) 
&= \int _0 ^T p_0(s) \, \langle \mu (\cdot) W^{(p_0,q)} (\cdot,s), \Phi _{\alpha,\vert n \vert} \rangle _{L^2 (0,1)} \, e^{-i\omega _{\alpha,n} s} \, ds 
\\
&\quad+ \int _0 ^T q(s) \, \langle \mu (\cdot) w^{(p_0)} (\cdot,s), \Phi _{\alpha,\vert n \vert} \rangle _{L^2 (0,1)} \, e^{-i\omega _{\alpha,n} s} \, ds 
\\
&=: \gamma _n ^{(p_0, \mu, W^{(p_0,q)}) } (T) + \gamma _n ^{(q, \mu, w^{(p_0)}) }  (T) .
\end{split}
\end{equation}
We apply Lemma \ref{lem-cond-h2} first choosing $p=p_0$ and $g= \mu  W^{(p_0,q)}$. Since $p_0\in L^2(0,T)$ and $\mu W^{(p_0,q)}\in C^0([0,T],V^{(2,0)}_\alpha(0,1)$  we obtain
\begin{equation}
\label{prop-Sn-1pq}
\sum _{n \in \mathbb Z^*} \lambda _{\alpha ,\vert n \vert } ^2 \, \vert \gamma _n ^{(p_0, \mu, W^{(p_0,q)}) } (T) \vert  ^2  < \infty ,
\end{equation}
and furthermore
\begin{equation}
\begin{split}
\label{prop-Sn-2pq}
&\Bigl( \sum _{n \in \mathbb Z^*} \lambda _{\alpha ,\vert n \vert } ^2 \, \vert \gamma _n ^{(p_0, \mu, W^{(p_0,q)}) } (T) \vert  ^2  \Bigr) ^{1/2}\\
&\qquad\qquad\quad  \leq C(\alpha,T) \, \Vert p_0 \Vert _{L^2(0,T)}  \,  \Vert \mu  W^{(p_0,q)} \Vert _{C^0 ([0,T], V^{(2,0)} _\alpha (0,1))}  
\\
&\qquad\qquad\quad\leq C'(\alpha,T) \, \Vert p_0 \Vert _{L^2(0,T)}  \,  \Vert \mu  \Vert _{V^2 _\alpha} \, \Vert W^{(p_0,q)} \Vert _{C^0 ([0,T], D(A))} 
\\
&\qquad\qquad\quad\leq C''(\alpha,T) \, \Vert p_0 \Vert _{L^2(0,T)}  \,  \Vert \mu  \Vert _{V^2 _\alpha} \, \Vert q \Vert _{L^2 (0,T)} \, \Vert \mu w^{(p_0)} \Vert _{C^0([0,T],H^1 _\alpha (0,1)))}
\\
&\qquad\qquad\quad\leq C'''(\alpha,T) \, \Vert p_0 \Vert _{L^2(0,T)}  \, \Vert q \Vert _{L^2 (0,T)} \,  \Vert \mu  \Vert _{V^2 _\alpha} ^2  \, \Vert  w^{(p_0)} \Vert _{C^0([0,T],D(A))} ,
\end{split}
\end{equation}
In the same way, we apply Lemma \ref{lem-cond-h2} with $p=q$ and $g= \mu  w^{(p_0)}$ and we get
\begin{equation}
\label{prop-Sn-1pq2}
\sum _{n \in \mathbb Z^*} \lambda _{\alpha ,\vert n \vert } ^2 \, \vert \gamma _n ^{(q, \mu, w^{(p_0)}) } (T) \vert  ^2  < \infty.
\end{equation}
and, moreover,
\begin{equation}
\begin{split}
\label{prop-Sn-2pq2}
&\Bigl( \sum _{n \in \mathbb Z^*} \lambda _{\alpha ,\vert n \vert } ^2 \, \vert \gamma _n ^{(q, \mu, w^{(p_0)}) } (T) \vert  ^2  \Bigr) ^{1/2}
\\
&\qquad\qquad\qquad\leq C(\alpha,T) \, \Vert q \Vert _{L^2(0,T)}  \,  \Vert \mu  w^{(p_0)} \Vert _{C^0 ([0,T], 
V^{(2,0)} _\alpha (0,1))}  
\\
&\qquad\qquad\qquad\leq C'(\alpha,T) \, \Vert q \Vert _{L^2(0,T)}  \,  \Vert \mu  \Vert _{V^2 _\alpha} \, \Vert w^{(p_0)} \Vert _{C^0 ([0,T], D(A))} 
.
\end{split}
\end{equation}
We have proved that
$$ \sum _{n \in \mathbb Z^*} \lambda _{\alpha , \vert n \vert } ^2 \, \vert \gamma ^{(p_0,q)} _n (T) \vert  ^2  < \infty ,$$
so, from \eqref{condition-H3H2pq} we have that 
$$ (W ^{(p_0,q)} (T), W ^{(p_0,q)} _t(T)) \in H^3 _{(0)} (0,1) \times D(A) ,$$
and furthermore
\begin{equation*}
\begin{split}
&\left \Vert \left( \begin{array}{c} W^{(p_0,q)}(\cdot ,T)  \\ W^{(p_0,q)} _t (\cdot ,T) \end{array} \right)
- \Gamma ^{(p_0,q)} _0 (T) \right \Vert _{H^3 _{(0)} (0,1) \times D(A)}
\\
&\qquad\qquad\quad\leq C \Bigl( \sum _{n \in \mathbb Z ^*} \lambda _{\alpha, \vert n \vert} ^2 \vert \gamma ^{(p_0,q)} _n (T) \vert ^2 \Bigr) ^{1/2}
\\
&\qquad\qquad\quad\leq C'''(\alpha,T) \, \Vert p_0 \Vert _{L^2(0,T)}  \, \Vert q \Vert _{L^2 (0,T)} \,  \Vert \mu  \Vert _{V^2 _\alpha(0,1)} ^2  \, \Vert  w^{(p_0)} \Vert _{C^0([0,T],D(A))} 
\\
&\qquad\qquad\quad\quad+ C'(\alpha,T) \, \Vert q \Vert _{L^2(0,T)}  \,  \Vert \mu  \Vert _{V^2 _\alpha(0,1)} \, \Vert w^{(p_0)} \Vert _{C^0 ([0,T], D(A))} 
.
\end{split}
\end{equation*}
However, $\Gamma ^{(p_0,q)} _0 (T)$ is independent of $x$, hence
$$\left \Vert \Gamma ^{(p_0,q)} _0 (T) \right\Vert _{H^3 _{(0)} (0,1) \times D(A)}
= \left\Vert \Gamma ^{(p_0,q)} _0 (T) \right\Vert _{L^2 (0,1) \times L^2 (0,1)} \leq C \int _0 ^T \left\vert R_0 (s) \right\vert \, ds  ,$$
and 
\begin{equation*}
\begin{split}
\vert R_0 (s) \vert 
&\leq \vert p_0 (s) \vert \, \vert \langle \mu W ^{(p_0,q)} (s), \Phi _{\alpha,0} \rangle _{L^2 (0,1)} \vert + \vert q (s) \vert \, \vert \langle \mu w ^{(p_0)} (s), \Phi _{\alpha,0} \rangle _{L^2 (0,1)} \vert
\\
&\leq \vert p_0 (s) \vert \Vert \mu W ^{(p_0,q)} (s) \Vert _{L^2 (0,1)} + \vert q (s) \vert \, \Vert \mu w ^{(p_0)} (s) \Vert _{L^2 (0,1)}
\\
&\leq C \vert p_0 (s) \vert \, \Vert \mu \Vert _{V^2 _\alpha(0,1)} \Vert W ^{(p_0,q)}  \Vert _{C^0 ([0,T], D(A))}\\
&\qquad\qquad\qquad\qquad\qquad\qquad
+ C \vert q (s) \vert \, \Vert \mu \Vert _{V^2 _\alpha(0,1)} \Vert w ^{(p_0)}  \Vert _{C^0 ([0,T], D(A))}
\\
&\leq C' \vert p_0 (s) \vert \, \Vert \mu \Vert _{V^2 _\alpha(0,1)} ^2 \, \Vert q \Vert _{L^2 (0,T)} \Vert w ^{(p_0)}  \Vert _{C^0 ([0,T], D(A))}\\
&\qquad\qquad\qquad\qquad\qquad\qquad+ C \vert q (s) \vert \, \Vert \mu \Vert _{V^2 _\alpha(0,1)} \Vert w ^{(p_0)}  \Vert _{C^0 ([0,T], D(A))}.
\end{split}
\end{equation*}
Thus
$$ \Vert R_0 \Vert _{L^1 (0,T)} \leq C(\alpha, T, \mu, w^{(p_0)}) \Vert q \Vert _{L^2 (0,T)} . $$
and therefore
$$ \left \Vert \left( \begin{array}{c} W^{(p_0,q)}(\cdot ,T)  \\ W^{(p_0,q)} _t (\cdot ,T) \end{array} \right) \right \Vert _{H^3 _{(0)} (0,1) \times D(A)}
\leq C(\alpha, T, \mu, w^{(p_0)}) \Vert q \Vert _{L^2 (0,T)} .$$
Hence, we have proved that the application $q\mapsto (W^{(p_0,q})(T),W^{(p_0,q)}_t(T))$ is continuous. This concludes the proof of Lemma \ref{lem-diff-ordre1}. \qed


\subsubsection{Proof of Lemma \ref{lem-diff-ordre2}.} \hfill
\label{sec-diff-ordre2}

Function $v^{(p_0,q)}$, defined in \eqref{def-xi}, is the classical solution of
\begin{equation}
\label{CMU-eq-ctr-10xi}
\begin{cases}
v^{(p_0,q)} _{tt} - (x^\alpha v^{(p_0,q)} _x)_x = p_0(t) \mu (x) v^{(p_0,q)} + q(t) \mu (x) (w ^{(p_0+q)} - w ^{(p_0)})  , 
\\
(x^\alpha v^{(p_0,q)} _x)(x=0,t)=0 , 
\\
v^{(p_0,q)} _x (x=1,t)=0 , 
\\
v^{(p_0,q)} (x,0)=0 ,  
\\
v^{(p_0,q)} _t (x,0)=0  
,
\end{cases}
\end{equation}
that is actually a problem similar to \eqref{CMU-eq-ctr-10-W} with $p=p_0$ and  $w ^{(p_0+q)} - w ^{(p_0)}$ that replaces $w^{p}$. Then, the linear control system \eqref{CMU-eq-ctr-10xi} is well-posed, and 
$$ (v ^{(p_0,q)}, v ^{(p_0,q)}_t) \in C^0 ([0,T], D(A) \times H^1 _\alpha (0,1)).$$
So, \eqref{estimW} gives that
\begin{equation}
\label{estimWpqvv}
\begin{split}
\Vert v ^{(p_0,q)} \Vert _{C^0([0,T],D(A))} &+ \Vert v_t ^{(p_0,q)} \Vert _{C^0([0,T],H^1 _\alpha (0,1))}
\\
&\leq C \Vert q \mu (w^{(p_0+q)}-w^{(p_0)}) \Vert _{L^2(0,T;H^1 _\alpha (0,1)))}
\\
&\leq C \Vert q \Vert _{L^2 (0,T)} \Vert \mu (w^{(p_0+q)}-w^{(p_0)}) \Vert _{C^0([0,T],H^1 _\alpha (0,1)))}
\\
&\leq C \Vert q \Vert _{L^2 (0,T)} \Vert \mu  \Vert _{V^2 _\alpha(0,1)} \Vert (w^{(p_0+q)}-w^{(p_0)}) \Vert _{C^0([0,T],H^1 _\alpha (0,1)))};
\end{split}
\end{equation} 
We can decompose $(v ^{(p_0,q)}(T), v ^{(p_0,q)} _t (T))$ as follows:
denoting
\begin{equation}
\label{def-RnTWv}
z_n (s) = \langle p_0(s) \mu (\cdot) v^{(p_0,q)} (\cdot ,s) + q(s) \mu (\cdot) (w^{(p_0+q)}-w^{(p_0)}) (\cdot, s), \Phi _{\alpha,\vert n \vert} \rangle _{L^2 (0,1)},
\end{equation}
we have 
\begin{equation}
\label{formule-exp2Wvvv}
\left( \begin{array}{c} v^{(p_0,q)}(x,T)  \\ v^{(p_0,q)} _t (x,T) \end{array} \right)
=  E ^{(p_0,q)} _0 (T)
+ \sum _{n\in \mathbb Z^*} \frac{1}{2i \omega _{\alpha,n}}  \, \varepsilon ^{(p_0,q)} _n (T) \, \Psi _{\alpha,n} (x) \, e^{i\omega_{\alpha,n} T} ,
\end{equation}
with 
\begin{equation}
\label{def-gamma0TWvvv}
E ^{(p_0,q)} _0 (T) = \left( \begin{array}{c}  \int _0 ^T  z _0 (s) (T-s) \, ds   \\ \int _0 ^T  z _0 (s) \, ds  \end{array} \right) ,
\end{equation}
and
\begin{equation}
\label{def-gammanTWvvv}
\forall\, n \in \mathbb Z ^*, \quad 
\varepsilon ^{(p_0,q)} _n (T) = \int _0 ^T z_n (s)\, e^{-i \omega _{\alpha,n} s} \, ds .
\end{equation}
As showed in section \ref{sec-suff-cond},
\begin{multline}
\label{condition-H3H2pqvvv}
\sum _{n \in \mathbb Z} \lambda _{\alpha , \vert n \vert } ^2 \, \vert \varepsilon ^{(p_0,q)} _n (T) \vert  ^2  < \infty
\,\, \implies \,\, (v ^{(p_0,q)} (T), v ^{(p_0,q)} _t(T)) \in H^3 _{(0)} (0,1) \times D(A) ,
\end{multline}
Thus, to ensure that $(v^{(p_0,q)}(T),v^{(p_0,q)}_t(T)\in H^3_{(0)}(0,1)\times D(A)$, we have to prove the convergence of the above series. We observe that
\begin{multline*}
p_0(t) \mu (x) v^{(p_0,q)} + q(t) \mu (x) (w ^{(p_0+q)} - w ^{(p_0)}) 
\\
= (p_0(t)+q(t)) \mu (x) v^{(p_0,q)} + q(t) \mu (x) W ^{(p_0,q)},
\end{multline*}
therefore, we decompose $ \varepsilon ^{(p_0,q)} _n (T)$ as follows: $\forall\, n \neq 0$
\begin{equation}
\begin{split}
\label{form-gamman+pqvv}
\varepsilon ^{(p_0,q)} _n (T)
&= \int _0 ^T (p_0(s)+q(s)) \, \langle \mu (\cdot) v^{(p_0,q)} (\cdot,s), \Phi _{\alpha,\vert n \vert} \rangle _{L^2 (0,1)} \, e^{-i\omega _{\alpha,n} s} \, ds 
\\
&\quad+ \int _0 ^T q(s) \, \langle \mu (\cdot) W^{(p_0,q)} (\cdot,s), \Phi _{\alpha,\vert n \vert} \rangle _{L^2 (0,1)} \, e^{-i\omega _{\alpha,n} s} \, ds 
\\
&= \gamma _n ^{(p_0+q, \mu, v^{(p_0,q)}) } (T) + \gamma _n ^{(q, \mu, W^{(p_0,q)}) }  (T).
\end{split}
\end{equation}
Applying twice Lemma \ref{lem-cond-h2}, we obtain
\begin{equation}
\begin{split}
\label{prop-Sn-2pqvv}
&\left( \sum _{n =1} ^\infty \lambda _{\alpha ,n} ^2 \, \vert \gamma _n ^{(p_0+q, \mu, v^{(p_0,q)}) } (T) \vert  ^2  \right) ^{1/2}
\\
&\quad\quad\leq C(\alpha,T) \, \Vert p_0 +q \Vert _{L^2(0,T)}  \,  \Vert \mu  v^{(p_0,q)} \Vert _{C^0 ([0,T], 
V^{(2,0)} _\alpha (0,1))}  
\\
&\quad\quad\leq C'(\alpha,T) \, \Vert p_0 +q \Vert _{L^2(0,T)}  \,  \Vert \mu  \Vert _{V^2 _\alpha(0,1)} \, \Vert v^{(p_0,q)} \Vert _{C^0 ([0,T], D(A))} 
\\
&\quad\quad\leq C''(\alpha,T) \,  \Vert p_0 + q \Vert _{L^2(0,T)}   \, \Vert q \Vert _{L^2 (0,T)} \,  \Vert \mu  \Vert _{V^2 _\alpha(0,1)} ^2  \, \Vert  w^{(p_0+q)}-w^{(p_0)} \Vert _{C^0([0,T],D(A))} ,
\end{split}
\end{equation}
and 
\begin{equation}
\begin{split}
\label{prop-Sn-2pq2vv2}
\left( \sum _{n =1} ^\infty \lambda _{\alpha ,n} ^2 \, \vert \gamma _n ^{(q, \mu, W^{(p_0,q)}) } (T) \vert  ^2  \right) ^{1/2}&\leq C(\alpha,T) \, \Vert q \Vert _{L^2(0,T)}  \,  \Vert \mu  W^{(p_0,q)} \Vert _{C^0 ([0,T], V^{(2,0)} _\alpha (0,1))}  
\\
\leq& C'(\alpha,T) \, \Vert q \Vert _{L^2(0,T)}  \,  \Vert \mu  \Vert _{V^2 _\alpha(0,1)} \, \Vert W^{(p_0,q)} \Vert _{C^0 ([0,T], D(A))} 
\\
\leq& C''(\alpha,T) \, \Vert q \Vert _{L^2(0,T)} ^2  \,  \Vert \mu  \Vert _{V^2 _\alpha(0,1)} ^2 \, \Vert w^{(p_0)} \Vert _{C^0 ([0,T], D(A))}.
\end{split}
\end{equation}
Thus, we have proved that
$$ \sum _{n \in \mathbb Z^*} \lambda _{\alpha , \vert n \vert } ^2 \, \vert \varepsilon ^{(p_0,q)} _n (T) \vert  ^2  < \infty ,$$
which implies that 
$$ (v ^{(p_0,q)} (T), v ^{(p_0,q)} _t(T)) \in H^3 _{(0)} (0,1) \times D(A) .$$
Furthermore,
\begin{equation*}
\begin{split}
&\left \Vert \left( \begin{array}{c} v^{(p_0,q)}(\cdot ,T)  \\ v^{(p_0,q)} _t (\cdot ,T) \end{array} \right)
- E ^{(p_0,q)} _0 (T) \right \Vert _{H^3 _{(0)} (0,1) \times D(A)}
\\
&\quad\quad\leq C \Bigl( \sum _{n \in \mathbb Z} \lambda _{\alpha, \vert n \vert} ^2 \vert \varepsilon ^{(p_0,q)} _n (T) \vert ^2 \Bigr) ^{1/2}
\\
&\quad\quad\leq C''(\alpha,T) \,  \Vert p_0 + q \Vert _{L^2(0,T)}   \, \Vert q \Vert _{L^2 (0,T)} \,  \Vert \mu  \Vert _{V^2 _\alpha(0,1)} ^2  \, \Vert  w^{(p_0+q)}-w^{(p_0)} \Vert _{C^0([0,T],D(A))} 
\\
&\quad\quad\quad+ C''(\alpha,T) \, \Vert q \Vert _{L^2(0,T)} ^2  \,  \Vert \mu  \Vert _{V^2 _\alpha(0,1)} ^2 \, \Vert w^{(p_0)} \Vert _{C^0 ([0,T], D(A))}.
\end{split}
\end{equation*}
To conclude, we observe that 
\begin{equation}
\label{def-upq}
u^{(p_0,q)} := w^{(p_0+q)} - w^{(p_0)}
\end{equation}
is solution of 
\begin{equation}
\label{CMU-eq-ctr-10u}
\begin{cases}
u^{(p_0,q)} _{tt} - (x^\alpha u^{(p_0,q)} _x)_x = p_0(t) \mu (x) u^{(p_0,q)} + q(t) \mu (x) w^{(p_0+q)} , 
\\
(x^\alpha u^{(p_0,q)} _x)(x=0,t)=0 , 
\\
u^{(p_0,q)} _x (x=1,t)=0 , 
\\
u^{(p_0,q)} (x,0)=0 ,  
\\
u^{(p_0,q)} _t (x,0)=0  
,
\end{cases}
\end{equation}
hence, \eqref{estimW} implies that
\begin{equation*}
\begin{split}
\Vert w^{(p_0+q)} - w^{(p_0)} \Vert _{C^0([0,T],D(A))} 
&= \Vert u^{(p_0,q)} \Vert _{C^0([0,T],D(A))}
\\
&\leq C \Vert q \mu w^{(p_0+q)} \Vert _{L^2(0,T;H^1 _\alpha (0,1)))}
\\
&\leq C \Vert q \Vert _{L^2 (0,T)} \Vert \mu  \Vert _{V^2 _\alpha} \Vert w^{(p_0+q)} \Vert _{C^0([0,T],D(A))}.
\end{split}
\end{equation*} 
So, we get
\begin{equation*}
\begin{split}
&\left \Vert \left( \begin{array}{c} v^{(p_0,q)}(\cdot ,T)  \\ v^{(p_0,q)} _t (\cdot ,T) \end{array} \right)
- E ^{(p_0,q)} _0 (T) \right \Vert _{H^3 _{(0)} (0,1) \times D(A)}
\\
&\qquad\quad\leq C''(\alpha,T) \, \Vert q \Vert _{L^2 (0,T)} ^2  \Bigl(  \Vert p_0 + q \Vert _{L^2(0,T)}  \,  \Vert \mu  \Vert _{V^2 _\alpha(0,1)} ^3  \, \Vert  w^{(p_0+q)} \Vert _{C^0([0,T],D(A))} 
\\
&\qquad\quad\quad+  \,  \Vert \mu  \Vert _{V^2 _\alpha(0,1)} ^2 \, \Vert w^{(p_0)} \Vert _{C^0 ([0,T], D(A))} \Bigr) .
\end{split}
\end{equation*}
However, $E ^{(p_0,q)} _0 (T)$ is independent of $x$, therefore, we deduce that
$$ \Vert E ^{(p_0,q)} _0 (T) \Vert _{H^3 _{(0)} (0,1) \times D(A)}
= \Vert E ^{(p_0,q)} _0 (T) \Vert _{L^2 (0,1) \times L^2 (0,1)} \leq C \int _0 ^T \vert z_0 (s) \vert \, ds  ,$$
and 
\begin{equation*}
\begin{split}
\vert z_0 (s) \vert 
&\leq \vert p_0 (s) \vert \, \vert \langle \mu v ^{(p_0,q)} (s), \Phi _{\alpha,0} \rangle _{L^2 (0,1)} \vert + \vert q (s) \vert \, \vert \langle \mu (w ^{(p_0+q)}-w ^{(p_0)}) (s), \Phi _{\alpha,0} \rangle _{L^2 (0,1)} \vert
\\
&\leq \vert p_0 (s) \vert \Vert \mu v ^{(p_0,q)} (s) \Vert _{L^2 (0,1)} + \vert q (s) \vert \, \Vert \mu (w ^{(p_0+q)}-w ^{(p_0)}) (s) \Vert _{L^2 (0,1)}
\\
&\leq C \vert p_0 (s) \vert \, \Vert \mu \Vert _{V^2 _\alpha(0,1)} \Vert v ^{(p_0,q)}  \Vert _{C^0 ([0,T], D(A))}
\\
&\quad+ C \vert q (s) \vert \, \Vert \mu \Vert _{V^2 _\alpha(0,1)} \Vert w ^{(p_0+q)}-w ^{(p_0)}  \Vert _{C^0 ([0,T], D(A))}
\\
&\leq C' \vert p_0 (s) \vert \, \Vert \mu \Vert ^2 _{V^2 _\alpha(0,1)} \Vert q \Vert _{L^2 (0,T)}  \Vert (w^{(p_0+q)}-w^{(p_0)}) \Vert _{C^0([0,T],H^1 _\alpha (0,1)))}
\\
&\quad+ C \vert q (s) \vert \, \Vert \mu \Vert _{V^2 _\alpha(0,1)} ^2 \Vert q \Vert _{L^2 (0,T)} \Vert w^{(p_0+q)} \Vert _{C^0([0,T],D(A))} 
\\
&\leq C' \vert p_0 (s) \vert \, \Vert \mu \Vert ^3 _{V^2 _\alpha(0,1)} \Vert q \Vert ^2 _{L^2 (0,T)}  
 \Vert w^{(p_0+q)} \Vert _{C^0([0,T],D(A))}
 \\
&\quad+ C \vert q (s) \vert \, \Vert \mu \Vert _{V^2 _\alpha(0,1)} ^2 \Vert q \Vert _{L^2 (0,T)} \Vert w^{(p_0+q)} \Vert _{C^0([0,T],D(A))}.
\end{split}
\end{equation*}
Hence, we have showed that
$$ \Vert z_0 \Vert _{L^1 (0,T)} \leq C(\alpha, T, \mu, w^{(p_0)}) \Vert q \Vert ^2 _{L^2 (0,T)} , $$
and so,
$$ \left \Vert \left( \begin{array}{c} v^{(p_0,q)}(\cdot ,T)  \\ v^{(p_0,q)} _t (\cdot ,T) \end{array} \right) \right \Vert _{H^3 _{(0)} (0,1) \times D(A)}
\leq C(\alpha, T, \mu, w^{(p_0)}) \Vert q \Vert ^2 _{L^2 (0,T)} .$$
Thus, \eqref{def-xi-diff} is satisfied and this concludes the proof of Lemma \ref{lem-diff-ordre2}. \qed


\subsection{Proof of Proposition \ref{prop-"thm3"}, part c)} \hfill

To prove that $\Theta _T$ is of class $C^1$, we have to prove that
the application $D\Theta _T$ is continuous from $L^2 (0,T)$ into $\mathcal{L}_c (L^2 (0,T), H^3 _{(0)} \times D(A))$, namely
$$ \vert \vert \vert D\Theta _T (p_0 + \tilde p)  - D\Theta _T (p_0) \vert \vert \vert _{\mathcal{L}_c (L^2 (0,T), H^3 _{(0)} \times D(A))} \to 0, \quad \text{ as } \Vert \tilde p \Vert _{L^2 (0,T)} \to 0 .$$
Proceeding as in section \ref{sec-diff-ordre1}, it is easy to verify that there exists $C(\alpha, T, \mu, p_0) >0$ such that for any $\tilde p,\, q \in L^2 (0,T)$
\begin{equation}
\begin{split}
\label{reader}
\Vert D\Theta _T (p_0 +\tilde p)\cdot q - D\Theta _T (p_0)\cdot q \Vert _{H^3 _{(0)} \times D(A)} &= \Vert W^{(p_0+\tilde p, q)} (T) - W^{(p_0, q)} (T) \Vert _{H^3 _{(0)} \times D(A)} 
\\
&\leq C(\alpha, T, \mu, p_0) \, \Vert \tilde p \Vert _{L^2 (0,T)} \, \Vert q \Vert _{L^2 (0,T)} ,
\end{split}
\end{equation}
which implies that
$$ \forall \,\tilde p, \quad \vert \vert \vert D\Theta _T (p_0+\tilde p) - D\Theta _T (p_0)  \vert \vert \vert _{\mathcal{L}_c (L^2 (0,T), H^3 _{(0)} \times D(A))}  \leq C(\alpha, T, \mu, p_0) \, \Vert \tilde p \Vert _{L^2 (0,T)}.$$
Thus, $D\Theta _T$ is continuous and this concludes the proof of Proposition \ref{prop-"thm3"}, part c). 
We leave the proof of \eqref{reader} to the reader. \qed


\section{Reachability for $T>T_0$: proof of Theorem \ref{thm-ctr}}\hfill
\label{sec-pr-thm}


The proof follows from the classical inverse mapping theorem applied to the function $\Theta _T: L^2 (0,T) \to H^3 _{(0)} \times D(A)$ at the point $p_0=0$.
First, we recall that $\Theta _T (p_0=0) = (1,0)$. In the following we study $D\Theta _T (0)$. 

\subsection{Surjectivity of $D\Theta _T (0)$: study of the associated moment problem} \hfill
\label{sec-surj-cas1}
The key point of the proof of Theorem \ref{thm-ctr} is the following result.
\begin{Lemma}
\label{lem-DThetaT(0)inv}
Assume that \eqref{hyp-T0} and \eqref{hyp-mu} are satisfied.
Then, the linear application 
\begin{equation*}
\begin{split}
D\Theta _T (0): L^2 (0,T) &\to H^3 _{(0)} \times D(A)\\
q&\mapsto (W^{(0,q)} (T), W^{(0,q)} _t (T))
\end{split}
\end{equation*}
is surjective, and 
$$ D\Theta _T (0): \overline{\text{Vect } \{1,t, \cos \sqrt{\lambda _{\alpha,n}} t, \sin \sqrt{\lambda _{\alpha,n}} t, n\geq 1 \}} \to H^3 _{(0)} \times D(A)$$ is invertible.
\end{Lemma}

\begin{proof}[Proof of Lemma \ref{lem-DThetaT(0)inv}] Since $w^{(0)}=1$, \eqref{CMU-eq-ctr-10-W} implies that $W^{(0,q)}$ is solution of the following linear problem

\begin{equation}
\label{CMU-eq-ctr-10-W-0}
\begin{cases}
W^{(0,q)}_{tt} - (x^\alpha W^{(0,q)} _x)_x = q(t) \mu (x)  , \quad &x\in (0,1), t \in (0,T), 
\\
(x^\alpha W^{(0,q)}_x)(x=0,t)=0 , \quad &t \in (0,T) ,\\
W^{(0,q)}_x (x=1,t)=0 , \quad &t \in (0,T) , \\
W^{(0,q)}(x,0)=0 ,  \quad &x\in (0,1), 
\\
W^{(0,q)} _t (x,0)=0 , \quad &x\in (0,1) .
\end{cases}
\end{equation}
Following the procedure of section \ref{sec-form-gamman}, we introduce
\begin{equation}
\label{def-RnTW-0}
r_n (s) = \langle q(s) \mu , \Phi _{\alpha,n} \rangle _{L^2 (0,1)} = \mu_{\alpha,n} q(s) 
\quad \text{with} \quad \mu_{\alpha,n}= \langle  \mu , \Phi _{\alpha,n} \rangle _{L^2 (0,1)},
\end{equation}
and we have can express the solution of \eqref{CMU-eq-ctr-10-W-0} at time $T$, $(W^{(0,q)}(T),W^{(0,q)}_t(T))$ as follows
\begin{multline}
\label{form-pour-diff}
W^{(0,q)}(x,T) = \int _0 ^T r_0 (s) (T-s) \, ds 
\\
+ \sum _{n=1} ^\infty \Bigl( \int _0 ^T r_n (s) \frac{\sin \sqrt{\lambda _{\alpha,n}} (T-s)}{\sqrt{\lambda _{\alpha,n}}} \, ds  \Bigr) \Phi _{\alpha,n} (x) , 
\end{multline}
and
\begin{multline}
\label{form-pour-diff'}
W^{(0,q)} _t (x,T) = \int _0 ^T r_0 (s) \, ds
\\
+ \sum _{n=1} ^\infty  \Bigl( \sqrt{\lambda _{\alpha,n}} \int _0 ^T r_n (s) \frac{\cos \sqrt{\lambda _{\alpha,n}} (T-s)}{\sqrt{\lambda _{\alpha,n}}} \, ds \Bigr) \Phi _{\alpha,n} (x) .
\end{multline}
To prove Lemma \ref{lem-DThetaT(0)inv}, we choose any pair $(Y ^f, Z ^f) \in H^3 _{(0)} \times D(A)$, and we want to show that there exists $q\in L^2 (0,T)$ such that
\begin{equation}
\label{target}
 (W^{(0,q)}(T) , W^{(0,q)} _t (T) ) = (Y ^f, Z ^f).
 \end{equation}
Introducing the Fourier coefficients of the target state
$$ Y_{\alpha,n} ^f = \langle Y^f , \Phi _{\alpha,n} \rangle _{L^2 (0,1)} \quad \text{ and } \quad 
Z_{\alpha,n} ^f = \langle Z^f , \Phi _{\alpha,n} \rangle _{L^2 (0,1)} , $$
we can decompose $(Y ^f, Z ^f)$ as follows
$$ Y^f (x) = \sum _{n=0} ^\infty Y_{\alpha,n} ^f \Phi _{\alpha,n} (x) =  Y_{\alpha,0} ^f + \sum _{n=1} ^\infty Y_{\alpha,n} ^f \Phi _{\alpha,n} (x) $$
and
$$ Z^f (x) = \sum _{n=0} ^\infty Z_{\alpha,n} ^f \Phi _{\alpha,n} (x) = Z_{\alpha,0} ^f + \sum _{n=1} ^\infty Z_{\alpha,n} ^f \Phi _{\alpha,n} (x).$$
We derive from \eqref{form-pour-diff} and \eqref{form-pour-diff'} that \eqref{target} is satisfied if and only if
\begin{equation}
\label{syst-moments}
\begin{cases}
\int _0 ^T r_0 (s) \, ds = Z_{\alpha,0} ^f , \\
\sqrt{\lambda _{\alpha,n}} \int _0 ^T r_n (s) \frac{\cos \sqrt{\lambda _{\alpha,n}} (T-s)}{\sqrt{\lambda _{\alpha,n}}} \, ds =
 Z_{\alpha,n} ^f, & \text{ for all } n \geq 1,\\
 \int _0 ^T r_n (s) \frac{\sin \sqrt{\lambda _{\alpha,n}} (T-s)}{\sqrt{\lambda _{\alpha,n}}} \, ds = Y_{\alpha,n} ^f, &\text{ for all } n \geq 1,\\
\int _0 ^T r_0 (s) (T-s) \, ds = Y_{\alpha,0} ^f.
\end{cases}
\end{equation}
Introducing the function
$$ Q(s) := q (T-s) ,$$
\eqref{syst-moments} becomes
\begin{equation}
\label{syst-momentsQ}
\begin{cases}
\mu _{\alpha,0} \int _0 ^T Q (t) \, dt = Z_{\alpha,0} ^f , \\
\mu _{\alpha,n} \int _0 ^T Q (t) \cos \sqrt{\lambda _{\alpha,n}} t \, dt = Z_{\alpha,n} ^f, & \text{ for all } n \geq 1,\\
\mu _{\alpha,n} \int _0 ^T Q (t) \sin \sqrt{\lambda _{\alpha,n}}t \, dt = \sqrt{\lambda _{\alpha,n}} \, Y_{\alpha,n} ^f, & \text{ for all } n \geq 1,\\
\mu _{\alpha,0} \int _0 ^T   Q(t) \, t \, dt = Y_{\alpha,0} ^f.
\end{cases}
\end{equation}
System \eqref{syst-momentsQ} is usually called \emph{moment problem}. Observe that \eqref{hyp-mu} implies that the coefficients $\mu _{\alpha,n}$ are different from $0$ for all $n\geq 0$, which is a necessary condition to solve \eqref{syst-momentsQ}.

Let us introduce
\begin{equation}
\label{eq-coeffs} 
\begin{cases}
A_{\alpha,0} ^f :=\frac{ Z_{\alpha,0} ^f}{\mu_{\alpha,0}} ,\\
A_{\alpha,n} ^f := \frac{Z_{\alpha,n} ^f}{\mu _{\alpha,n}},   & \text{ for all } n \geq 1,\\
B_{\alpha,n} ^f := \frac{\sqrt{\lambda _{\alpha,n}} \, Y_{\alpha,n} ^f}{\mu _{\alpha,n}}, & \text{ for all } n \geq 1,\\
B_{\alpha,0} ^f := \frac{Y_{\alpha,0} ^f}{\mu_{\alpha,0}} ,
\end{cases} 
\end{equation}
and
\begin{equation}
\label{eq-riesz} 
\begin{cases}
c_{\alpha,0}: t\in (0,T) \mapsto 1 ,\\
c_{\alpha,n}: t\in (0,T) \mapsto \cos \sqrt{\lambda _{\alpha,n}} t \quad \text{ for all } n \geq 1,\\
s_{\alpha,n}: t\in (0,T) \mapsto \sin \sqrt{\lambda _{\alpha,n}} t \quad \text{ for all } n \geq 1,\\
\tilde s _{\alpha,0}: t\in (0,T) \mapsto t ,
\end{cases} 
\end{equation}
in such a way that \eqref{syst-momentsQ} can be written as
\begin{equation}
\label{syst-momentsQ-ps}
\begin{cases}
\langle  Q, c_{\alpha,0} \rangle _{L^2 (0,T)}   = A_{\alpha,0} ^f , \\
\langle  Q, c_{\alpha,n} \rangle _{L^2 (0,T)}   = A_{\alpha,n} ^f \quad \text{ for all } n \geq 1, \\
\langle  Q, s_{\alpha,n} \rangle _{L^2 (0,T)}   = B_{\alpha,n} ^f \quad \text{ for all } n \geq 1,\\
\langle  Q, \tilde s_{\alpha,0} \rangle _{L^2 (0,T)}   = B_{\alpha,0} ^f .
\end{cases}
\end{equation}
We are going to prove that system \eqref{syst-momentsQ-ps} has (at least) a solution $Q$ in two steps:
\begin{itemize}
\item we prove that the reduced system
\begin{equation}
\label{syst-momentsQ-ps-red}
\begin{cases}
\langle  Q, c_{\alpha,0} \rangle _{L^2 (0,T)}   = A_{\alpha,0} ^f , \\
\langle  Q, c_{\alpha,n} \rangle _{L^2 (0,T)}   = A_{\alpha,n} ^f \quad \text{ for all } n \geq 1, \\
\langle  Q, s_{\alpha,n} \rangle _{L^2 (0,T)}   = B_{\alpha,n} ^f \quad \text{ for all } n \geq 1
\end{cases}
\end{equation}
has at least a solution $Q_\alpha$,

\item using $Q_\alpha$, we construct a solution $Q$ of the full system \eqref{syst-momentsQ-ps}.
\end{itemize}
We use results stated in \cite[Appendix]{B-L} (see also \cite{Duca}).

{\it Step 1: Existence of a solution of the reduced system \eqref{syst-momentsQ-ps-red}.}
We consider the space
$$ E_\alpha := \overline{\text{ Vect } \{c_{\alpha,0}, c_{\alpha,n}, s_{\alpha,n},\, n\geq 1 \} } ,$$
which is a closed subspace of $L^2 (0,T)$. To solve the reduced system, we use the following characterization of Riesz Basis (see (\cite[Prop. 19]{B-L}): the family $\{c_{\alpha,0}, c_{\alpha,n}, s_{\alpha,n}, n\geq 1 \}$ is a Riesz basis of $E_\alpha$ if and only if there exist $C_1 (\alpha,T), C_2 (\alpha,T) >0$ such that, for all $N \geq 1$ and for any $(a_n)_{0\leq n\leq N}$, $(b_n)_{1\leq n\leq N}$ it holds that
\begin{equation}
\label{BL1}
C_1 \Bigl( a_0 ^2 + \sum _{n=1} ^N a_n ^2 + b_n ^2 \Bigr) 
\leq \int _0 ^T \Bigl\vert S^{(a,b)} (t) \Bigr\vert ^2 \, dt 
\leq C_2 \Bigl( a_0 ^2 + \sum _{n=1} ^N a_n ^2 + b_n ^2 \Bigr) ,
\end{equation}
where
\begin{equation}
\label{BL2}
S^{(a,b)} (t) = a_0 c_{\alpha,0} (t) + \sum _{n=1} ^N a_n c_{\alpha,n} (t) + b_n s_{\alpha,n} (t).
\end{equation}
We observe that \eqref{BL1} holds true as a consequence of Ingham theory: indeed, by expressing $\cos y$ and $\sin y$ as
$$ \cos y = \frac{e^{iy} + e^{-iy}}{2}, \quad \text{ and } \quad \sin y = \frac{e^{iy} - e^{-iy}}{2i},$$
we have that
\begin{equation*}
\begin{split}
S^{(a,b)} (t) &= a_0 e^{i \omega_{\alpha,0} t} + \sum _{n=1} ^N a_n \frac{e^{i\omega_{\alpha,n} t} + e^{-i\omega_{\alpha,n} t}}{2} + b_n \frac{e^{i\omega_{\alpha,n} t} - e^{-i\omega_{\alpha,n} t}}{2i} 
\\
&= \sum _{n=-N} ^N d_n e^{i\omega_{\alpha,n} t} ,
\end{split}
\end{equation*}
with 
$$ 
\begin{cases} 
d_0 = a_0, \\
d_n = \frac{a_n}{2} + \frac{b_n}{2i}, & \text{ for } n\geq 1 , \\
d_n = \frac{a_{-n}}{2} - \frac{b_{-n}}{2i}, & \text{ for } n\leq -1.
\end{cases} $$
Since $\omega_{\alpha,n+1} - \omega_{\alpha,n} >0$ for all $n\in \mathbb Z$ and 
$$ \forall\, \vert n \vert \geq 2, \quad \omega_{\alpha,n+1} - \omega_{\alpha,n} \geq \frac{2-\alpha}{2} \pi ,$$
we can apply a general result of Haraux \cite{Haraux} (see also \cite[Theorem 6]{B-L}) that ensures that if
$$ T> \frac{2\pi}{\frac{2-\alpha}{2} \pi} = \frac{4}{2-\alpha},$$
then, there exist $C^{(I)} _1, C^{(I)} _2 >0$ independent of $N$ and of the coefficients $(d_n)_{-N\leq n\leq N}$, such that
\begin{equation}
\label{BL3}
C^{(I)} _1 \sum _{n=-N} ^N \vert d_n \vert ^2 
\leq \int _0 ^T \Bigl\vert \sum _{n=-N} ^N d_n e^{i\omega_{\alpha,n} t}  \Bigr\vert ^2 \, dt 
\leq C^{(I)} _2 \sum _{n=-N} ^N \vert d_n \vert ^2.
\end{equation}
Since
$$ \sum _{n=-N} ^N \vert d_n \vert ^2 = a_0 ^2 + 2 \sum _{n=1} ^N \frac{a_n ^2 + b_n ^2}{4} ,$$
\eqref{BL3} implies that \eqref{BL1} holds true and so the family $\{c_{\alpha,0}, c_{\alpha,n}, s_{\alpha,n}, n\geq 1 \}$ is a Riesz basis of $E_\alpha$. Thus, the application $\mathcal F : E_\alpha \to \ell ^2 (\mathbb N)$:
$$\mathcal F(f)= ( \langle f, c_{\alpha,0} \rangle _{L^2 (0,T)}, \langle f, c_{\alpha,1} \rangle _{L^2 (0,T)} , \langle f, s_{\alpha,1} \rangle _{L^2 (0,T)}, \langle f, c_{\alpha,2} \rangle _{L^2 (0,T)}, \cdots) $$
is an isomorphism (see, e.g., \cite[Proposition 20]{B-L}).
We note that 
$$ Y^f \in H^3 _{(0)} (0,1) \quad \implies \quad \sum _{n=0} ^\infty \lambda _{\alpha,n} ^3 \vert Y^f _n \vert ^2 < \infty,$$
and
$$ Z^f \in D(A) \quad \implies \quad \sum _{n=0} ^\infty \lambda _{\alpha,n} ^2 \vert Z^f _n \vert ^2 < \infty,$$
and then \eqref{hyp-mu} ensures us that
$$ \vert A _{\alpha,0} ^f \vert ^2 + \sum _{n=1} ^\infty \vert A _{\alpha,n} ^f \vert ^2 + \vert B _{\alpha,n} ^f \vert ^2 < \infty . $$
Therefore, there exists a unique $Q_\alpha \in E_\alpha$ such that
$$ \mathcal F (Q_\alpha) = (A _{\alpha,0} ^f, A _{\alpha,1} ^f, B _{\alpha,1} ^f, A _{\alpha,2} ^f, \cdots).$$
Thus,
\begin{equation}
\label{syst-momentsQ-ps-part}
\begin{cases}
\langle  Q_\alpha, c_{\alpha,0} \rangle _{L^2 (0,T)}   = A_{\alpha,0} ^f , \\
\langle  Q_\alpha, c_{\alpha,n} \rangle _{L^2 (0,T)}   = A_{\alpha,n} ^f \quad \text{ for all } n \geq 1, \\
\langle  Q_\alpha, s_{\alpha,n} \rangle _{L^2 (0,T)}   = B_{\alpha,n} ^f \quad \text{ for all } n \geq 1;
\end{cases}
\end{equation}
and, moreover, the application 
\begin{equation*}
\begin{split}
 \ell ^2 (\mathbb N) &\to E_\alpha,\\
(A_{\alpha,0} ^f, A_{\alpha,1} ^f, B_{\alpha,1} ^f, A_{\alpha,2} ^f, \cdots) &\mapsto Q_\alpha 
\end{split}
\end{equation*}
is continuous.

{\it Step 2: Existence of a solution of the full system \eqref{syst-momentsQ-ps}.}
We claim that $\tilde s_{\alpha,0} \notin E_\alpha$: indeed, if $t\mapsto t$ was the limit of a sequence of linear combinations of $c_{\alpha,0}$, $c_{\alpha,n}$ and $s_{\alpha,n}$, the same would be true for the function $t\mapsto t^2$, by integration. Then, by integrating further, also $t\mapsto t^3$ would be the limit of a sequence of linear combinations of $c_{\alpha,0}$, $c_{\alpha,n}$ and $s_{\alpha,n}$. Thus, by iterating this procedure, we deduce that all the polynomials could be written in this form. Therefore, $L^2 (0,T)$ would be equal to $E_\alpha$ and \eqref{syst-momentsQ-ps-part} would have a unique solution. However, this is not the case: define $T_0:= \frac{4}{2-\alpha}$, and choose $q_\alpha$ smooth, compactly supported in $(0, \frac{T-T_0}{2})$ and different from $Q_\alpha$ on that interval. Now, consider the following problem
\begin{equation}
\label{syst-momentsQ-ps-part2}
\begin{cases}
\langle  \tilde Q_\alpha, c_{\alpha,0} \rangle _{L^2 (\frac{T-T_0}{2},T)}   = A_{\alpha,0} ^f - \langle  q_\alpha, c_{\alpha,0} \rangle _{L^2 (0,\frac{T-T_0}{2})}, \\
\langle  \tilde Q_\alpha, c_{\alpha,n} \rangle _{L^2 (\frac{T-T_0}{2},T)}   = A_{\alpha,n} ^f - \langle  q_\alpha, c_{\alpha,n} \rangle _{L^2 (0,\frac{T-T_0}{2})}\quad \text{ for all } n \geq 1, \\
\langle  \tilde Q_\alpha, s_{\alpha,n} \rangle _{L^2 (\frac{T-T_0}{2},T)}   = B_{\alpha,n} ^f - \langle  q_\alpha, s_{\alpha,n} \rangle _{L^2 (0,\frac{T-T_0}{2})}\quad \text{ for all } n \geq 1.
\end{cases}
\end{equation}
Since $T-\frac{T-T_0}{2} = \frac{T+T_0}{2} >T_0$, and the sequences $(\langle q_\alpha, c_{\alpha,n} \rangle _{L^2 (0,\frac{T-T_0}{2})})_n$ and $(\langle q_\alpha, s_{\alpha,n} \rangle _{L^2 (0,\frac{T-T_0}{2})})_n$ are square-summable (by integration by parts), there exists a solution $\tilde Q _\alpha \in L^2 (\frac{T-T_0}{2},T)$ of \eqref{syst-momentsQ-ps-part2}. 
So, the function
$$ Q_\alpha ^* := \begin{cases} q_\alpha \text{ on } (0,\frac{T-T_0}{2}), \\ 
\tilde Q_\alpha \text{ on } (\frac{T-T_0}{2},T)
\end{cases}$$
solves \eqref{syst-momentsQ-ps-part}. However, $Q_\alpha ^* \neq Q_\alpha$ on $(0, \frac{T-T_0}{2})$, which is contradiction with the fact that $Q_\alpha$ is the unique solution of \eqref{syst-momentsQ-ps-part}. Therefore $\tilde s_{\alpha,0} \notin E_\alpha$, and 
if we denote $p ^\perp _{\alpha,0}$ the orthogonal projection of $\tilde s_{\alpha,0}$ on $E_\alpha$, then $\tilde s_{\alpha,0}- p ^\perp _{\alpha,0} \neq 0$, and 
$$ Q_\alpha ^\perp := \frac{\tilde s_{\alpha,0}- p ^\perp _{\alpha,0}}{\Vert \tilde s_{\alpha,0}- p ^\perp _{\alpha,0} \Vert _{L^2 (0,T)} ^2} $$ is orthogonal to $E_\alpha$, and furthermore
$$ \langle  Q_\alpha ^\perp , \tilde s_{\alpha,0} \rangle _{L^2 (0,T)} = 1.$$
Thus,
$$ Q := Q_\alpha + B_{\alpha,0} Q_\alpha ^\perp $$
solves \eqref{syst-momentsQ-ps}. Moreover,
$$ \Vert Q \Vert _{L^2 (0,T)} ^2 = \Vert Q_\alpha \Vert _{L^2 (0,T)} ^2  + \Vert B_{\alpha,0}^f Q_\alpha ^\perp \Vert _{L^2 (0,T)} ^2 
\leq C \Bigl( \sum _{n=0} ^\infty \vert A_{\alpha,n}^f \vert ^2 + \vert B_{\alpha,n}^f \vert ^2 \Bigr) ,$$
which completes the proof of Lemma \ref{lem-DThetaT(0)inv}.

\end{proof}


\subsection{Proof of Theorem \ref{thm-ctr} (inverse mapping argument)} \hfill
\label{sec-concl-T>T0}

We define the space 
$$ F_\alpha := \overline{\text{Vect } \{1,t, \cos \sqrt{\lambda _{\alpha,n}} t, \sin \sqrt{\lambda _{\alpha,n}} t, n\geq 1 \}}.$$
Then, the restriction of $\Theta _T$ to $F_\alpha$
\begin{equation*}
\begin{split}
\Theta _{\alpha,T}: F_\alpha &\to H^3 _{(0)} (0,1) \times D(A), \\
p &\mapsto \Theta _{\alpha,T} (p) := \Theta _{T}(p)
\end{split}
\end{equation*}
is $C^1$ (Proposition \ref{prop-"thm3"}) and $D\Theta _{\alpha,T} (0)$ is invertible (Lemma \ref{lem-DThetaT(0)inv}). Thus, the inverse mapping theorem ensures that there exists a neighborhood $\mathcal V (0) \subset F_\alpha$ and a neighborhood $\mathcal V (1,0) \subset H^3 _{(0)} (0,1) \times D(A)$ such that
$$ \Theta _{\alpha,T}: \mathcal V (0) \to \mathcal V (1,0) $$
is a $C^1$-diffeomorphism. Hence, given $(w_0 ^f, w_1 ^f) \in \mathcal V (1,0)$, we choose $p^f:= \Theta _{\alpha,T} ^{-1} (w_0 ^f, w_1 ^f)$, and so the solution of \eqref{CMU-eq-ctr-10} with $p= p^f$ satisfies 
$$ (w(T), w_t (T)) = \Theta _T (p^f) = \Theta _T (\Theta _{\alpha,T} ^{-1} (w_0 ^f, w_1 ^f)) = (w_0 ^f, w_1 ^f) .$$
This concludes the proof of Theorem \ref{thm-ctr}. \qed


\section{Proof of Theorem \ref{thm-ctr=}: Reachability for $T=T_0$}
\label{sec-pr-thm=}

\subsection{Proof of Theorem \ref{thm-ctr=} first part: Reachability for $T=T_0$ and $\alpha \in [0,1)$} \hfill
\label{sec-pr-thm=0-1}

The proof follows from classical arguments concerning families of exponentials (\cite{Avdonin-Ivanov}) in the space $L^2 (0,T)$ and the strategy of Beauchard \cite{B-bi}:
\begin{itemize}
\item we study the solvability of the moment problem \eqref{syst-moments} (or equivalently \eqref{syst-momentsQ}), 
\item we conclude using the inverse mapping theorem.
\end{itemize}

\subsubsection{Main tools to study the solvability of the moment problem \eqref{syst-momentsQ}} \hfill

In order to use classical results on complex exponentials, we can transform \eqref{syst-momentsQ} into the following 
system

\begin{equation}
\label{syst-momentsQc}
\begin{cases}
\mu _{\alpha,0} \int _0 ^T Q (t) \, dt = Z_{\alpha,0} ^f , \\
\mu _{\alpha,n} \int _0 ^T Q (t) e^{-i\sqrt{\lambda _{\alpha,n}} t } \, dt = Z_{\alpha,n} ^f - i \sqrt{\lambda _{\alpha,n}} \, Y_{\alpha,n} ^f, & \text{ for all } n \geq 1,\\
\mu _{\alpha,n} \int _0 ^T Q (t) e^{i\sqrt{\lambda _{\alpha,n}} t } \, dt = Z_{\alpha,n} ^f + i \sqrt{\lambda _{\alpha,n}} \, Y_{\alpha,n} ^f, & \text{ for all } n \geq 1,\\
\mu _{\alpha,0} \int _0 ^T   Q(t) \, t \, dt = Y_{\alpha,0} ^f .
\end{cases}
\end{equation}
By introducing the notation
\begin{equation}
\label{syst-coeffs}
\begin{cases}
C_{\alpha,0}^f := \frac{Z_{\alpha,0} ^f}{\mu _{\alpha,0}}, \\
 C_{\alpha,n}^f := \frac{Z_{\alpha,n} ^f - i \sqrt{\lambda _{\alpha,n}} \, Y_{\alpha,n} ^f}{\mu _{\alpha,n}} \quad \text{for all } n \geq 1 , \\
C_{\alpha,-n}^f := \frac{Z_{\alpha,n} ^f + i \sqrt{\lambda _{\alpha,n}} \, Y_{\alpha,n} ^f}{\mu _{\alpha,n}} \quad \text{for all } n \geq 1 ,
\end{cases}
\end{equation}
and the natural scalar product in $L^2(0,T;\mathbb C)$
$$ \forall f,g \in L^2(0,T;\mathbb C), \quad \langle  f, g \rangle _{L^2 (0,T; \mathbb C)} = \int _0 ^T f(t) \overline{g(t)} \, dt ,$$
\eqref{syst-momentsQc} can be written as
\begin{equation}
\label{syst-momentsQc-ps}
\begin{cases}
\langle  Q, e^{i \omega _{\alpha,0}t} \rangle _{L^2 (0,T;\mathbb C)}   = C_{\alpha,0} ^f , \\
\langle  Q, e^{i \omega _{\alpha,n}t} \rangle _{L^2 (0,T;\mathbb C)}   = C_{\alpha,n} ^f \quad \text{ for all } n \geq 1, \\
\langle  Q, e^{i \omega _{\alpha,n}t} \rangle _{L^2 (0,T;\mathbb C)}   = C_{\alpha,n} ^f \quad \text{ for all } n \leq -1,\\
\langle  Q, \tilde s_{\alpha,0} \rangle _{L^2 (0,T;\mathbb C)}   = B_{\alpha,0} ^f ,
\end{cases}
\end{equation}
(where $\tilde s_{\alpha,0}$ and $B_{\alpha,0} ^f $ have been defined in \eqref{eq-coeffs} and \eqref{eq-riesz}). We are going to study first the solvability of the subsystem composed by the first three equations, that is, the following moment problem
\begin{equation}
\label{syst-momentsQc-ps-ss}
 \forall\, n \in \mathbb Z, \quad \langle  Q, e^{i \omega _{\alpha,n}t} \rangle _{L^2 (0,T;\mathbb C)}   = C_{\alpha,n} ^f .
\end{equation}


\subsubsection{Main solvability results for $\alpha \in [0,1)$} \hfill
\label{sec-pr-thm-T=T0-01}
\begin{Lemma}
\label{lem-Riesz-T0}
Let $\alpha \in [0,1)$. Then, the sequence $(e^{i \omega _{\alpha,n}t})_{n\in \mathbb Z}$ is a Riesz basis in $L^2(0,T_0)$.
\end{Lemma}
From the previous Lemma we deduce the following result.
\begin{Lemma}
\label{lem-Riesz-T0-moment}
Let $\alpha \in [0,1)$ and $T=T_0$. Then, the moment problem \eqref{syst-momentsQc-ps-ss} has one and only one solution $Q \in L^2 (0,T_0;\mathbb R)$.
\end{Lemma}
Lemma \ref{lem-Riesz-T0-moment} will imply the following result.
\begin{Lemma}
\label{lem-non-solv}
Let $\alpha\in[0,1)$ and $T=T_0$. There exists a closed hyperplane of $H^3_{(0)} \times D(A)$, denoted by $P^f _\alpha$ and defined in \eqref{eq-hyperplan-def}, such that the moment problem \eqref{syst-momentsQc-ps} has a solution if and only if $(Y^f, Z^f) \in P^f _\alpha$.
\end{Lemma}

\subsubsection{Proof of Lemma \ref{lem-Riesz-T0}} \hfill

The proof follows from the Kadec' s $\frac{1}{4}$ Theorem (\cite{Kadec}, \cite[Theorem 1.14 p. 42]{Young}). First, we note that the sequence $(\omega _{\alpha,n})_{n\in \mathbb Z}$ is odd, that is $\omega_{\alpha,-n}=-\omega_{\alpha,n}$, and 
$$ \forall\, n \geq 1, \quad \omega _{\alpha,n} = \kappa _\alpha j_{-\nu_\alpha +1,n} = \frac{2-\alpha}{2} j_{\frac{1}{2-\alpha},n} .$$
Mac Mahon's formula (see \cite[p. 506]{Watson})
provides the following asymptotic development of $j_{\nu,n}$ as $n\to \infty$
 $$ j_{\nu,n} = \pi ( n + \frac{\nu}{2} - \frac{1}{4}) + O(\frac{1}{n}) .$$
Hence, we have
 $$ \omega _{\alpha,n} = \frac{2-\alpha}{2}\pi \Bigl( n + \frac{1}{2(2-\alpha)} - \frac{1}{4} \Bigr) + O(\frac{1}{n})
 = \frac{2-\alpha}{2}\pi \Bigl(n+ \frac{\alpha}{4(2-\alpha)} \Bigr) + O(\frac{1}{n}) .$$
Therefore, we deduce that
\begin{equation}
\label{DAS-omega_n}
\frac{2}{\pi(2-\alpha)} \, \omega _{\alpha,n} = n+ \frac{\alpha}{4(2-\alpha)} + O(\frac{1}{n}) \quad \text{ as } n \to +\infty ,
\end{equation}
and, in particular,
$$ \frac{2}{\pi(2-\alpha)} \, \omega _{\alpha,n} - n \to \frac{\alpha}{4(2-\alpha)} \quad \text{ as } n \to +\infty.
$$
Since $\alpha \in [0,1)$, it holds that
 $$ \frac{\alpha}{4(2-\alpha)} = \frac{1}{4} \, \frac{\alpha}{2-\alpha} < \frac{1}{4} ,$$
thus, we gather that for any $L \in (\frac{1}{4}\frac{\alpha}{2-\alpha}, \frac{1}{4})$ there exists $N_0$ such that
 $$ \forall\, n \geq N_0, \quad  \Bigl\vert \frac{2}{\pi(2-\alpha)} \, \omega _{\alpha,n} - n \Bigr\vert \leq L < \frac{1}{4} .$$
Since the sequence $(\frac{2}{\pi(2-\alpha)} \, \omega _{\alpha,n})_{n\in \mathbb Z}$ is odd, the above bound holds also for negative indices, hence
 $$ \forall\, \vert n \vert \geq N_0, \quad  \Bigl\vert \frac{2}{\pi(2-\alpha)} \, \omega _{\alpha,n} - n \Bigr\vert \leq L < \frac{1}{4} .$$
Then, consider the sequence $(\tilde \omega _{\alpha,n})_{n\in \mathbb Z}$ defined by
$$ 
\tilde \omega _{\alpha,n} := 
\begin{cases} 
 n , & \quad \forall \vert n \vert < N_0, \\
\frac{2}{\pi(2-\alpha)} \, \omega _{\alpha,n}, & \quad \forall \vert n \vert \geq N_0 .
\end{cases}$$
The new sequence $(\tilde \omega _{\alpha,n})_{n\in \mathbb Z}$ satisfies
$$ \forall n \in \mathbb Z, \quad \Bigl\vert \tilde \omega _{\alpha,n} - n \Bigr\vert \leq L < \frac{1}{4} ,$$
and we can apply Kadec's $\frac{1}{4}$ Theorem, which implies that the sequence $(e^{i \tilde \omega _{\alpha,n} \tau})_{n\in \mathbb Z}$ is a Riesz basis in $L^2(-\pi,\pi; \mathbb C)$ (where $\tau$ is the variable in $(-\pi,\pi)$), see \cite[Theorem 1.14 p. 42]{Young}. Thanks to the change of variables
$$ [-\pi, \pi] \to [0,T_0], \quad \tau \mapsto t = \frac{\tau}{\pi} \frac{T_0}{2} + \frac{T_0}{2} ,$$
we obtain that the sequence $(e^{i \tilde \omega _{\alpha,n} \frac{2\pi}{T_0} t})_{n\in \mathbb Z}$ is a Riesz basis in $L^2(0,T_0; \mathbb C)$. Indeed, for any $g\in L^2(0,T_0; \mathbb C)$, we define
$$ f(\tau) := g\left(\frac{T_0}{2\pi} \tau + \frac{T_0}{2}\right) \in L^2(-\pi,\pi; \mathbb C) ,$$
which can be developed using the Riesz basis $(e^{i \tilde \omega _{\alpha,n} \tau})_{n\in \mathbb Z}$
$$ f(\tau) = \sum _{n\in \mathbb Z} c_n e^{i \tilde \omega _{\alpha,n} \tau} , \quad \text{ with } \quad 
A_0 \sum _{n\in \mathbb Z} \vert c_n  \vert ^2 \leq \Vert f \Vert ^2 \leq B_0 \sum _{n\in \mathbb Z} \vert c_n  \vert ^2, $$
where $A_0$ and $B_0$ are suitable positive constants. Going back to the original time interval $(0,T_0)$, we obtain
\begin{equation*}
\begin{split}
 g(t) &= f\left(2\pi \frac{t-\frac{T_0}{2}}{T_0}\right) = \sum _{n\in \mathbb Z} c_n e^{i \tilde \omega _{\alpha,n} 2\pi \frac{t-\frac{T_0}{2}}{T_0}}\\
& = \sum _{n\in \mathbb Z} c_n \, e^{-i \tilde \omega _{\alpha,n} \pi}  \, e^{i \tilde \omega _{\alpha,n} \frac{2\pi}{T_0} t} 
 = \sum _{n\in \mathbb Z} \tilde c_n \, e^{i \tilde \omega _{\alpha,n} \frac{2\pi}{T_0} t}
 \end{split}
 \end{equation*}
where $\tilde c_n = c_n \, e^{-i \tilde \omega _{\alpha,n} \pi}$. We further deduce that there exist two positive constants $A_1$ and $B_1$ such that 
 $$ A_1 \sum _{n\in \mathbb Z} \vert \tilde c_n  \vert ^2 \leq \Vert g \Vert ^2 \leq B_1 \sum _{n\in \mathbb Z} \vert \tilde c_n  \vert ^2 , $$
because $\vert \tilde c_n \vert = \vert c_n \vert$ and $(e^{i \tilde \omega _{\alpha,n} \tau})_{n\in \mathbb Z}$ is a Riesz basis of $L^2(-\pi,\pi; \mathbb C)$. 
 
We notice that
 $$ \forall\, \vert n \vert \geq N_0, \quad \tilde \omega _{\alpha,n} \frac{2\pi}{T_0} = \tilde \omega _{\alpha,n} \frac{2\pi}{\frac{4}{2-\alpha}} = \frac{2}{\pi(2-\alpha)} \, \omega _{\alpha,n} \, \frac{\pi(2-\alpha)}{2} =  \omega _{\alpha,n}, $$
 and since modifying a finite number of terms does not affect the fact of being a Riesz basis (\cite[Lemma II.4.11 p. 105]{Avdonin-Ivanov}), we deduce that $(e^{i \omega _{\alpha,n} t})_{n\in \mathbb Z}$ is a Riesz basis in $L^2(0,T_0; \mathbb C)$. \qed


\subsubsection{Proof of Lemma \ref{lem-Riesz-T0-moment}} \hfill

 Since $(e^{i \omega _{\alpha,n} t})_{n\in \mathbb Z}$ is a Riesz basis in $L^2(0,T_0; \mathbb C)$, there exists one and only one biorthogonal sequence: $(\sigma _m (t))_{m\in \mathbb Z}$ satisfying
$$ \langle \sigma _m, e^{i \omega _{\alpha,n} t} \rangle _{L^2(0,T_0;\mathbb C)} = \int _0 ^{T_0} \sigma _m (t) e^{-i \omega _{\alpha,n} t} \, dt = \delta _{mn} .$$
Taking the conjugate, we obtain that
$$ \int _0 ^{T_0} \overline{\sigma _m (t)} e^{i \omega _{\alpha,n} t} \, dt = \delta _{mn} .$$
Recalling that $\omega _{\alpha,n} = - \omega _{\alpha,-n}$, we have
$$ \langle \overline{\sigma _m}, e^{i \omega _{\alpha,-n} t} \rangle _{L^2(0,T_0;\mathbb C)}
= \delta _{mn} = \langle \sigma _{-m}, e^{i \omega _{\alpha,-n} t} \rangle _{L^2(0,T_0;\mathbb C)} ,$$
which implies that 
$$\forall\, m\in \mathbb Z, \quad \overline{\sigma _m} = \sigma _{-m} .$$
Now, using once again that $(e^{i \omega _{\alpha,n} t})_{n\in \mathbb Z}$ is a Riesz basis in $L^2(0,T_0; \mathbb C)$, the moment problem \eqref{syst-momentsQc-ps-ss} has ne and only one solution, given by
\begin{equation}
\label{formule-sol-moments}
Q(t) = \sum _{m\in \mathbb Z} C_{\alpha,m} ^f \sigma _m (t) .
\end{equation}
It remains to verify that $Q$ takes its values in $\mathbb R$: taking the conjugate, we have
$$ \overline{Q(t)} = \sum _{m\in \mathbb Z} \overline{C_{\alpha,m} ^f} \overline{\sigma _m (t)}
= \sum _{m\in \mathbb Z} C_{\alpha,-m} ^f \sigma _{-m} (t) = Q(t) .$$
Hence $Q \in L^2 (0,T_0;\mathbb R)$. \qed

\subsubsection{Proof of Lemma \ref{lem-non-solv}} \hfill

We have proved in Lemma \ref{lem-Riesz-T0-moment} that the subsystem \eqref{syst-momentsQc-ps-ss} admits a unique solution $Q \in L^2 (0,T_0;\mathbb R)$, given by \eqref{formule-sol-moments}. Therefore, the moment problem \eqref{syst-momentsQc-ps} is satisfied if and only if the solution $Q$ given by \eqref{formule-sol-moments} satisfies also the last equation in \eqref{syst-momentsQc-ps}.
Since $(e^{i \omega _{\alpha,n} t})_{n\in \mathbb Z}$ is a Riesz basis in $L^2(0,T_0; \mathbb C)$, there exists a unique sequence $(\beta_n)_{n\in \mathbb Z} \in \ell ^2 (\mathbb Z)$ such that
$$ \tilde s _{\alpha,0} (t) = \sum _{m\in \mathbb Z} \beta_m \, e^{i \omega _{\alpha,m} t} \quad \text{ in } L^2(0,T_0; \mathbb C).$$
Since $\tilde s _{\alpha,0}$ is real-valued, we have that
$$ \sum _{m\in \mathbb Z} \beta_m \, e^{i \omega _{\alpha,m} t} = \tilde s _{\alpha,0} (t)
= \overline{\tilde s _{\alpha,0} (t)}
= \sum _{m\in \mathbb Z} \overline{\beta_m} \, e^{-i \omega _{\alpha,m} t} 
= \sum _{m\in \mathbb Z} \overline{\beta_{-m}} \, e^{i \omega _{\alpha,m} t},$$
from which we deduce that
$$ \forall\, m \in \mathbb Z, \quad \overline{\beta_m} = \beta _{-m} .$$
Thus,
\begin{equation*} 
\begin{split}
\langle Q, \tilde s _{\alpha,0} \rangle _{L^2(0,T_0;\mathbb C)}
&= \langle Q, \sum _{m\in \mathbb Z} \beta_m \, e^{i \omega _{\alpha,m} t} \rangle _{L^2(0,T_0;\mathbb C)}
\\
&= \sum _{m\in \mathbb Z} \overline{\beta_m} \, \langle Q, e^{i \omega _{\alpha,m} t} \rangle _{L^2(0,T_0;\mathbb C)}
= \sum _{m\in \mathbb Z} \overline{\beta_m} \, C_{\alpha,m} ^f.
\end{split}
\end{equation*}
Hence, the solution $Q$ given by \eqref{formule-sol-moments} solves \eqref{syst-momentsQc-ps} if and only if
$$ B_{\alpha,0} ^f = \sum _{m\in \mathbb Z} \overline{\beta_m} \, C_{\alpha,m} ^f , $$
or, equivalently, if and only if
\begin{equation}
\label{ssev-29avril}
\frac{Y_{\alpha,0} ^f}{\mu _{\alpha,0}}
= \beta_0 \frac{Z_{\alpha,0} ^f}{\mu _{\alpha,0}} 
+ \sum _{m\geq 1} \Bigl(  2 (\text{Re } \beta_m) \frac{Z_{\alpha,m} ^f}{\mu _{\alpha,m}} 
- 2 (\text{Im } \beta_m) \omega _{\alpha,m} \frac{Y_{\alpha,m} ^f}{\mu _{\alpha,m}}
\Bigr) ,
\end{equation}
where the relation of $B^f_{\alpha,0}$ and $C^f_{\alpha,m}$ with $(Y^f,Z^f)$ are given in \eqref{eq-coeffs} and \eqref{syst-coeffs}. We now introduce the closed hyperplane $P^f _\alpha$ of $H^3 (0) \times D(A)$ defined by
\begin{equation}
\label{eq-hyperplan-def}
P^f _\alpha := \{(Y^f,Z^f) \in H^3 (0) \times D(A) \text{ such that \eqref{ssev-29avril} is satisfied} \}.
%
\end{equation}
Therefore \eqref{syst-momentsQc-ps} has a solution if and only if $(Y^f,Z^f) \in P_\alpha ^f$. \qed 

\subsubsection{Proof of Theorem \ref{thm-ctr=} first part (inverse mapping argument)} \hfill
\label{sec-concl-T=T0}

As in section \ref{sec-concl-T>T0}, we consider the application
\begin{equation}
\label{def-thetaT0}
\Theta _{T_0}: L^2 (0,T_0) \to H^3 _{(0)} \times D(A), \quad \Theta _{T_0} (p) := (w^{(p)} (T_0), w^{(p)} _t (T_0)) .
\end{equation}
We recall that $P^f _\alpha$ is the closed hyperplane of $H^3 _{(0)} \times D(A)$ defined by \eqref{eq-hyperplan-def}.

From the previous section it follows that the application $D \Theta _{T_0}$ satisfies
$$ D\Theta _{T_0} (0) (L^2 (0,T_0) )  \subset P_\alpha ^f .$$
Indeed, if $(Y^f,Z^f)= (W^{(0,q)} (T_0), W_t ^{(0,q)} (T_0)$, then the moment problem is satisfied and the Fourier coefficients of $Y^f$ and $Z^f$ satisfy \eqref{ssev-29avril}, hence $(Y^f,Z^f) \in P^f _\alpha$. Moreover,
$$ D\Theta _{T_0}(0): L^2 (0,T_0) \to  P_\alpha ^f,  $$
is invertible. In fact, it follows from Lemma \ref{lem-Riesz-T0-moment} and formula \eqref{formule-sol-moments}.

Now consider $(Y^\perp, Z^\perp) \neq 0$ and orthogonal to $P^f _\alpha$: this allows us to decompose
the space $H^3 _{(0)} \times D(A)$ into
$$ H^3 _{(0)} \times D(A) = P^f _\alpha \oplus (P^f _\alpha) ^\perp $$
where $(P^f _\alpha) ^\perp = \mathbb R (Y^\perp, Z^\perp)$ is one dimensional.
Consider the associated orthogonal projections $\text{proj}_{P^f _\alpha}$ and $\text{proj}^\perp _{P^f _\alpha}$. Any $(Y,Z)\in H^3_{(0)}(0,1)\times D(A)$ can be decompose as
$$ (Y,Z) = \text{proj}_{P^f _\alpha} (Y,Z) + \text{proj}^\perp _{P^f _\alpha} (Y,Z) \text{ with } \begin{cases}
\text{proj}_{P^f _\alpha} (Y,Z) \in P^f _\alpha,\\
\text{proj}^\perp _{P^f _\alpha} (Y,Z) \in (P^f _\alpha) ^\perp  . \end{cases}$$
The application
$$ \tilde \Theta _{\alpha,T_0} : L^2(0,T_0) \to P^f _\alpha, \quad \tilde \Theta _{\alpha,T_0} (q) = \text{proj}_{P^f _\alpha} (\Theta _{T_0} (q)) $$
satisfies
$$ D\tilde \Theta _{\alpha,T_0}(0) : L^2(0,T_0) \to P^f _\alpha, \quad D\tilde\Theta _{\alpha,T_0}(0) = \text{proj}_{P^f _\alpha}  (D\Theta _{T_0}(0)) .$$
Hence $D\tilde \Theta _{\alpha,T_0}(0) : L^2(0,T_0) \to P^f _\alpha$ is invertible, and therefore the inverse mapping theorem implies that there exists a neighborhood $\mathcal V (0) \subset L^2(0,T_0)$ and a neighborhood $\mathcal V (\text{proj}_{P^f _\alpha} ((1,0)) \subset P^f _\alpha$ such that
$$ \tilde \Theta _{\alpha,T_0} : \mathcal V (0) \subset L^2(0,T_0) \to \mathcal V (\text{proj}_{P^f _\alpha} ((1,0)) \subset P^f _\alpha $$ is a $C^1$-diffeomorphism. Therefore,
$$ \Theta _{T_0} (\mathcal V (0)) = \{ (Y,Z) + \text{proj}^\perp _{P^f _\alpha} (\Theta _{T_0} (\tilde \Theta ^{-1} _{\alpha,T_0} (Y,Z))) , (Y,Z) \in \mathcal V (\text{proj}_{P^f _\alpha} ((1,0)) \} , $$
which means that $\Theta _{T_0} (\mathcal V (0))$ is the graph of the application 
$$ \mathcal V (\text{proj}_{P^f _\alpha} ((1,0)) \to (P^f _\alpha) ^\perp, \quad (Y,Z) \mapsto \text{proj}^\perp _{P^f _\alpha} (\Theta _{T_0} (\tilde\Theta ^{-1} _{\alpha,T_0} (Y,Z))) ,$$
hence $\Theta _{T_0} (\mathcal V (0))$ is a submanifold of codimension 1. This concludes the proof  of Theorem \ref{thm-ctr<} in the case  $T=T_0$ and $\alpha \in [0,1)$. \qed


\subsection{Proof of Theorem \ref{thm-ctr=} second part: Reachability when $T=T_0$ and $\alpha \in [1,2)$} \hfill
\label{sec-pr-thm=1-2}

When $\alpha \in [1,2)$, we derive from \eqref{DAS-omega_n} that
$$ \frac{2}{\pi(2-\alpha)} \, \omega _{\alpha,n} - n \to \frac{\alpha}{4(2-\alpha)} \quad \text{ as } n \to \infty .$$
However, we notice that
$$ \alpha \in [1,2) \quad \implies \quad \frac{\alpha}{4(2-\alpha)} \geq \frac{1}{4} .$$
This fact represents the main difference with respect to the analysis of the solvability of moment problem \eqref{syst-momentsQc-ps} of section \ref{sec-pr-thm=0-1} (Lemma \ref{lem-Riesz-T0}-\ref{lem-non-solv}).

\subsubsection{Main solvability results when $\alpha \in [1,2)$} \hfill
\label{sec-pr-thm-T=T0-1-2}

In this section will prove an extension of the Kadec's $\frac{1}{4}$ Theorem (\cite{Kadec}, \cite[Theorem 1.14 p. 42]{Young}). Our results are similar to those of \cite[Theorem F p. 149]{Joo}, however, thanks to our assumptions, we are able to give very simple statements and proofs.

\begin{Lemma}
\label{lem-Riesz-T0-1-2} 
Consider an odd sequence of real numbers $(x_n)_{n\in \mathbb Z}$, that is $x_{-n}=-x_n$, such that there exist $k \in \mathbb N^* $, $\delta \in (0, \frac{1}{4})$ and $N_0 \geq 0$ for which
\begin{equation}
\label{hyp-Kadec-k}
n  \geq N_0 \quad \implies \quad  \left\vert x_n - n - \frac{k}{2} \right\vert \leq \frac{1}{4} - \delta .
\end{equation}
Then, choosing distinct real numbers $x' _1, \cdots, x'_k \notin \{ x_n, n \in \mathbb Z \}$, the set 
$$ \{ e^{ix'_1 t}, \cdots, e^{ix'_k t} \} \cup \{ e^{ix_n t}, n\in \mathbb Z \}$$
is a Riesz basis of $L^2 (-\pi,\pi)$.
\end{Lemma}

As a consequence of the above result, we deduce the following Lemma.

\begin{Lemma}
\label{cor-Riesz}
Let $\alpha \in [1,2)$. Assume that 
\begin{equation}
\label{cond-alpha}
\frac{1}{2-\alpha} \notin \mathbb N .
\end{equation}
Denote by $k_\alpha$ the integer part of $\frac{1}{2-\alpha}$.
Then, the set $\{ e^{i\omega_{\alpha,n} t }, n\in \mathbb Z \}$ can be complemented by $k_\alpha$ exponentials to form a Riesz basis of $L^2 (0,T_0)$.
\end{Lemma}

Lemma \ref{cor-Riesz} implies the following result.

\begin{Lemma}
\label{lem-Riesz-T0-moment-1-2}
Let $\alpha \in [1,2)$ satisfy \eqref{cond-alpha} and let $T=T_0$. Then, moment problem \eqref{syst-momentsQc-ps} has a unique solution
$$Q \in L^2 (0,T_0;\mathbb R) \in F_\alpha := \overline{\text{Vect } \{\tilde s _{\alpha,0}, e^{i\omega _{\alpha,n}t} , n\in \mathbb Z \} } .$$
\end{Lemma}


\subsubsection{Proof of Lemma \ref{lem-Riesz-T0-1-2}} \hfill

We are going to prove it for $k=1$ and $k=2$, and then the other cases are easily deduced.

\underline{Case $k$ even.} We consider $k=2$, however the method applies similarly for all $k$ even. For $k=2$, the assumption reads as: there exist $\delta\in\left(0,\frac{1}{4}\right)$ and $N_0\geq 0$ such that
\begin{equation}
\label{hyp-Kadec-2}
n \geq N_0 \quad \implies \quad  \vert x_n - (n + 1) \vert \leq \frac{1}{4} - \delta .
\end{equation}
Then, let us consider the following sequence
\begin{equation}
\label{def-om-mod2}
\forall n \in \mathbb Z, \quad x ^{(mod)} _n = 
\begin{cases}
x _{n-1}  & \text{ if } n \geq N_0 +1 , \\
n & \text{ if } \vert n \vert \leq N_0 , \\
x _{n+1}  & \text{ if } n \leq - N_0 -1 .
\end{cases}
\end{equation}
We claim that
\begin{equation}
\label{def-om-mod2-prop}
\forall\, n \in \mathbb Z, \quad \vert x ^{(mod)} _n  - n \vert \leq \frac{1}{4} - \delta.
\end{equation}
Indeed, \eqref{def-om-mod2-prop} is straightforward for any $\vert n \vert \leq N_0$. Moreover, if $n\geq N_0+1$ we get that
$$ \vert x ^{(mod)} _n  - n \vert = \vert x _{n-1}  - n \vert \leq \frac{1}{4} - \delta $$
thanks to \eqref{hyp-Kadec-2}. Finally, if  $n \leq - N_0 -1$, 
\begin{multline*}
\vert x ^{(mod)} _n  - n \vert 
= \vert x _{n+1}   - n \vert
= \Bigl\vert - x _{\vert n+1 \vert } + \vert n \vert \Bigr\vert 
\\
= \Bigl\vert - x _{\vert n \vert -1 } + \vert n \vert \Bigr\vert
= \Bigl\vert x _{\vert n \vert - 1 } - \vert n \vert \Bigr\vert 
\leq \frac{1}{4} - \delta 
\end{multline*}
once again using \eqref{hyp-Kadec-2}.
Then \eqref{def-om-mod2-prop} is satisfied. We deduce from the Kadec's $\frac{1}{4}$ Theorem (\cite{Kadec}, \cite[Theorem 1.14 p. 42]{Young}) that the set $\{ e^{ix ^{(mod)} _n t},\, n\in \mathbb Z \}$ is a Riesz basis of $L^2 (-\pi,\pi)$. However, we can reorder the family $ \{ e^{ix ^{(mod)} _n t},\, n\in \mathbb Z \} $ as follows
\begin{equation*}
\begin{split}
 \{ e^{ix ^{(mod)} _n t},\,& n\in \mathbb Z \} 
 \\ &= \{ e^{ix _{n+1} t},\, n \leq -N_0 - 1 \} \cup \{ e^{ix ^{(mod)} _n t},\, \vert n \vert \leq N_0 \} \cup \{ e^{ix _{n-1} t},\, n \geq N_0 + 1 \}
\\
&= \{ e^{ix _{m} t},\, m \leq - N_0 \} \cup \{ e^{ix ^{(mod)} _n t},\, \vert n \vert \leq N_0 \} \cup \{ e^{ix _{m} t},\, m \geq N_0 \}  \\
&= \{ e^{ix _{m} t},\, \vert m \vert  \geq N_0 \} \cup \{ e^{ix ^{(mod)} _n t},\, \vert n \vert \leq N_0 -1 \} \cup \{ e^{ix ^{(mod)} _n t},\, \vert n \vert = N_0 \}.
\end{split}
\end{equation*}
In order to keep the property to be a Riesz basis, we are allowed to modify a finite number of the elements of the family $ \{ e^{ix ^{(mod)} _n t}, n\in \mathbb Z \} $ (see \cite[Lemma II.4.11 p. 105]{Avdonin-Ivanov}), if we do not consider twice the same element. Therefore, we can transform the set of $2N_0 + 1$ elements
$$ \{ e^{ix ^{(mod)} _n t}, \vert n \vert \leq N_0 -1 \} \cup \{ e^{ix ^{(mod)} _n t}, \vert n \vert = N_0 \} $$
into 
$$ \{ e^{ix _n t}, \vert n \vert \leq N_0 -1 \} \cup \{ e^{ix ' _0 t}, e^{ix '' _0 t} \} $$
with $x'_0\neq x'' _0$ and $x'_0,x''_0\notin\{x_n, n\in \mathbb Z \}$. Thus,
$$ \{ e^{ix _{m} t}, m \in \mathbb Z \} \cup \{ e^{ix ' _0 t}, e^{ix '' _0 t} \} $$
is a Riesz basis of $L^2 (-\pi,\pi)$. Therefore Lemma \ref{lem-Riesz-T0-1-2} is proved when $k=2$ and when $k$ is even, with the same method. \qed

\underline{Case $k$ odd.} In the same way, we treat the case $k=1$, which can be easily extended for any $k$ odd. For $k=1$, the assumption reads as: that there exist $\delta\in \left(0,\frac{1}{4}\right)$ and $N_0\geq0$ such that
\begin{equation}
\label{hyp-Kadec-1}
n \geq N_0 \quad \implies \quad  \left\vert x_n - n  - \frac{1}{2} \right\vert \leq \frac{1}{4} - \delta .
\end{equation}
Consider the following sequence
\begin{equation}
\label{def-om-mod1}
\forall\, n\, \in \mathbb Z, \quad x ^{(mod)} _n = 
\begin{cases}
x _{n} - \frac{1}{2} & \text{ if } n \geq N_0 , \\
n & \text{ if } -N_0 \leq n  \leq N_0 -1 , \\
x _{n+1} - \frac{1}{2}  & \text{ if } n \leq - N_0 -1 .
\end{cases}
\end{equation}
We claim that
\begin{equation}
\label{def-om-mod1-prop}
\forall n\, \in \mathbb Z, \quad \left\vert x ^{(mod)} _n  - n \right\vert \leq \frac{1}{4} - \delta .
\end{equation}
Indeed, for $-N_0 \leq n  \leq N_0 -1$ \eqref{def-om-mod1-prop} is trivially true. Moreover, for $n\geq N_0$ we have that
$$ \left\vert x ^{(mod)} _n  - n \right\vert = \left\vert x _{n} - \frac{1}{2} - n \right\vert \leq \frac{1}{4} - \delta $$
thanks to \eqref{hyp-Kadec-1}. Finally, for  $n \leq - N_0 -1$, 
\begin{multline*}
\left\vert x ^{(mod)} _n  - n \right\vert 
= \left\vert x _{n+1} - \frac{1}{2}  - n \right\vert
= \left\vert - x _{\vert n+1 \vert } - \frac{1}{2} + \vert n \vert \right\vert 
\\
= \left\vert - x _{\vert n \vert -1 } - \frac{1}{2} + \vert n \vert \right\vert
= \left\vert x _{\vert n \vert - 1 } - \vert n \vert + \frac{1}{2} \right\vert 
= \left\vert x _{\vert n \vert - 1 } - (\vert n \vert - 1) - \frac{1}{2}  \right\vert
\leq \frac{1}{4} - \delta 
\end{multline*}
using that $\{x_n\}_{n\in\mathbb Z}$ is odd and, once again, thanks to \eqref{hyp-Kadec-1}.
Then, \eqref{def-om-mod1-prop} is satisfied. We deduce from the Kadec's $\frac{1}{4}$ Theorem (\cite{Kadec}, \cite[Theorem 1.14 p. 42]{Young}) that the set $\{ e^{ix ^{(mod)} _n t}, n\in \mathbb Z \}$ is a Riesz basis of $L^2 (-\pi,\pi)$. Now, we shift this basis. To this purpose, we observe that if $f \in L^2 (-\pi,\pi)$, then $g: t\mapsto g(t) = f(t) e^{-it/2}$ is still a function of $L^2 (-\pi,\pi)$. Hence, it can be decomposed as
$$ f(t) e^{-it/2} = \sum _{n\in \mathbb Z} c_n e^{ix ^{(mod)} _n t} \quad \text{ with } \quad 
A \sum _{n\in \mathbb Z} \vert c_n \vert ^2 \leq \Vert g \Vert ^2 _{L^2 (-\pi,\pi)} \leq B \sum _{n\in \mathbb Z} \vert c_n \vert ^2 ,$$
since $\{e^{ix_n^{(mod)}t},\,n\in\mathbb Z\}$ is a Riesz basis of $L^2(-\pi,\pi)$. Therefore, we have that
$$ f(t) = \sum _{n\in \mathbb Z} c_n e^{i(x ^{(mod)} _n + \frac{1}{2})t} \quad \text{ with } \quad 
A \sum _{n\in \mathbb Z} \vert c_n \vert ^2 \leq \Vert f \Vert ^2 _{L^2 (-\pi,\pi)} \leq B \sum _{n\in \mathbb Z} \vert c_n \vert ^2 .$$
Hence, the set $\{ e^{i(x ^{(mod)} _n + \frac{1}{2}) t}, n\in \mathbb Z \}$ is another Riesz basis of $L^2 (-\pi,\pi)$ and can be rewritten as
\begin{multline*}
 \{ e^{i(x ^{(mod)} _n + \frac{1}{2})) t},\, n\in \mathbb Z \} 
 \\ 
 = \{ e^{ix _{n+1} t},\, n \leq -N_0 - 1 \} \cup \{ e^{i(x ^{(mod)} _n + \frac{1}{2})) t},\, -N_0 \leq n  \leq N_0 -1 \} \cup \{ e^{ix _{n} t},\, n \geq N_0  \}
\\
= \{ e^{ix _{m} t},\, \vert m \vert  \geq N_0 \} \cup \{ e^{ix ^{(mod)} _n t},\, -N_0 \leq n  \leq N_0 -1 \} .
\end{multline*}
The last set on the right-hand side of the above formula contains $2N_0$ elements which can be modified without changing the Riesz basis property as follows
$$ \{ e^{ix _{m} t}, \vert m \vert  \leq N_0 -1 \} \cup \{ e^{ix _0 ' t} \},$$
where $x' _0 \notin  \{x_n, n\in \mathbb Z \}$.  Therefore, Lemma \ref{lem-Riesz-T0-1-2} is proved for $k=1$ and, similarly, for any $k$ odd. \qed 


\subsubsection{Proof of Lemma \ref{cor-Riesz}} \hfill

We know from \eqref{DAS-omega_n} that
$$ \frac{2}{\pi(2-\alpha)} \, \omega _{\alpha,n} - n \to \frac{\alpha}{4(2-\alpha)} =:\ell _\alpha \quad \text{ as } n \to +\infty. $$
Hence, we introduce
$$ \forall\, n\in \mathbb Z, \quad  x_n := \frac{2}{\pi(2-\alpha)} \, \omega _{\alpha,n} .$$
Since $\frac{1}{2-\alpha} \notin \mathbb N$, we can decompose it into the sum of its integer part $k_\alpha$ and its fractional part $\theta _\alpha$:
$$ \frac{1}{2-\alpha} = \left[\frac{1}{2-\alpha}\right] + \left\{ \frac{1}{2-\alpha} \right\} = k_\alpha + \theta _\alpha \quad \text{ where } \theta _\alpha \in (0,1) .$$
Then, we can rewrite $\ell_\alpha$ as
$$ \ell _\alpha = \frac{\alpha}{4(2-\alpha)} = \frac{1}{4} (k_\alpha +\theta_\alpha ) \Bigl( 2 - \frac{1}{k_\alpha+\theta_\alpha} \Bigr) 
= \frac{1}{4}( 2k_\alpha + 2 \theta_\alpha - 1 ) = \frac{k_\alpha}{2} + \frac{2\theta_\alpha -1}{4} 
.$$
Hence,
$$ x_n  - n - \frac{k_\alpha}{2} \to \ell _\alpha - \frac{k_\alpha}{2} = \frac{2\theta_\alpha -1}{4} \in \left(- \frac{1}{4}, \frac{1}{4}\right),$$
and \eqref{hyp-Kadec-k} is satisfied with $k=k_\alpha$. Therefore, the set $\{ e^{ix_n t }, n\in \mathbb Z \}$ can be complemented by $k_\alpha$ exponentials to form a Riesz basis of $L^2 (-\pi,\pi)$.Consequently, the set $\{ e^{i\omega_{\alpha,n} t }, n\in \mathbb Z \}$ can be complemented by $k_\alpha$ exponentials to form a Riesz basis of $L^2 (0,T_0)$ (as we have seen in the proof of Lemma \ref{lem-Riesz-T0}). This concludes the proof of Lemma \ref{cor-Riesz}. \qed


\subsubsection{Proof of Lemma \ref{lem-Riesz-T0-moment-1-2}} \hfill

As a consequence of Lemma \ref{cor-Riesz}, the set of solutions of the moment problem \eqref{syst-momentsQc-ps-ss} is an affine space generated by a vectorial space of dimension $k_\alpha$. To solve the whole moment problem \eqref{syst-momentsQc-ps}, it is sufficient to note that 
$$ \tilde s_{\alpha,0} \notin \overline{\text{Vect } \{e^{i\omega _{\alpha,n}t} , n\in \mathbb Z \}}:$$
indeed, if this was not the case, then $\overline{\text{Vect } \{e^{i\omega _{\alpha,n}t} , n\in \mathbb Z \}}$ would contain
all the polynomials (integrating, as in step 2 of section \ref{sec-surj-cas1}). Hence, it would contain $L^2(0,T_0)$. However, this is contradiction with the fact that $\overline{\text{Vect } \{e^{i\omega _{\alpha,n}t} , n\in \mathbb Z \}}$ is of codimension 
$k_\alpha$ in $L^2(0,T_0)$. Therefore, proceeding as in step 2 of section \ref{sec-surj-cas1}, we deduce that there exists a unique solution in $\overline{\text{Vect } \{\tilde s _{\alpha,0}, e^{i\omega _{\alpha,n}t} , n\in \mathbb Z \} }$ of moment problem \eqref{syst-momentsQc-ps}. Note that this solution is real-valued: indeed, we derive from \eqref{syst-momentsQc-ps} that
$$ \int _0 ^T Q(t) \, \cos \omega _{\alpha,n} t \, dt , \int _0 ^T Q(t) \, \sin \omega _{\alpha,n} t \, dt, \int _0 ^T Q(t) \, t \, dt \in \mathbb R. $$
Thus, denoting by $Q_2$ the imaginary part of $Q$, we have that
$$\int _0 ^T Q_2(t) \, \cos \omega _{\alpha,n} t \, dt = \int _0 ^T Q_2(t) \, \sin \omega _{\alpha,n} t \, dt, = \int _0 ^T Q_2(t) \, t \, dt = 0 .$$
Hence, the following conditions
$$ \begin{cases} \int _0 ^T Q_2(t) \, e^{-i \omega _{\alpha,n} t} \, dt = 0 \quad \text{for all } n \in \mathbb Z, 
 \\
 \int _0 ^T Q_2(t) \, t \, dt = 0 ,
 \end{cases} $$
imply that $Q_2$ is orthogonal to $\overline{\text{Vect } \{\tilde s _{\alpha,0}, e^{i\omega _{\alpha,n}t} , n\in \mathbb Z \} }$. However, $2i Q_2 = Q - \overline Q \in \overline{\text{Vect } \{\tilde s _{\alpha,0}, e^{i\omega _{\alpha,n}t} , n\in \mathbb Z \} }$. Therefore, $Q_2=0$ and $Q$ is real-valued. \qed


\subsubsection{Proof of Theorem \ref{thm-ctr=} second part part (inverse mapping argument)} \hfill
\label{sec-concl-T=T0-1-2}

Thanks to the previous results, we can conclude the proof Theorem \ref{thm-ctr=} with the same procedure explained in section \ref{sec-concl-T>T0}. \qed


\section{Proof of Theorem \ref{thm-ctr<}: Reachability for $T<T_0$}
\label{sec-pr-thm<}

\subsection{Proof of Theorem \ref{thm-ctr<} when $\alpha \in [0,1)$} \hfill
\label{sec-ctr<-01}

We recall that in Lemma \ref{lem-Riesz-T0} we have proved that $(e^{i\omega _{\alpha,n}t})_{n\in \mathbb Z}$ is a Riesz basis of $L^2 (0,T_0)$. 

As in sections \ref{sec-pr-thm} and \ref{sec-pr-thm-T=T0-01}, the proof of Theorem \ref{thm-ctr<} is divided in two steps:
\begin{itemize}
\item we first study the solvability of moment problem \eqref{syst-momentsQc},
\item then, we conclude by an inverse mapping argument.
\end{itemize}


\subsubsection{The moment problem \eqref{syst-momentsQc-ps} is overdetermined when $T<T_0$} \hfill

Let $T <T_0$. Following \cite[p. 100]{Avdonin-Ivanov}, we introduce
$$ \forall\, r >0, \quad n(r) := \text{ card } \{n \in \mathbb Z, \vert \omega _{\alpha,n} \vert < r \} .$$
We recall from Propositions \ref{prop-vp-w} and \ref{prop-vp-s} that the sequence
$(\omega _{\alpha,n+1} - \omega _{\alpha,n})_n$ is nonincreasing and goes to  $\kappa _\alpha \pi$ as $n\to\infty$. Hence, we deduce that
$$ \forall\, n\geq 0, \quad \omega _{\alpha,n} \geq \kappa _\alpha \pi (n-1) .$$
Therefore,
$$ n-1 \geq \frac{r}{\kappa _\alpha \pi} \quad \implies \quad \omega _{\alpha,n} \geq r .$$
This gives that
\begin{equation}
\label{n<}
n(r) \leq 2 \frac{r}{\kappa _\alpha \pi} +1  ,
\end{equation}
where the factor $2$ comes from the negatives indices.

On the other hand, given $\varepsilon >0$, there exists $n_0 \geq 0$ such that 
$$ \forall\, n \geq n_0, \quad \kappa _\alpha \pi \leq \omega _{\alpha,n+1} - \omega _{\alpha,n} \leq \kappa _\alpha \pi + \varepsilon .$$
Hence,
$$ \forall n \geq n_0, \quad \omega _{\alpha,n_0} +  \kappa _\alpha \pi (n-n_0) 
\leq  \omega _{\alpha,n}  \leq \omega _{\alpha,n_0} +  (\kappa _\alpha \pi + \varepsilon) (n-n_0) .$$
Thus, given $r >0$ 
$$ 0 \leq n < n_0 + \frac{r - \omega _{\alpha,n_0}}{\kappa _\alpha \pi + \varepsilon} 
\quad \implies \quad  \omega _{\alpha,n} < r .$$
We derive that for all $r$ large enough
\begin{equation}
\label{n>}
 n(r) \geq 2 \Bigl( n_0 + \frac{r - \omega _{\alpha,n_0}}{\kappa _\alpha \pi + \varepsilon} \Bigr) -1 ,
 \end{equation}
Then, we obtain from \eqref{n<} and \eqref{n>} that
\begin{equation}
\label{n=}
\lim _{r\to +\infty} \frac{n(r)}{r} = \frac{2}{\kappa _\alpha \pi }= \frac{4}{(2-\alpha)\pi} = \frac{T_0}{\pi} .
 \end{equation}
Since $T < T_0$, we have
$$ \limsup _{r\to +\infty} \frac{n(r)}{r} > \frac{T}{\pi} ,$$
and so we can to apply \cite[Corollary II.4.2 p. 100]{Avdonin-Ivanov} and we obtain that 
the family $\{ e^{i\omega _{\alpha,n}t} , n\in \mathbb Z\}$ is not minimal in $L^2 ((0,T), \mathbb C)$. 

\subsubsection{How much overdetermined the moment problem \eqref{syst-momentsQc-ps} is when $T<T_0$} \hfill
\label{sec-overd}

As proved in Lemma \ref{lem-Riesz-T0}, $\{ e^{i\omega _{\alpha,n}t} , n\in \mathbb Z\}$ is a Riesz basis of $L^2 ((0,T_0), \mathbb C)$. From Horv\'ath-Jo\'o \cite{Horvath-Joo} (see also \cite[Theorem II.4.16 p. 107]{Avdonin-Ivanov}) we deduce that there exists a subfamily $\{ e^{i\omega _{\alpha,\varphi (n)}t} , n\in \mathbb Z\}$ which is a Riesz basis of $L^2 ((0,T), \mathbb C)$. Then, consider 
$$ \forall\, r >0, \quad n_\varphi (r) := \text{ card } \{n \in \mathbb Z, \vert \omega _{\alpha,\varphi (n)} \vert < r \} .$$
Since $\{ e^{i\omega _{\alpha,\varphi (n)}t} , n\in \mathbb Z\}$ is minimal in $L^2 ((0,T), \mathbb C)$, we derive from \cite[Corollary II.4.2 p. 100]{Avdonin-Ivanov}, that 
$$ \limsup _{r\to +\infty} \frac{n_\varphi (r)}{r} \leq \frac{T}{\pi} ,$$
hence,
\begin{equation}
\label{n-diff}
\liminf _{r\to +\infty} \frac{n(r) - n_\varphi (r)}{r} \geq \frac{T_0-T}{\pi} .
 \end{equation}
Since 
$$ n(r) - n_\varphi (r) = \text{ card } \{n \in \mathbb Z, \vert \omega _{\alpha,n} \vert < r \text{ and } n \notin \text{ Im }\varphi\} ,$$
the asymptotic behaviour \eqref{n-diff} gives an idea of how much overdetermined the moment problem \eqref{syst-momentsQc-ps} is.

\subsubsection{Solvability of the moment problem} \hfill

We consider \eqref{syst-momentsQc}, or equivalently \eqref{syst-momentsQc-ps}. First, assume that \eqref{syst-momentsQc-ps}
has a solution $Q \in L^2 (0,T; \mathbb R)$. This implies that
\begin{equation}
\label{syst-momentsQc-ps-phi}
\langle  Q, e^{i \omega _{\alpha,\varphi (n)}t} \rangle _{L^2 (0,T;\mathbb C)}   = C_{\alpha, \varphi (n)} ^f \quad \text{ for all } n \in \mathbb Z .
\end{equation}
Now, consider $m\notin \text{Im } \varphi$. Then, $e^{i \omega _{\alpha,m}t}$ can be decomposed as follows
$$ e^{i \omega _{\alpha,m}t} = \sum _{n\in \mathbb Z} \Omega ^{(m)} _{\alpha, \varphi (n)} e^{i \omega _{\alpha,\varphi (n)}t} \quad 
\text{ in } L^2 (0,T; \mathbb C) . $$
Therefore, we have that
\begin{multline*}
C_{\alpha, m} ^f = \langle  Q, e^{i \omega _{\alpha,m}t} \rangle _{L^2 (0,T;\mathbb C)}
= \langle  Q, \sum _{n\in \mathbb Z} \Omega ^{(m)} _{\alpha, \varphi (n)} e^{i \omega _{\alpha,\varphi (n)}t} \rangle _{L^2 (0,T;\mathbb C)}
\\
= \sum _{n\in \mathbb Z} \overline{\Omega ^{(m)} _{\alpha, \varphi (n)}} \langle  Q, e^{i \omega _{\alpha,\varphi (n)}t} \rangle _{L^2 (0,T;\mathbb C)}
= \sum _{n\in \mathbb Z} \overline{\Omega ^{(m)} _{\alpha, \varphi (n)}} C_{\alpha, \varphi (n)} ^f .
\end{multline*}
In the same way, $\tilde s _{\alpha,0}$ can be decomposed as follows
$$ \tilde s _{\alpha,0} = \sum _{n\in \mathbb Z} \tilde S  _{\alpha, \varphi (n)} e^{i \omega _{\alpha,\varphi (n)}t} \quad 
\text{ in } L^2 (0,T; \mathbb C) , $$
which implies that
$$ B _{\alpha,0} ^f = \sum _{n\in \mathbb Z} \overline{\tilde S _{\alpha, \varphi (n)}} C_{\alpha, \varphi (n)} ^f .$$
Hence,
$$ Q \text{ solution of \eqref{syst-momentsQc-ps}} \quad \implies \quad 
\begin{cases} 
C_{\alpha, m} ^f = \displaystyle{\sum _{n\in \mathbb Z} \overline{\Omega ^{(m)} _{\alpha, \varphi (n)}}} C_{\alpha, \varphi (n)} ^f  \quad \forall\, m \notin \text{ Im } \varphi , \\ 
B _{\alpha,0} ^f = \displaystyle{\sum _{n\in \mathbb Z} \overline{\tilde S _{\alpha, \varphi (n)}}} C_{\alpha, \varphi (n)} ^f .
\end{cases} $$
This leads to consider the space
\begin{equation}
\label{def-Halpha}
H_\alpha ^f := \left\{ (Y^f,Z^f) \in H^3 _{(0)} \times D(A),\, \begin{cases} 
C_{\alpha, m} ^f = \displaystyle{\sum _{n\in \mathbb Z} \overline{\Omega ^{(m)} _{\alpha, \varphi (n)}}} C_{\alpha, \varphi (n)} ^f  \quad \forall\, m \notin \text{ Im } \varphi , \\ 
B _{\alpha,0} ^f = \displaystyle{\sum _{n\in \mathbb Z} \overline{\tilde S _{\alpha, \varphi (n)}}} C_{\alpha, \varphi (n)} ^f \end{cases}\right\} ,
\end{equation}
where the relations between $(Y^f,Z^f)$ and $B^f_{\alpha,0}$, $C^f_{\alpha,m}$, $m\geq1$ are given in \eqref{eq-coeffs} and \eqref{syst-coeffs}. Thus, we have proved that 
\begin{equation}
\label{image-dtheta-subset}
D \Theta _T (0) (L^2 (0,T;\mathbb R) \subset H_\alpha ^f .
\end{equation}
Now, let us prove the following reverse inclusion.

\begin{Lemma}
\label{lem-descr-imageDtheta}
Let $\alpha \in [0,1)$, $T<T_0$ and $H_\alpha ^f$ be defined in \eqref{def-Halpha}. Then, the following identity holds
\begin{equation}
\label{image-dtheta-supset}
D \Theta _T (0) (L^2 (0,T;\mathbb R) = H_\alpha ^f .
\end{equation}
\end{Lemma}

\begin{proof}[Proof of Lemma \ref{lem-descr-imageDtheta}] 
Since we already proved \eqref{image-dtheta-subset}, it is sufficient to prove that $H_\alpha ^f \subset D \Theta _T (0) (L^2 (0,T;\mathbb R) $.
Let $ (Y^f,Z^f) \in H_\alpha ^f$. Since $(e^{i\omega _{\alpha,\varphi (n)}t})_{n\in \mathbb Z}$ is a Riesz basis of $L^2 (0,T)$, the moment problem \eqref{syst-momentsQc-ps-phi} has one and only one solution $Q$ (which can be expressed using the unique biorthogonal family to $(e^{i\omega _{\alpha,\varphi (n)}t})_{n\in \mathbb Z}$).
Then, for all $m\notin \text{ Im } \varphi$, we have
\begin{multline*}
\langle  Q, e^{i \omega _{\alpha,m}t} \rangle _{L^2 (0,T;\mathbb C)}
= \langle  Q, \sum _{n\in \mathbb Z} \Omega ^{(m)} _{\alpha, \varphi (n)} e^{i \omega _{\alpha,\varphi (n)}t} \rangle _{L^2 (0,T;\mathbb C)}
\\
= \sum _{n\in \mathbb Z} \overline{\Omega ^{(m)} _{\alpha, \varphi (n)}} \langle  Q, e^{i \omega _{\alpha,\varphi (n)}t} \rangle _{L^2 (0,T;\mathbb C)}
= \sum _{n\in \mathbb Z} \overline{\Omega ^{(m)} _{\alpha, \varphi (n)}} C_{\alpha, \varphi (n)} ^f = C_{\alpha, m} ^f ,
\end{multline*}
where the last equality derives from the fact that $(Y^f,Z^f) \in H_\alpha ^f$.
In the same way, we get
$$ \langle  Q, \tilde s _{\alpha,0} \rangle _{L^2 (0,T;\mathbb C)} = 
\sum _{n\in \mathbb Z} \overline{\tilde S _{\alpha, \varphi (n)}} \langle  Q, e^{i \omega _{\alpha,\varphi (n)}t} \rangle _{L^2 (0,T;\mathbb C)}
= \sum _{n\in \mathbb Z} \overline{\tilde S _{\alpha, \varphi (n)}} C_{\alpha, \varphi (n)} ^f = B_{\alpha, 0} ^f .$$
Hence, $Q$ solves the whole moment problem \eqref{syst-momentsQc-ps}. It remains to prove that $Q$ is real-valued: this follows easily from the fact that
$$ \forall\, n \geq 0, \quad 
\begin{cases}
\int _0 ^T Q(t) \, e^{-i\omega _{\alpha,n} t} \, dt = C _{\alpha,n} ^f = \frac{Z_{\alpha,n} ^f - i \sqrt{\lambda _{\alpha,n}} \, Y_{\alpha,n} ^f}{\mu _{\alpha,n}} ,\\
\int _0 ^T Q(t) \, e^{i\omega _{\alpha,n} t} \, dt = \overline{C _{\alpha,n} ^f} = \frac{Z_{\alpha,n} ^f + i \sqrt{\lambda _{\alpha,n}} \, Y_{\alpha,n} ^f}{\mu _{\alpha,n}}.
\end{cases} $$
By adding (subtracting) one to each other the above equations, we obtain
$$ \forall\, n \geq 0, \quad 
\begin{cases}
\int _0 ^T Q(t) \, \cos \omega _{\alpha,n} t \, dt = 2\frac{Z_{\alpha,n} ^f }{\mu _{\alpha,n}} ,\\
\int _0 ^T Q(t) \, 2i \sin \omega _{\alpha,n} t \, dt = \frac{2i \sqrt{\lambda _{\alpha,n}} \, Y_{\alpha,n} ^f}{\mu _{\alpha,n}}.
\end{cases} $$
Thus, the real part of $Q$ solves \eqref{syst-momentsQ}, and its imaginary part $Q_2$ satisfies
$$ \forall\, n \geq 0, \quad \int _0 ^T Q_2 (t) \, \cos \omega _{\alpha,n} t \, dt = 0 = \int _0 ^T Q_2 (t) \, \sin \omega _{\alpha,n} t \, dt .$$
Therefore,
$$ \forall\, n \in \mathbb Z, \quad \langle  Q_2, e^{i \omega _{\alpha, \varphi (n)}t} \rangle _{L^2 (0,T;\mathbb C)} =0,$$
which implies $Q_2 =0$ since $(e^{i \omega _{\alpha, \varphi (n)}t})_{n\in \mathbb Z}$ is a Riesz basis of $L^2(0,T)$. So, we have proved that $Q$ is real-valued and this completes the proof of \eqref{image-dtheta-supset}. \end{proof}

We conclude by proving the following 
\begin{Lemma}
\label{lem-Halpha-codim}
$H_\alpha ^f$ is a closed vectorial space of $H^3 _{(0)} \times D(A)$ of infinite dimension and infinite codimension.
\end{Lemma}

\noindent {\it Proof of Lemma \ref{lem-Halpha-codim}.} Let us consider
$$ \forall\, m \notin \text{Im } \varphi, \quad L_T ^m: H^3 _{(0)} \times D(A) \to \mathbb C , \quad L _T ^m (Y^f,Z^f) := C_{\alpha, m} ^f - \sum _{n\in \mathbb Z} \overline{\Omega ^{(m)} _{\alpha, \varphi (n)}} C_{\alpha, \varphi (n)} ^f ,$$
and
$$ \ell_T ^0: H^3 _{(0)} \times D(A) \to \mathbb C , \quad \ell _T ^0 (Y^f,Z^f) := B _{\alpha,0} ^f - \sum _{n\in \mathbb Z} \overline{\tilde S _{\alpha, \varphi (n)}} C_{\alpha, \varphi (n)} ^f  .$$
Observe that $\ell_T ^0$ and $L_T ^m$ are linear continuous forms, and
$$  H_\alpha ^f = \text{Ker } \ell_T ^0 \cap \Bigl( \cap _{m \notin \text{Im } \varphi} \text{Ker } L_T ^m \Bigr) .$$
Hence, $H_\alpha ^f$ is a closed vectorial space of $H^3 _{(0)} \times D(A)$, and is of infinite dimension.
To prove that $H_\alpha ^f$ has infinite codimension, we use the fact that $\mathbb Z \setminus ( \text{Im } \varphi )$
is infinite (see section \ref{sec-overd}). Then, fix $N\geq 1$, and $n_1, \cdots, n_N \notin ( \text{Im } \varphi \cup \{0\})$, and consider
$$ L_T ^{n_1, \cdots, n_N}: H^3 _{(0)} \times D(A) \to \mathbb C^N , \quad L _T ^{n_1, \cdots, n_N} (Y^f,Z^f) := 
\left( \begin{array}{c} L _T ^{n_1} (Y^f,Z^f) \\ \cdots \\ L _T ^{n_N} (Y^f,Z^f) \end{array} \right) .$$
Observe that
$$ 
L _T ^{n_1, \cdots, n_N} (\Phi _{\alpha, n_1},0) = 
-i \frac{\omega _{\alpha,n_1}}{\mu _{\alpha, n_1}} \varepsilon _1 , \quad \cdots \quad , 
L _T ^m (\Phi _{\alpha, n_N},0) = 
-i \frac{\omega _{\alpha,n_N}}{\mu _{\alpha, n_N}} \varepsilon _N ,$$
where $\varepsilon _1, \cdots, \varepsilon _N$ is the canonical basis of $\mathbb C^N$. Hence, $L_T ^{n_1, \cdots, n_N}$ is surjective, and its kernel is of codimension $N$. Since this is true for all $N\geq 1$, this implies that $H_\alpha ^f$ is a closed vectorial space of $H^3 _{(0)} \times D(A)$ of infinite codimension. \qed

\subsubsection{Proof of Theorem \ref{thm-ctr<} for $\alpha \in [0,1)$ (inverse mapping argument)} \hfill
\label{sec-concl-T<T0}

The proof follows the scheme of section \ref{sec-concl-T=T0}, replacing the hyperplane $P^f _\alpha$ by $H^f _\alpha$. Therefore,
$\Theta _T (\mathcal V (0))$ turns out to be a submanifold of infinite dimension and infinite codimension. \qed

\subsection{Proof of Theorem \ref{thm-ctr<}: modifications when $\alpha \in [1,2)$} \hfill

When $\alpha \in [1,2)$ and $\frac{1}{2-\alpha} \notin \mathbb N$, we replace Lemma \ref{lem-Riesz-T0} by Lemma \ref{cor-Riesz} to obtain that $\{e^{i\omega _{\alpha,n}t}\}_{n\in \mathbb Z}$, complemented by a finite number of exponentials $\{e^{ix'_0t},\dots,e^{ix'_{k_\alpha}t}\}$, is a Riesz basis of $L^2 (0,T_0)$. Then, one can extract a subfamily of $\{e^{i\omega_{\alpha,n}t}\}_{n\in\mathbb Z}\cup\{e^{ix'_0t},\dots,e^{ix'_{k_\alpha}t}\}$ that will be a Riesz basis of $L^2 (0,T)$. Replacing the possible elements coming from the finite set of exponentials $\{e^{ix'_0t},\dots,e^{ix'_{k_\alpha}t}\}$ by the same number of exponentials of the original family $\{e^{i\omega_{\alpha,n}t}\}_{n\in\mathbb Z}$, we have a subfamiliy of $\{e^{i\omega _{\alpha,n}t}\}_{n\in \mathbb Z}$ that is a Riesz basis of $L^2 (0,T)$, and then we can complete the proof as in section \ref{sec-ctr<-01}.

In a more general way, without assuming $\frac{1}{2-\alpha} \notin \mathbb N$, we consider $x\in \mathbb R$, $r>0$ and define
$$ N(x,r):= \text{card } \{\omega _{\alpha,n}, x \leq \omega _{\alpha,n} < x+r \}.$$
Thanks to \eqref{DAS-omega_n}, one can prove that
\begin{equation}
\label{lim-Nxr}
\frac{N(x,r)}{r} \to \frac{T_0}{2\pi} \quad \text{ as } r \to +\infty 
\end{equation}
uniformly with respect to $x\in \mathbb R$. Indeed, assume first that we are in the most simple case: 
\begin{equation}
\label{lim-Nxr-app}
\forall\, n \geq 1, \quad \omega _{\alpha,n} = \frac{\pi (2-\alpha)}{2} \Bigl( n+ \frac{\alpha}{4(2-\alpha)} \Bigr)
= \frac{\pi (2-\alpha)}{2} n + \frac{\pi \alpha}{8} ,
\end{equation}
and so the gap $\omega _{\alpha,n+1} - \omega _{\alpha,n}$ is constant with respect to $n\geq 1$:
$$  \forall\, n \geq 1, \quad \omega _{\alpha,n+1} - \omega _{\alpha,n} = \frac{\pi (2-\alpha)}{2} .$$
Then it is clear that any "window" $[x,x+r)$ contains $N(x,r)$ terms, where $N(x,r)$ behaves as $\frac{2}{\pi (2-\alpha)} r$. More precisely, there exists some computable $N_0$ such that 
$$ \forall\, x \in \Bbb R, \forall\, r >0, \quad 
\frac{2}{\pi (2-\alpha)} r - N_0 \leq N(x,r) \leq \frac{2}{\pi (2-\alpha)} r + N_0 . $$
Then, in this case we have
$$ \forall\, x \in \Bbb R, \forall\, r >0, \quad 
\frac{2}{\pi (2-\alpha)} - \frac{N_0}{r}  \leq \frac{N(x,r)}{r} \leq \frac{2}{\pi (2-\alpha)}  + \frac{N_0}{r} ,$$
and therefore 
$$ \frac{N(x,r)}{r} \to \frac{2}{\pi (2-\alpha)} = \frac{T_0}{2\pi} \quad \text{ as } r \to +\infty \quad \text{uniformly with respect to $x\in \mathbb R$,}. $$
Hence, for this particular case, \eqref{lim-Nxr} is satisfied.

The general case is not as simple as the one assumed in \eqref{lim-Nxr-app}. However, \eqref{DAS-omega_n} implies that, given $\varepsilon _1 >0$ small, there exists $N_1\geq 0$ such that
$$ \forall\, n \geq N_1, \quad \omega _{\alpha,n} \in 
 ( \frac{\pi (2-\alpha)}{2} n + \frac{\pi \alpha}{8} - \varepsilon _1, \frac{\pi (2-\alpha)}{2} n + \frac{\pi \alpha}{8} + \varepsilon _1) . $$
So, if $x \geq x_1 := \frac{\pi (2-\alpha)}{2} N_1 + \frac{\pi \alpha}{8} - \varepsilon _1$, the situation is very close to the simple one studied before, and there exists $N_2$ such that
$$ \forall\, x \geq x_1, \forall\, r >0, \quad 
\frac{2}{\pi (2-\alpha)} r - N_2 \leq N(x,r) \leq \frac{2}{\pi (2-\alpha)} r + N_2 . $$
By symmetry, the situation is similar if $x+r \leq -x_1$. And finally, if the window $[x,x+r)$ intersects $[-x_1, x_1]$, then only the location of $\omega _{-N_1 +1}$, $\cdots$ , $\omega _{N_1-1}$ can modify the counting number $N(x,r)$, and thus there exists $N_3$ (depending only on $N_1$ and $N_2$) such that
$$ \forall\, x \in \Bbb R, \forall\, r >0, \quad 
\frac{2}{\pi (2-\alpha)} r - N_3 \leq N(x,r) \leq \frac{2}{\pi (2-\alpha)} r + N_3 . $$
Therefore the conclusion \eqref{lim-Nxr} follows also for the general case.

Now that \eqref{lim-Nxr} is proved, we deduce from \eqref{lim-Nxr} and from \cite[Theorem II.4.18 p. 109]{Avdonin-Ivanov} that, given $T<T_0$, the family $(e^{i\omega _{\alpha,n}t})_{n\in \mathbb Z}$ contains a subfamily that forms a Riesz basis of $L^2 (0,T)$. And then, once again, we can proceed as in 
section \ref{sec-ctr<-01} to prove Theorem \ref{thm-ctr<}. \qed


\section{Proof of Proposition \ref{prop-mu-ex-dense}} 
\label{sec-exemple-mu}

First we check that $$\mu (x) = x^{2-\alpha}$$ satisfies all the regularity assumptions:
\begin{itemize}
\item $\alpha \in [0,1)$: first, we observe that $\mu '(x)= (2-\alpha) x^{1-\alpha} \in L^1 (0,1)$. Hence, $\mu$ is absolutely continuous on $[0,1]$. Moreover, $x^{\alpha /2} \mu '(x)= (2-\alpha) x^{1-\frac{\alpha}{2}} \in L^2 (0,1)$. Thus, $\mu \in H^1 _\alpha (0,1)$. Furthermore, $x^{\alpha} \mu ' (x) = (2-\alpha) x \in H^1 (0,1)$ and so $\mu \in H^2 _\alpha (0,1)$. 
Finally, we observe that $x^{\alpha /2} \mu '(x)= (2-\alpha) x^{1-\frac{\alpha}{2}} \in L^\infty (0,1)$ that implies 
that $\mu \in V ^{(2,\infty)} _\alpha (0,1)$;
\item $\alpha \in [1,2)$: in this case it easy to check that $\mu \in L^2 (0,1)$. Moreover, $x^{\alpha /2} \mu '(x)= (2-\alpha) x^{1-\frac{\alpha}{2}} \in L^2 (0,1)$ and therefore $\mu \in H^1 _\alpha (0,1)$. Furthermore, $x^{\alpha} \mu ' (x) = (2-\alpha) x \in H^1 (0,1)$. Thus, we have that $\mu \in H^2 _\alpha (0,1)$. Finally, we note that $x^{\alpha /2} \mu '(x)= (2-\alpha) x^{1-\frac{\alpha}{2}} \in L^\infty (0,1)$, 
and $(x^\alpha \mu ')' = 2-\alpha \in L^\infty (0,1)$. Hence, $\mu \in V ^{(2,\infty, \infty)} _\alpha (0,1)$.
\end{itemize}
We have showed that the regularity assumptions are satisfied. It remains to check the validity of \eqref{hyp-mu}. Direct computations show that
$$ \langle \mu , \Phi _{\alpha,0} \rangle _{L^2 (0,1)} = \int _0 ^1 x^{2-\alpha} \, dx = \frac{1}{3-\alpha} ,$$
and, for all $n\geq 1$, we develop the scalar product as follows
\begin{equation*}
\begin{split}
\langle \mu , \Phi _{\alpha,n} \rangle _{L^2 (0,1)}
&= \int _0 ^1 \mu (x) \Phi _{\alpha,n} \, dx = \frac{1}{\lambda _{\alpha,n}} \int _0 ^1 \mu (x) \lambda _{\alpha,n} \Phi _{\alpha,n} \, dx
\\
&= \frac{1}{\lambda _{\alpha,n}} \int _0 ^1 \mu (x) (-x^\alpha \Phi _{\alpha,n} ')' \, dx
\\
&= \frac{1}{\lambda _{\alpha,n}}  \Bigl( [-x^\alpha \mu (x) \Phi _{\alpha,n} '(x) ] _0 ^1 + \int _0 ^1 x^\alpha \mu '(x) \Phi _{\alpha,n} ' (x) \Bigr).
\end{split}
\end{equation*}
Recalling that $\mu (x)=x^{2-\alpha}$, we obtain
\begin{equation*}
\begin{split}
\int _0 ^1 x^\alpha \mu '(x) \Phi _{\alpha,n} ' (x) 
&= (2-\alpha) \int _0 ^1 x \Phi _{\alpha,n} ' (x) 
\\
&= (2-\alpha) [x \Phi _{\alpha,n}  (x)] _0 ^1 - (2-\alpha) \int _0 ^1 \Phi _{\alpha,n} (x) \, dx
\\
&= (2-\alpha) [x \Phi _{\alpha,n}  (x)] _0 ^1 - (2-\alpha) \langle \Phi _{\alpha,0} , \Phi _{\alpha,n} \rangle _{L^2 (0,1)} .
\end{split}
\end{equation*}
Since the eigenfunctions are orthogonal, we have that
$$ \langle \Phi _{\alpha,0} , \Phi _{\alpha,n} \rangle _{L^2 (0,1)} = 0 ,$$
hence,
$$ \langle \mu , \Phi _{\alpha,n} \rangle _{L^2 (0,1)}
= \frac{1}{\lambda _{\alpha,n}}  \left( \left[-x^2 \Phi _{\alpha,n} '(x) \right] _0 ^1
+ (2-\alpha) \left[x \Phi _{\alpha,n}  (x)\right] _0 ^1 \right) .$$
From the Neumann boundary conditions satisfied by $\Phi _{\alpha,n}$, we know that
$x \Phi _{\alpha,n} '(x) \to 0$ as $x\to 0$ and as $x\to 1$, thus
$$  [-x^2 \Phi _{\alpha,n} '(x) ] _0 ^1 = 0 .$$
We have also proved (Lemmas \ref{lem-prop-phi-n-0-1-w} and \ref{lem-prop-phi-n-0-1-s}) that $\Phi _{\alpha,n}$ has a finite limit as $x\to 0$, therefore
$$ x \Phi _{\alpha,n}  (x) \to 0, \quad \text{ as } x \to 0 .$$
Finally, once again from Lemmas \ref{lem-prop-phi-n-0-1-w} and \ref{lem-prop-phi-n-0-1-s} we have that $\vert \Phi _{\alpha,n} (1)\vert =\sqrt{2-\alpha}$ that yields
$$ \vert (2-\alpha) [x \Phi _{\alpha,n}  (x)] _0 ^1 \vert = (2-\alpha)^{3/2} ,$$
and
$$ \vert \langle \mu , \Phi _{\alpha,n} \rangle _{L^2 (0,1)} \vert = \frac{(2-\alpha)^{3/2}}{\lambda _{\alpha,n}} .$$
Hence, \eqref{hyp-mu} is satisfied. \qed

Now, let us prove that the set of functions $\mu$ satisfying \eqref{hyp-mu} is dense in $V^2 _\alpha$. By integrating by parts, we get
\begin{multline*}
\langle \mu , \Phi _{\alpha,n} \rangle _{L^2 (0,1)}
= \frac{1}{\lambda _{\alpha,n}}  \Bigl( [-x^\alpha \mu (x) \Phi _{\alpha,n} '(x) ] _0 ^1 + \int _0 ^1 x^\alpha \mu '(x) \Phi _{\alpha,n} ' (x) \, dx \Bigr)
\\
= \frac{1}{\lambda _{\alpha,n}}  \Bigl( [-x^\alpha \mu (x) \Phi _{\alpha,n} '(x) ] _0 ^1 + [x^\alpha \mu '(x) \Phi _{\alpha,n}  (x)] _0 ^1 - \int _0 ^1 (x^\alpha \mu '(x))' \Phi _{\alpha,n}  (x) \, dx \Bigr) .
\end{multline*}
Then, since $\mu \in L^\infty (0,1)$, we have
$$ [-x^\alpha \mu (x) \Phi _{\alpha,n} '(x) ] _0 ^1 = 0 .$$
Moreover, since $x^{\alpha /2} \mu' \in L^\infty (0,1)$ and $\Phi _{\alpha,n}$ has a finite limit as $x\to 0$, we deduce
$$ x^\alpha \mu '(x) \Phi _{\alpha,n}  (x) \to 0 \quad \text{ as } x \to 0 .$$
Thus, we obtain that
$$ [x^\alpha \mu '(x) \Phi _{\alpha,n}  (x)] _0 ^1 = \mu ' (1) \Phi _{\alpha,n}  (1).$$
Finally, since $(x^\alpha \mu '(x))' \in L^2 (0,1)$, we get
$$ \int _0 ^1 (x^\alpha \mu '(x))' \Phi _{\alpha,n}  (x) \, dx = \langle  (x^\alpha \mu '(x))', \Phi _{\alpha,n} \rangle _{L^2 (0,1)} \to 0 ,\quad \text{ as } n \to \infty .$$
So, recalling that $|\Phi_{\alpha,n}(1)|=\sqrt{2-\alpha}$, we infer
$$ \mu \in V_\alpha ^2(0,1) \quad \implies \quad \vert \lambda _{\alpha,n} \langle \mu , \Phi _{\alpha,n} \rangle _{L^2 (0,1)} \vert 
\to  \sqrt{2-\alpha} \, \vert \mu '(1) \vert, \quad \text{ as } n \to \infty .$$
We define the spaces 
$$ \mathcal V _n := 
\begin{cases} 
\{ \mu \in V_\alpha ^2(0,1),\, \langle \mu , \Phi _{\alpha,n} \rangle _{L^2 (0,1)} \neq 0 \} \quad &\text{ for } n \geq 0 , \\
\{ \mu \in V_\alpha ^2(0,1),\, \mu '(1) \neq 0 \}  \quad &\text{ for } n =-1 ,
\end{cases} $$
and 
$$ \mathcal V _\alpha ^2 := \cap _{n=-1} ^\infty \mathcal V_n .$$ 
Every $\mathcal V_n$ is open and dense in $V _\alpha ^2$. Indeed, consider $\tilde \mu \in V^2 _\alpha(0,1)$ such that $\tilde \mu \notin  \mathcal V_n$ for some $n\geq -1$, and define
$$ \tilde \mu _\varepsilon (x) := \tilde \mu (x) + \varepsilon x^{2-\alpha} $$
where $\varepsilon \in \mathbb R^*$. Then, 
if $n\geq 0$, we have
$$ \langle \tilde \mu _\varepsilon, \Phi _{\alpha,n} \rangle _{L^2 (0,1)}
= \varepsilon \langle x^{2-\alpha}, \Phi _{\alpha,n} \rangle _{L^2 (0,1)} \neq 0 ,$$
and if $n=-1$, we have
$$ \tilde \mu _\varepsilon '(1) = \varepsilon (2-\alpha) \neq 0.$$
Therefore, $\tilde \mu _\varepsilon \in \mathcal V_n$ and it is close to $\tilde \mu$ in $V^2 _\alpha$ if $\varepsilon$ is sufficiently small. This proves that $\mathcal V_n$ is dense in $V _\alpha ^2$. Thus, $\mathcal V _\alpha ^2$ is the intersection of a sequence of open and dense subsets and, thanks to Baire Theorem, it is dense in $V _\alpha ^2$. \qed


\medskip

\noindent {Acknowledgements.} The authors would like to thank A. Duca and V. Komornik for interesting discussions.


\bibliographystyle{plain}
\bibliography{references}

\end{document}